\documentclass[11pt]{amsart}
\usepackage{palatino}
\usepackage{graphicx}
\usepackage{geometry}
\usepackage{hyperref}
\usepackage{soul,color}
\usepackage{enumerate}
\usepackage{pstricks}
\usepackage{paralist, amsthm,amsfonts,amsmath,amssymb,graphics}
\newcommand{\footnotetextplain}[1]{\begingroup\def\@thefnmark{}%
  \long\def\@makefntext##1{\parindent 0pt\noindent ##1}\@footnotetext{#1}
  \endgroup}

\makeatletter 
\newcommand{\TeoremaAmbFinalMarcat}[1]{%
  \expandafter\gdef\csname end#1\endcsname{\@endtheorem}}
  
               {\hfill\rule{2.5mm}{2.5mm} \vspace{\parskip} } 
\newtheorem{theorem}{Theorem}[section]
\newtheorem{proposition}[theorem]{Proposition}
\newtheorem{corollary}[theorem]{Corollary}
\newtheorem{lemma}[theorem]{Lemma}
\newtheorem{hypothesis}[theorem]{Hypothesis}

\theoremstyle{definition}
\newtheorem{definition}[theorem]{Definition} \TeoremaAmbFinalMarcat{defi}
 \TeoremaAmbFinalMarcat{ex}
\newtheorem{remark}[theorem]{Remark} \TeoremaAmbFinalMarcat{rem}
 \TeoremaAmbFinalMarcat{conv}

\def\@enum@{\list{\csname label\@enumctr\endcsname}%
           {\usecounter{\@enumctr}\def\makelabel##1{\hss\llap{##1}}
           \itemsep=2pt\parsep=0pt\topsep=3pt plus 1pt minus 1 pt}}

\newcommand{\xnorm}[1]{ \Vert #1 \Vert }

\newcommand{\zz}[1]{\mathbb #1}

\newcommand{\var}{\mathrm{Var}}

\title[Statistics of Self-intersection Counts]{Statistical regularities of self-intersection counts for 
geodesics on negatively curved surfaces}
\author{Steven P. Lalley} \address{University of Chicago\\ Department
of Statistics \\ 5734
University Avenue \\
Chicago IL 60637.}
\email{lalley@galton.uchicago.edu}

\date{\today}
\thanks{Supported by NSF grant DMS-1106669}
\makeindex

\begin{document}

\begin{abstract}
Let $\Upsilon $ be a compact, negatively curved surface.  From the
(finite) set of all closed geodesics on $\Upsilon$ of length $\leq L$,
choose one, say $\gamma_{L}$, at random and let $N (\gamma_{L})$ be
the number of its self-intersections. It is known that there is a
positive constant $\kappa$ depending on the metric such that $N
(\gamma_{L})/L^{2} \rightarrow \kappa$ in probability as $L\rightarrow
\infty$. The main results of this paper concern the size of typical
fluctuations of $N (\gamma_{L})$ about $\kappa L^{2}$. It is proved
that if the metric has constant curvature $-1$ then typical
fluctuations are of order $L$, in particular, $(N (\gamma_{L})-\kappa
L^{2})/L$ converges weakly to a nondegenerate probability
distribution. In contrast, it is also proved that if the metric has variable
negative curvature then fluctuations of $N (\gamma_{L})$ are of order
$L^{3/2}$, in particular, $(N (\gamma_{L})-\kappa L^{2})/L^{3/2}$
converges weakly to a Gaussian distribution. Similar results are
proved for generic geodesics, that is, geodesics whose initial tangent
vectors are chosen randomly according to normalized Liouville measure.
\end{abstract}

\maketitle
\newpage
\tableofcontents

\section{Introduction}\label{sec:intro}

\subsection{Self-intersections of random geodesics}\label{ssec:randomGD}

Choose a point $x$ and a direction $\theta$ at random on a compact,
negatively curved surface $\Upsilon$ --- that is, so that the distribution of
the random unit vector $(x,\theta)$ is the  normalized Liouville
measure on the unit tangent bundle $S\Upsilon$ --- and let
${\gamma}(t)={\gamma} (t;x,\theta)$ be the unit speed
geodesic ray in direction $\theta$ 
started at $x$, viewed as a curve in $S\Upsilon$. Let $p:S\Upsilon
\rightarrow \Upsilon$ be the natural projection, and denote by $N
(t)=N (\gamma [0,t])$ the number of
transversal\footnote{If the initial point $x$ and direction $\theta$
are chosen randomly (according to the normalized Liouville measure on
the unit tangent bundle) then there is probability $0$ that the
resulting geodesic will be periodic, so with probability $1$ every
self-intersection will necessarily be transversal.}
self-intersections of the geodesic segment $p\circ \gamma [0,t]$.
  For large $t$ the number $N (t)$  will be of
order $t^{2}$; in fact, 
\begin{equation}\label{eq:lln}
	\lim_{t \rightarrow \infty} N (t)/t^{2} = 1/ (4\pi
	|\Upsilon|):=\kappa_{\Upsilon} 
\end{equation}
with probability $1$. See section~\ref{ssec:lln} below for the (easy)
proof. A similar result holds for a randomly chosen \emph{closed}
geodesic \cite{lalley:si1}: if from among all closed geodesics of
length $\leq L$ one is chosen at random, then the number of
self-intersections, normalized by $L^{2}$, will, with probability
approaching one as $L \rightarrow \infty$, be close to
$\kappa_{\Upsilon}$.  (See \cite{pollicott-sharp} for a related
theorem). Closed geodesics with \emph{no} self-intersections have long
been of interest in geometry --- see, for instance,
\cite{birman-series:1,birman-series:2,mirzakhani} --- and it is known
\cite{mirzakhani,rivin} that the number of simple closed geodesics of
length $\leq t$ grows at a polynomial rate in $t$. The fact that there
are arbitrarily long simple closed geodesics implies that the maximal
variation in $N (t)$ is of order $t^{2}$. The problems we address in
this paper concern the order of magnitude of \emph{typical} variations
of the self-intersection count $N (t)$ about $\kappa_{\Upsilon}t^{2}$
for both random and random closed geodesics. For random geodesics the
main result is the following theorem.

\begin{theorem}\label{theorem:random}
Let $\Upsilon$ be a compact surface equipped with a Riemannian metric
of negative curvature.
Assume that $u= (x,\theta)$ is a random unit tangent vector
distributed according to normalized Liouville measure on $S\Upsilon$, and
let $N (T)$ be the number of transversal self-intersections of 
the geodesic segment $\gamma ([0,T];x,\theta)$ with initial tangent
vector $u$. Then as $T \rightarrow \infty$,
\begin{equation}\label{eq:random}
	\frac{N (T)-\kappa_{\Upsilon}T^{2}}{T}
	\stackrel{\mathcal{D}}{\longrightarrow} \Psi 
\end{equation}
for some probability  distribution $\Psi$ on $\zz{R}$ (which will in
general depend on the surface and the Riemannian metric).
\end{theorem}

Here $\stackrel{\mathcal{D}}{\rightarrow}$ indicates \emph{convergence
in distribution} (i.e., weak convergence, cf. \cite{billingsley:wc}):
a family of real-valued random variables $Y_{t}$ is said to converge
in distribution to a Borel probability measure $G$ on $\zz{R}$ if for
every bounded, continuous function $f$
\[
	\lim_{t \rightarrow \infty}Ef (Y_{t})=\int f \,dG.
\]

Since to first order $N (T)$ is approximately
$\kappa_{\Upsilon}T^{2}$, and since $T=\sqrt{T^{2}}$,
Theorem~\ref{theorem:random} might at first sight appear to be typical
``central limit'' behavior (see \cite{ratner:clt} for the classical
central limit theorem for geodesic flows). But it isn't.  The limit
distribution $\Psi$ in \eqref{eq:random} is a limit of Gaussian
quadratic forms, and therefore is most likely not Gaussian.\footnote{A
\emph{Gaussian quadratic form} is a random variable of the form
$\sum_{i=1}^{m} \sigma^{2}_{i}Z_{i}^{2}$ where the random variables
$Z_{i}$ are independent standard Gaussians. The proof of
Theorem~\ref{theorem:random} will show that $\Psi$ is the weak limit
of some sequence of such random variables. Unfortunately it does not
seem possible to compute the variances $\sigma^{2}_{i}$ in these
approximating quadratic forms, so the limit distribution $\Psi$ cannot
be explicitly identified.} Moreover, a closer look will show that for
central limit behavior the typical order of magnitude of fluctuations
should be $T^{3/2}$, not $T$. This is what occurs for \emph{localized}
self-intersection counts, as we now explain.

Label the points of self-intersection of $\gamma ([0,T])$ on
$\Upsilon$ as $x_{1},x_{2},\dotsc ,x_{N (T)}$ (the ordering is
irrelevant). For any smooth, nonnegative function $\varphi:\Upsilon
\rightarrow \zz{R}_{+}$ define the $\varphi-$\emph{localized}
self-intersection count $N_{\varphi} (T)$ by
\begin{equation}\label{eq:localSI}
	N_{\varphi} (T)=N_{\varphi} (\gamma [0,T])=\sum_{i=1}^{N (T)} \varphi (x_{i}).
\end{equation}
Like the global self-intersection count $N (T)$, the localized
self-intersection count $N_{\varphi} (T)$ grows quadratically in $T$: in
particular, if the initial tangent vector $(x,\theta)$ is chosen
randomly according to the normalized Liouville measure then with
probability one,
\[
	\lim_{T \rightarrow \infty}\frac{N_{\varphi}
	(T)}{T^{2}}=\kappa_{\Upsilon}\xnorm{\varphi}_{1} 
\]
where $\xnorm{\varphi}_{1}$ denotes the integral of $\varphi$ against
normalized surface area measure on $\Upsilon$.

\begin{theorem}\label{theorem:local}
For any compact, negatively curved surface $\Upsilon$ there is a
constant $\varepsilon >0$ with the following property. If $\varphi\geq
0$ is smooth and not identically $0$ but has support of diameter less
than $\varepsilon$, then under the hypotheses of
Theorem~\ref{theorem:random}, for some constant $\sigma >0$ depending
on $\varphi$,
\begin{equation}\label{eq:local}
	\frac{N_{\varphi} (T)-\kappa_{\Upsilon}\xnorm{\varphi}_{1}T^{2}}{\sigma T^{3/2}}
	\stackrel{\mathcal{D}}{\longrightarrow} \Phi 
\end{equation}
a $T \rightarrow \infty$. Here $\Phi$ is the standard unit Gaussian
distribution on $\zz{R}$.
\end{theorem}

In the course of the proof we will show that a lower bound for
$\varepsilon$ is the distance between two non-intersecting closed
geodesics.

\subsection{Self-intersections of closed geodesics}\label{ssec:closed}
On any compact, negatively curved surface there are countably many
closed geodesics, and only finitely many with length 
in any given bounded interval $[0,L]$. According to the celebrated
``Prime Geodesic Theorem'' of Margulis \cite{margulis},
\cite{sharp:book}, the number $\pi (L)$ of closed geodesics of length
$\leq L$ satisfies the asymptotic law
\[
	\pi (L)\sim \frac{e^{hL}}{hL} \quad \text{as} \;\; L
	\rightarrow \infty,
\]
where $h>0$ is the topological entropy of the geodesic
flow. Furthermore, the closed geodesics are equidistributed according
to the maximal entropy measure, in the following sense
\cite{lalley:acta}: if a closed geodesic is chosen at random from
among those of length $\leq L$, then with probability approaching $1$
as $L \rightarrow \infty$, the empirical distribution of the geodesic
chosen will be in a weak neighborhood of the maximal entropy
measure. (See also Bowen \cite{bowen:equidistribution} for a somewhat
weaker statement. It does not matter whether the random closed
geodesic is chosen from the set of \emph{prime} closed geodesics or
the set of \emph{all} closed geodesics, because Margulis' theorem
implies that the number of non-prime closed geodesics of length $\leq
L$ is $O (e^{hL/2})$.)  In addition, the maximal entropy measure
governs the statistics of closed geodesics even at the level of
``fluctuations'', in that central limit theorems analogous to that
governing random geodesics (cf. \cite{ratner:clt}) hold for randomly
chosen \emph{closed} geodesics -- see \cite{lalley:axiomA} and
\cite{lalley:homology} for precise statements. Thus, it is natural to
expect that the maximum entropy measure also controls the statistics
of self-intersections.

For compact surfaces  of \emph{constant}
negative curvature the maximum entropy measure and the (normalized)
Liouville measure coincide, so it is natural to expect that in this
case there should be some connection between the fluctuations in
self-intersection count of closed geodesics with those of random
geodesics. For compact surfaces of \emph{variable} negative curvature,
however, the maximum entropy measure and the Liouville measure are
mutually singular. Thus, it is natural to conjecture that the
fluctuations of self-intersection counts in the constant curvature and
variable curvature cases to obey different statistical laws. The next
theorem asserts that this is the case.

\begin{theorem}\label{theorem:closed}
Let $\Upsilon$ be a compact surface of  negative
curvature, let $\gamma_{L}$ be a closed geodesic randomly chosen from
the $\pi (L)$ closed geodesics of length $\leq L$, and let $N
(\gamma_{L})$ be the number of self-intersections of $\gamma_{L}$.
(A) If $\Upsilon$ has \emph{constant} negative curvature,
then  for some probability
distribution $\Psi$ on $\zz{R}$, as
$L \rightarrow \infty$,
\begin{equation}\label{eq:closed-CC}
	\frac{N (\gamma_{L})-\kappa_{\Upsilon}L^{2}}{L}
	\stackrel{\mathcal{D}}{\longrightarrow} \Psi .
\end{equation}
(B) If $\Upsilon$ has \emph{variable} negative curvature then for some
constants $\sigma>0$ and $\kappa^{*}>0$,
\begin{equation}\label{eq:closed-VC}
	\frac{N (\gamma_{L})-\kappa^{*}L^{2}}{L^{3/2}}
	\stackrel{\mathcal{D}}{\longrightarrow} \Phi ,
\end{equation}
where $\Phi$ is the standard unit Gaussian distribution.
\end{theorem}

In the variable curvature case, the constant $\kappa^{*}$ need not in
general be the same as the constant $\kappa_{\Upsilon}$ in the
corresponding law \eqref{eq:random}, because the self-intersection
statistics for closed geodesics are governed by the maximum entropy
invariant measure, whereas for random geodesics they are governed by
the Liouville measure.

Similar results hold for the number of intersections of two randomly
chosen closed geodesics. Let $\gamma_{L}$ and $\gamma'_{L}$ be
\emph{independently} chosen at random from the set of closed geodesics
of length $\leq L$, and let $N (\gamma_{L},\gamma '_{L})$ be the
number of intersections of $\gamma_{L}$ with $\gamma'_{L}$. (If by
chance $\gamma_{L}=\gamma '_{L}$, set
$N(\gamma_{L},\gamma'_{L})=N(\gamma_{L})$. Because the probability of
choosing the same closed geodesic twice is $1/\pi (L) \rightarrow 0$,
this event has negligible effect on the distribution of
$N(\gamma_{L},\gamma '_{L})$ in the large $L$ limit.)

\begin{theorem}\label{theorem:closed-pair}
If $\Upsilon$ has constant negative curvature then for some
probability distribution $\Psi^{*}=\Psi^{*}_{\Upsilon}$ on $\zz{R}$, 
\begin{equation}\label{eq:closed-pair-CC}
	\frac{N (\gamma_{L},\gamma'_{L})-\kappa_{\Upsilon}L^{2}}{L}
	\stackrel{\mathcal{D}}{\longrightarrow} \Psi^{*} 
\end{equation}
as $L \rightarrow \infty$. If $\Upsilon$ has variable negative
curvature then  
\begin{equation}\label{eq:closed-pair-CC}
	\frac{N (\gamma_{L},\gamma'_{L})-\kappa^{*}L^{2}}{L^{3/2}}
	\stackrel{\mathcal{D}}{\longrightarrow} \Phi .
\end{equation}
\end{theorem}

We shall omit the proof, as it is very similar to that
of Theorem~\ref{theorem:closed}.

It is noteworthy that in constant negative curvature the order of
magnitude of typical fluctuations in Theorems
\ref{theorem:closed}--\ref{theorem:closed-pair} is $L$. This should be
compared to the main result of \cite{chas-lalley}, which concerns
fluctuations in self-intersection number for randomly closed curves on
a surface, where sampling is by \emph{word length} rather than
\emph{geometric length}. Let $\Upsilon$ be be an orientable, compact
surface with boundary and negative Euler characteristic $\chi$, and
let $\mathcal{F}=\mathcal{F}_{\Upsilon}$ be its fundamental
group. This is a free group on $2g$ generators, where
$g=2-2\chi$. Each conjugacy class in $\mathcal{F}$ represents a free
homotopy class of closed curves on $\Upsilon$. For each such conjugacy
class $\alpha$ there is a well-defined \emph{word-length} $L=L
(\alpha)$ (the minimal word length of a representative element) and a
well-defined \emph{self-intersection count} $N (\alpha)$ (the minimum
number of transversal self-intersections of a closed curve in the free
homotopy class). The main result of \cite{chas-lalley} states that if
$\alpha$ is randomly chosen from among all conjugacy classes with word
length $L$ then for certain positive constants $\kappa^{*},\sigma^{*}$
depending on the Euler characteristic, as $L \rightarrow \infty$,
\begin{equation}\label{eq:chas-lalley}
	\frac{N (\alpha_{L})-\kappa^{*}_{\Upsilon}L^{2}}{\sigma^{*} L^{3/2}}
	\stackrel{\mathcal{D}}{\longrightarrow} \Phi 
\end{equation}
where $\Phi$ is the standard unit Gaussian distribution. The methods
of this paper can be adapted to show that the main result of
\cite{chas-lalley} extends to compact surfaces without boundary and
with genus $g\geq 2$. To reconcile this with
Theorem~\ref{theorem:closed} (which at first sight might appear to
suggest that fluctuations should be on the order of $L$, not
$L^{3/2}$), observe that when closed geodesics are randomly chosen
according to word length $L$, the order of magnitude of fluctuations
in geometric length is $L^{1/2}$; this is enough to increase the size
of typical fluctuations in self-intersection count by a factor
$L^{1/2}$.

\bigskip \textbf{Standing Notation.}  Throughout the paper,
$p:S\Upsilon \rightarrow \Upsilon$ will denote the natural projection
from the unit tangent bundle $S\Upsilon$ to the surface $\Upsilon$, and $\gamma
(t)=\gamma (t;v)$ will denote the orbit of the geodesic flow with
initial tangent vector $v\in S\Upsilon$. The letter $\pi$ will be
reserved for the semi-conjugacy of flows constructed in
section~\ref{sec:symbolicDynamics}, and $\phi_{t}$ for the suspension
flow in this construction. Finally, $\sigma$ will be used to denote
the unilateral shift on any of the sequence spaces $\Sigma
,\Sigma^{+}$, etc. used in the symbolic dynamics.

\bigskip \textbf{Acknowledgments.} The author thanks Moira Chas,
Jayadev Athreya, and Dimitri Dolgopyat for some helpful discussions,
and a referee for the reference to Otal's paper~\cite{otal}.

\section{Intersection kernel}\label{sec:kernel}

The proofs of the main results will rely on the fact that the
self-intersection counts are (in the terminology of \cite{hoeffding})
$U-$\emph{statistics}, that is, functions that can be written in the
form \eqref{eq:u-statistic-A} below. In this section we collect some
essential properties of the kernel function $H_{\delta}$ in this
representation, and then present a short heuristic argument that shows
how the $U-$statistic representation can be used to reduce the weak
convergence assertion \eqref{eq:random} to Ratner's central limit
theorem for the geodesic flow.

\subsection{The intersection kernel}\label{ssec:i-kernel} Geodesics on
any surface, regardless of its
curvature, look locally like straight lines. Hence, for any
\emph{compact} surface $ \Upsilon$ with smooth Riemannian metric there
exists $\varrho >0$ such that if $\alpha$ and $\beta $ are geodesic
segments of length $\leq \varrho$ then $\alpha$ and $\beta$ intersect
transversally, if at all, in at most one point. It follows that for
any geodesic segment $\gamma$ of length $T$ the self-intersection
number $N (\gamma)=N_{T} (\gamma)$ can be computed by partitioning
$\gamma$ into nonoverlapping segments of common length $\delta \leq
\varrho$ and counting the number of pairs that intersect
transversally.  Let $\alpha_{i}$ and $\alpha_{j}$ be two such
segments; then the event that these segments intersect is completely
determined by their initial points and directions, as is the angle of
intersection.

\begin{definition}\label{definition:intersectionKernel}
The \emph{intersection kernel} $H_{\delta}:S\Upsilon\times S\Upsilon \rightarrow
\zz{R}_{+}$ is the nonnegative function that takes the value
$H_{\delta} (u,v)=1$ if the geodesic segments of length
$\delta$ with initial tangent vectors $u$ and $v$ intersect
transversally, and $H_{\delta} (u,v)=0$ otherwise.
\end{definition}

Assume henceforth that $\delta \leq \varrho$; then for any geodesic
$\tilde{\gamma} $,
\begin{equation}\label{eq:u-statistic-A}
	N_{T} (\tilde{\gamma} )=\frac{1}{2} \sum_{i=1}^{n} \sum_{j=1}^{n}
	H_{\delta} (\tilde{\gamma} (i\delta ), \tilde{\gamma} (j\delta)).
\end{equation}
The factor  $1/2$ compensates for the double-counting that results
from letting both indices of summation $i,j$ range over all $n$
geodesic segments.  The diagonal terms in this sum are all $0$,
because the segment $\gamma (i\delta )$ does not intersect itself
{transversally}.

\subsection{The associated integral
operators}\label{ssec:integralOperator}
The intersection kernel  $H_{\delta } (u,v)$ is symmetric in its
arguments and Borel measurable, but not continuous, because 
self-intersections can be created or destroyed by small perturbations
of the initial vectors $u,v$.  Nevertheless,  $H_{\delta}$ induces a
self-adjoint integral operator on the Hilbert space $L^{2} (\nu_{L})$ by
\begin{equation}\label{eq:kernel}
	H_{\delta} \psi (u)=\int_{v\in S\Upsilon} H_{\delta} (u,v) \psi (v)
	\,d\nu_{L} (v). 
\end{equation}

\begin{lemma}\label{lemma:H1}
For all sufficiently small $\delta >0$,
\begin{equation}\label{eq:H1}
	H_{\delta} 1(u):
	 =\int H_{\delta} (u,v)\,d\nu_{L} (v) =
	 \delta^{2}\kappa
\end{equation}
for all $u\in M$. Thus, the constant function $1$ is an eigenfunction
of the operator $H_{\delta}$, and consequently the normalized kernel
$H_{\delta} (u,v)/\delta^{2}\kappa_{\varphi}$ is a symmetric Markov
kernel on $S\Upsilon \times S\Upsilon$.
\end{lemma}

\begin{remark}\label{remark:liouville}
This result (simple though it may be) is the crucial geometric
property of the intersection kernel. Clearly, the intersection kernel
induces an integral operator on $L^{2} (\mu)$ for any finite measure
$\mu$ on $S\Upsilon$ (just replace $\nu_{L}$ by $\mu$ in the
definition \eqref{eq:kernel}). But in general -- and in particular
when $\mu\not =\nu_{L}$ is a Gibbs state -- the constant function $1$
will not be an eigenfunction of this operator.
\end{remark}

\begin{proof}
Denote by $\gamma  =\gamma ( [0,\delta];u)$ the geodesic segment of
length $\delta$ with initial tangent vector $u$.  For small $\delta
>0$ and fixed angle $\theta$, the set of points $x\in \Upsilon $ such that a
geodesic segment of length $\delta$ with initial base point $x$
intersects $\gamma $ at angle $\theta $ is approximately a rhombus of
side $\delta$ with an interior angle $\theta$, with area
$\delta^{2}|\sin \theta |$, and this approximation is asymptotically
sharp as $\delta \rightarrow 0$. Consequently, as $\delta \rightarrow 0$,
\begin{equation}\label{eq:sim}
	\int H_{\delta} (u,v)\,d\nu_{L} (v) \sim
	\delta^{2}\int_{0}^{2\pi} \varphi (\theta) |\sin \theta | \,d\theta / (
	2\pi |M|)=\delta^{2} \kappa_{\varphi},
\end{equation}
and the relation $\sim$ holds uniformly for $u\in S\Upsilon$.

It remains to show that the approximate equality $\sim$ is actually an
equality for small $\delta >0$. Recall that for $\delta \leq \varrho$, any two
distinct geodesic segments of length $\delta$ can intersect transversally
at most once. Consider the geodesic segments of length $\delta$ with
initial direction vectors $u$ and $v$. For any integer $m\geq 2$, each
of these segments can be partitioned into $m$ non-overlapping
sub-segments (each open on one end and closed on the other) of length
$\delta /m$. At most one pair of these constituent sub-segments can
intersect; hence,
\[
	 H_{\delta} (u,v) =\sum_{i=0}^{m-1} \sum_{j=0}^{m-1}
	 	    H_{\delta /m} (\tilde{\gamma}  (i\delta ;u),\tilde{\gamma} (j\delta
		    ;v)) .	
\]
Integrating over $v$ with respect to the Liouville measure $\nu_{L}$
and using the invariance of $\nu_{L}$ relative to the geodesic flow
we obtain that
\[
	H_{\delta}1 (u)
	= \sum_{i=0}^{m-1} mH_{\delta /m}1 (\tilde{\gamma}  (i\delta ;u)).
\]
Let $m \rightarrow \infty$ and use the approximation \eqref{eq:sim}
(with $\delta$ replaced by $\delta /m$); since this approximation
holds uniformly, it follows that $H_{\delta}1 (u)=\delta^{2}\kappa_{S}$.
\end{proof}

Lemma~\ref{lemma:H1} is the only result from this section that will be
needed for the proofs of the main results. The remainder of this
section is devoted to a proof of the ``law of large numbers''
\eqref{eq:lln} and to a  heuristic argument for
Theorem~\ref{theorem:random} that is  simpler and more
illuminating  than the formal proofs that will follow.

\begin{lemma}\label{lemma:compactness}
For each sufficiently small $\delta >0$, the integral operator
$H_{\delta}$ on $L^{2} (\nu_{L})$ is compact.
\end{lemma}

\begin{proof}
If the kernel $H_{\delta} (u,v)$ were jointly continuous in its
arguments $u,v$ then this would follow by standard results about
integral operators --- see, e.g., \cite{widom}. Since $H_{\delta}$ is
not continuous, these standard results do not apply; nevertheless, the
argument for compactness is elementary.  It suffices to show that the
mapping $u \mapsto H_{\delta} (u,\cdot)$ is continuous relative to the
$L^{2}-$norm. Take $u,u' \in S\Upsilon$, and let $\alpha ,\alpha '$ be the
geodesic segments of length $\delta$ starting at $u,u'$, respectively.
If $u,u'$ are close, then the geodesic segments $\alpha ,\alpha '$ are
also close. Hence, for all but very small angles $\theta$ the set of
points $x\in M$ such that a geodesic segment of length $\delta$ with
initial base point $x$ intersects $\alpha$ at angle $\theta $ but does
not intersect $\alpha '$ is small. Consequently, the functions
$H_{\delta} (u,\cdot)$ and $H_{\delta} (u',\cdot)$ differ on a set of
small measure.
\end{proof}

Lemma~\ref{lemma:compactness} implies that Hilbert-Schmidt theory
(cf. \cite{widom}) applies. In particular, the non-zero spectrum of
$H_{\delta}$ consists of isolated real eigenvalues $\lambda_{j}$ of
finite multiplicity (and listed according to multiplicity). The
corresponding (real) eigenfunctions $\psi_{j}$ can be chosen so as to
consititute an orthonormal basis of $L^{2} (\nu_{L})$, and the
eigenvalue sequence $\lambda_{j}$ is square-summable.

\begin{lemma}\label{lemma:mixing}
The  kernel $\bar{H}_{\delta}:=H_{\delta}/\delta^{2}\kappa_{\varphi}$
satisfies the \emph{Doeblin condition}: there exist an integer $n\geq
1$ and a positive real number $\varepsilon$ such that
\begin{equation}\label{eq:doeblin}
	\bar{H}^{(n)}_{\delta} (u,v)\geq \varepsilon  \quad \text{for
	all} \;\; u,v\in S\Upsilon,
\end{equation}
where $H^{{(n)}}_{\delta }$ denotes the kernel of the $n-$fold
iterated integral operator $H_{\delta}$.
\end{lemma}

\begin{proof}
Choose $n$ so large that for any two points $x,y\in \Upsilon$ there is a
sequence $\{x_{i} \}_{0\leq i\leq n}$ of $n+1$ points beginning with
$x_{0}=x$ and ending at $x_{n}=y$, and such that each successive pair
of consecutive points 
$x_{i},x_{i+1}$ are at distance $<\delta /4$. Then for any two
geodesic segments $\alpha ,\beta$ of length $\delta$ on $S$ there is a
chain of $n+1$ geodesic segments $\alpha_{i}$, all of length $\delta$,
beginning at $\alpha_{0}=\alpha$ and ending at $\alpha_{n}=\beta$,
such that any two successive segments $\alpha_{i}$ and $\alpha_{i+1}$
intersect transversally.  Since the intersections are transversal, the
initial points and directions of these segments can be jiggled
slightly without undoing any of the transversal intersections. This
implies \eqref{eq:doeblin}.
\end{proof}

\begin{corollary}\label{corollary:doeblin}
The eigenvalue $\delta^{2}\kappa_{\varphi}$ is a simple eigenvalue of the
integral operator $H_{\delta}$, and the rest of the spectrum lies in
a disk of radius $<\delta^{2}\kappa_{\varphi}$.
\end{corollary}

\begin{proof}
This is a standard result in the theory of Markov operators.
\end{proof}

\begin{corollary}\label{corollary:meanZeroEigenfunctions}
For every $j\geq 2$ the eigenfunction $\psi_{j}$ has mean zero
relative to $\nu_{L}$, and distinct eigenfunctions are uncorrelated.
\end{corollary}

\begin{proof}
The spectral theorem guarantees   orthogonality of the
eigenfunctions. The key point is that $\psi_{1}=1$ is an
eigenfunction, and so the orthogonality $\psi_{j}\perp \psi_{1}$
implies that each $\psi_{j}$ for $j\geq 2$ has mean zero.
\end{proof}

\begin{lemma}\label{lemma:nontrivialEigenvalues}
If $\delta >0$ is sufficiently small then $H_{\delta}$ has
eigenvalues other than $0$ and $\lambda_{1} (\delta)$.
\end{lemma}

\begin{proof}
Otherwise, the Markov operator  $\bar{H}_{\delta}$ would be a
projection operator: for every $\psi \in L^{2} (\nu_{L})$ the
function $\bar{H}_{\delta}\psi$ would be constant. But if $\delta
>0$ is small, this is obviously not the  case.
\end{proof}

\subsection{Law of large numbers}\label{ssec:lln}
The law of large numbers \eqref{eq:lln} for random geodesics can be
deduced from Birkhoff's ergodic theorem using the representation
\eqref{eq:u-statistic-A} of the self-intersection count. The first step
is to approximate the kernel $H_{\delta}$ by continuous kernels. Fix $0<\delta
< \varrho$, where $\varrho >0$ is small enough that any two
geodesic segments on the surface $\Upsilon$ of length $\varrho$ will
intersect transversally at  most once.

\begin{lemma}\label{lemma:sandwich}
For each $\varepsilon >0$ there exist continuous functions
$H_{\delta}^{-},H_{\delta}^{+}: S\Upsilon \times S\Upsilon \rightarrow
[0,1]$ such that $H_{\delta}^{-}\leq H_{\delta}\leq H_{\delta}^{+}$
and such that for each $u\in S\Upsilon$,
\begin{equation}\label{eq:sandwich}
	\int (H_{\delta}^{+} (u,v)-H_{\delta}^{-} (u,v))\,d\nu_{L}
	(v)<\varepsilon.
\end{equation}
\end{lemma}

\begin{proof}
Fix $\varepsilon '>0$ such that $\delta +2\varepsilon '<\varrho $, and
let $\psi :[0,1 ]\rightarrow [0,1]$ be a continuous  function such that
$\psi (0)=\psi (1)=0$ and $\psi =1$ on the interval $[\varepsilon ',1
-\varepsilon ']$. For unit tangent vectors $u,v\in S\Upsilon$ such
that the geodesic segments $\gamma_{u},\gamma_{v}$ of length $\delta$
based at $u,v$ intersect at angle $\theta \in (0,\pi)$ at times
$t_{u},t_{v}\in [0,\delta]$, set
\[
	H_{\delta}^{-} (u,v)=\psi (\theta/\pi )\psi (t_{u}/\delta)\psi
	(t_{v}/\delta),
\]
and for all other $u,v$ set $H_{\delta}^{-} (u,v)=0$.  Similarly, for 
unit tangent
vectors $u,v\in S\Upsilon$  such that the geodesic segments $\gamma_{u},\gamma_{v}$
of length $\delta+2\varepsilon '$ based at $u,v$ intersect at times
$t_{u},t_{v}\in (-\varepsilon ',\delta +\varepsilon ')$,
set
\[
	H_{\delta}^{+} (u,v)=\psi (t_{u}/ (\delta+2\varepsilon ')\psi
	(t_{v}/ (\delta +2\varepsilon ')),
\]
and for all other $u,v$ set $H_{\delta}^{+} (u,v)=0$.  Clearly, $0\leq
H_{\delta}^{-}\leq H_{\delta}\leq H_{\delta}^{+}$, and by an argument
like that in the proof of Lemma~\ref{lemma:H1} it can be shown that if
$\varepsilon '>0$ is sufficiently small then \eqref{eq:sandwich} will
hold for all $u$.
\end{proof}

\begin{proposition}\label{proposition:ergodic}
Let $(\mathcal{X},d)$ be a compact metric space and let
$K:\mathcal{X}^{2} \rightarrow \zz{R}$ be continuous. If $\mu$ is a
Borel probability measure on $\mathcal{X}$ and $T:\mathcal{X}
\rightarrow \mathcal{X}$ is an ergodic, measure-preserving
transformation (not necessarily continuous) relative to $\mu$, then
\begin{equation}\label{eq:ergodic}
	\lim_{n \rightarrow \infty}
	\frac{1}{n^{2}}\sum_{i=1}^{n}\sum_{j=1}^{n} K (T^{i}x,T^{j}x)
	= \iint_{\mathcal{X}\times \mathcal{X}} K (y,z) \,d\mu (y) d\mu (z)
\end{equation}
for $\mu -$almost every $x$.
\end{proposition}

\begin{proof}
The function $K$ is bounded, since it is continuous,
so the double integral in \eqref{eq:ergodic} is well-defined and
finite. Furthermore, the set of functions $K_{x}$ defined by $K_{x}
(y) :=K (x,y)$, where $x$ ranges over the space $\mathcal{X}$, is
equicontinuous, and the function
\[
	\bar{K}_{x}:=\int_{\mathcal{X}}K_{x} (y) \,d\mu (y)
\]
is continuous in $x$.  The equicontinuity of the functions $K_{x}$
implies, by the Arzela-Ascoli theorem, that for any $\varepsilon >0$
there is a finite subset $F_{\varepsilon}=\{ x_{i}\}_{1\leq i\leq I}$
such that for any $x\in SM$ there is at least one index $i\leq I$ such
that
\begin{equation*}
	\xnorm{K_{x}-K_{x_{i}}}_{\infty}<\varepsilon .
\end{equation*}
It follows that the time average of $K_{x}$ along any trajectory
differs from the corresponding time average of $K_{x_{i}}$ by less
than $\varepsilon$.  Since the set $F_{\varepsilon}$ is finite,
Birkhoff's  theorem implies that for $\mu -$a.e. $x\in \mathcal{X}$,
\[
	\lim_{n \rightarrow \infty} \frac{1}{n}\sum_{j=1}^{n}K
	(y,T^{j}x) = \int K (y,x')\, d\mu (x') \quad 
	\text{for each} \;\;y\in F_{\varepsilon}.
\] 
Consequently,
it follows from equicontinuity (let $\varepsilon \rightarrow
0$) and the continuity in $x$ of the averages $\bar{K}_{x}$ that
almost surely
\begin{equation*}
    \lim_{n \rightarrow \infty} \frac{1}{n}\sum_{j=1}^{n}K
	(y,T^{j}x) = \int K (y,x) \,d\mu (y) 
\end{equation*}
\emph{uniformly} for all $y\in \mathcal{X}$. The uniformity of this
convergence guarantees that \eqref{eq:ergodic} holds $\mu -$almost
surely.
\end{proof}

\begin{proof}
[Proof of the strong law of large numbers \eqref{eq:lln}]
Let $H_{\delta}^{+}$ and $H_{\delta}^{-}$ be as in the statement of
Lemma~\ref{lemma:sandwich}.  By Proposition~\ref{proposition:ergodic}, 
for $\nu_{L}-$ almost every $u\in S\Upsilon$,
\[
	\lim_{n \rightarrow \infty}
	\frac{1}{n^{2}}\sum_{i=1}^{n}\sum_{j=1}^{n} 
	H_{\delta}^{\pm} (\tilde{\gamma}_{u} (i\delta),\tilde{\gamma}_{u}( j\delta)) 
	=\int H_{\delta}^{\pm} (v,w) \,d\nu_{L} (v)d\nu_{L} (w).
\]
Hence, by Lemma~\ref{lemma:sandwich} (with $\varepsilon ' \rightarrow
0$) and Lemma~\ref{lemma:H1}, for $\nu_{L}-$ almost every $u\in S\Upsilon$,
\[
	\lim_{n \rightarrow \infty}
	\frac{1}{n^{2}}\sum_{i=1}^{n}\sum_{j=1}^{n} 
	H_{\delta} (\gamma_{u} (i\delta),\gamma_{u} ( j\delta)) 
	=\int H_{\delta} (v,w) \,d\nu_{L} (v)d\nu_{L} (w)=\delta^{2}\kappa. 
\]
This proves that \eqref{eq:lln} holds for $t \rightarrow \infty$ along
the sequence $t=n\delta$. Since $\delta >0$ is arbitrary, and since
the self-intersection counts are obviously monotone in $t$, relation
\eqref{eq:lln} follows.
\end{proof}

\subsection{Weak convergence: heuristics}\label{ssec:heuristics} The
results of sections \ref{ssec:i-kernel}---\ref{ssec:integralOperator}
can be used to give a compelling --- but non-rigorous --- explanation
of the weak convergence asserted in Theorem~\ref{theorem:random}.  The
Hilbert-Schmidt theorem asserts that a symmetric integral kernel in
the class $L^{2} (\nu_{L}\times \nu_{L})$ has an $L^{2}-$convergent
eigenfunction expansion. The intersection kernel $H_{\delta} (u,v)$
meets the requirements of this theorem, and so its eigenfunction
expansion converges in $L^{2} (\nu_{L}\times \nu_{L})$:
\begin{equation}\label{eq:eigenfunctionExp}
	H_{\delta} (u,v)=\sum_{k=1}^{\infty} \lambda_{k} 
	\psi_{k} (u) \psi_{k} (v).
\end{equation}
The $L^{2}-$convergence of the series does not, of course, imply
\emph{pointwise} convergence;  this is why the following argument is
not a proof. Nevertheless, let us proceed formally, ignoring convergence
issues. Recall (Corollary~\ref{corollary:meanZeroEigenfunctions}) that
the eigenfunctions are mutually uncorrelated, and so all except the
constant eigenfunction $\psi_{1}$ have mean zero relative to
$\nu_{L}$. Thus, the representation \eqref{eq:u-statistic-A} of the
intersection number $N_{\varphi} (n\delta)$ can be rewritten as
follows:
\begin{align}\label{eq:heuristics}
	N_{\varphi } (n\delta) - (n\delta )^{2}\kappa_{g}&=  
	\frac{1}{2} \sum_{i=1}^{n} \sum_{j=1}^{n} H_{\delta} (\tilde{\gamma} 
	(i\delta ), \tilde{\gamma}  (j\delta)) - (n\delta )^{2}\kappa_{g}\\
\notag 	&=	\frac{1}{2} \sum_{i=1}^{n} \sum_{j=1}^{n}
	  \sum_{k=2}^{\infty} \lambda_{k} (\delta) \psi_{k} (\tilde{\gamma}
	  (i\delta ))\psi_{k} (\tilde{\gamma} (j\delta )) \\
\notag 	  &=\frac{1}{2} \sum_{k=2}^{\infty} \lambda_{k} (\delta) \left(
	  \sum_{i=1}^{n} \psi_{k} (\tilde{\gamma} 
	  (i\delta ))\right)^{2}.
\end{align}
If the eigenfunctions $\psi_{j}$ were H\"{o}lder
continuous, the central limit theorem for the geodesic flow
\cite{ratner:clt} would imply that  for
any finite $K$ the joint distribution of the random vector 
\begin{equation}\label{eq:K-vector}
	\left( \frac{1}{\sqrt{n}}\sum_{i=1}^{n} \psi_{k} (\tilde{\gamma}
	(i\delta ))\right)_{2\leq k\leq K}
\end{equation}
converges, as $n \rightarrow \infty$, to a (possibly degenerate)
$K-$variate Gaussian distribution centered at the origin. (The central
limit theorem in \cite{ratner:clt} is stated only for the case $K=1$,
but the general case follows by standard weak convergence arguments
[the ``Cramer-Wold device''], as in \cite{billingsley:wc}, ch.~1.) Hence,
for every $K<\infty$ the distribution of the truncated sum
\begin{equation}\label{eq:K-form}
	\frac{1}{n} \sum_{k=2}^{K} \lambda_{k} (\delta) \left(
	\sum_{i=1}^{n} \psi_{k} (\tilde{\gamma} 
	  (i\delta ))\right)^{2}
\end{equation}
should converge, as $n \rightarrow \infty$, to that of a quadratic
form in the entries of the limiting Gaussian
distribution. \footnote{This would follow by the
spectral theorem for symmetric matrices and elementary properties of
the multivariate Gaussian distribution. To see this,  suppose that the
limit distribution of the random vector \eqref{eq:K-vector} is
mean-zero Gaussian with (possibly degenerate) covariance matrix
$\Sigma$; this distribution is the same as that of $\Sigma^{1/2}Z$,
where $Z$ is a Gaussian random vector with mean zero and identity
covariance matrix. Let $\Lambda$ be the diagonal matrix with diagonal
entries $\lambda_{j} (\delta)$. Then the limit distribution of
\eqref{eq:K-form} is identical to that of $Z^{T}MZ$, where
$M=\Sigma^{1/2}\Lambda \Sigma^{1/2}$. But the matrix $M$ is symmetric,
so it may be factored as $M=U^{T} D U$, where $U$ is an orthogonal
matrix and $D$ is diagonal.  Now if $Z$ is mean-zero Gaussian with the
identity covariance matrix, then so is $UZ$, since $U$ is
orthogonal. Thus, $Z^{T}MZ$ is a quadratic form in independent,
standard normal random variables. }

Unfortunately, it seems that there is no way to make this argument
rigorous, because there is no obvious way to show that the series
\eqref{eq:eigenfunctionExp} converges pointwise. (If the kernel
$H_{\delta}$ were nonnegative semi-definite then Mercer's theorem
might be applied, in conjunction with a smoothing argument; however,
$H_{\delta}$ is in general not nonnegative semi-definite.) Thus, it
will be necessary to proceed by a more indirect route, via
{symbolic dynamics} and {thermodynamic formalism}.

\section{Symbolic dynamics}\label{sec:symbolicDynamics}

\subsection{Shifts and suspension flows}\label{ssec:overview}
\emph{Symbolic dynamics} 
provides an approach to the study of hyperbolic flows that transforms
questions about orbits of the flow to equivalent (or nearly
equivalent) questions about a shift of finite type. The \emph{shift of finite
type} (either one-sided or two-sided) on a finite alphabet $\mathcal{A}$ with 
transition matrix $A:\mathcal{A}\times \mathcal{A} \rightarrow \{0,1 \}$  is the
system $(\Sigma^{+},\sigma)$ or $(\Sigma ,\sigma )$ where
\begin{align*}
	\Sigma^{+}&=\{(x_{n})_{n\geq 0} \in \mathcal{A}^{\zz{Z}_{+}}\;|\, A
	(x_{n},x_{n+1})=1 \;\;\forall \;n\geq 0\} \quad \text{and}\\
	\Sigma &=\{(x_{n})_{n\in \zz{Z}} \in \mathcal{A}^{\zz{Z}}\;|\, A
	(x_{n},x_{n+1})=1 \;\;\forall \;n \in \zz{Z}\} ,
\end{align*}
and $\sigma$ is the forward shift operator on
sequences. (Equivalently, one can define a shift of finite type to be
the forward shift operator acting on the space of all sequences with
entries in $\mathcal{A}$ in which certain subwords of a certain length $r$ do
not occur.)  If there exists $m\geq 1$ such that $A^{m}$ has strictly
positive entries then the corresponding shifts are topologically
mixing relative to the usual topology on $\Sigma$ or $\Sigma^{+}$,
which is metrizable by
\[
	d (x,y)=2^{-n (x,y)},
\]
 where $n (x,y)$ is the minimum nonnegative integer $n$ such that
$x_{j}\not =y_{j}$ for $j=\pm n$. 

For a continuous function  $F:\Sigma \rightarrow (0,\infty)$ on
$\Sigma$ (or on $\Sigma^{+}$), the \emph{suspension flow} under $F$ is
a flow $\phi_{t}$ on the space
\[
	\Sigma_{F} :=\{(x,t)\,:\, x\in \Sigma \;\; \text{and} \; 0\leq
	t\leq F (x)\} 
\]
with points $(x,F (x))$ and $(\sigma x,0)$ identified. In the suspension flow
an orbit beginning at some point $(x,s)$ moves at unit speed up the
fiber over $(x,0)$ until reaching the roof $(x,F (x))$, at which
time it jumps to the point $(\sigma x,0)$ and then proceeds up the
fiber over $(\sigma x, 0)$, that is
\begin{align*}
	\phi_{t} (x,s)&= (x,s+t) \quad \text{if} \; s+t\leq F (x);\\
	\phi_{t} (x,s)&= \phi_{t-F (x)+s} (\sigma x, 0) \quad \text{otherwise}. 
\end{align*}
We shall use the notation
\begin{align*}
	\mathcal{F}_{x}:&=\{(x,s)\,:\,s\in [0,F (x))\} \quad \text{and}\\
	\mathcal{F}^{r}_{x}:&= \phi_{r} (\mathcal{F}_{x})
\end{align*}
for fibers and their time shifts.  
An orbit of the suspension flow is periodic if and only
if it passes through a point $(x,0)$ such that $x$ is a periodic
sequence; if the minimal period of the sequence $x$ is $n$, then the 
minimal period of the corresponding periodic orbit of $\phi_{t}$ is
the sum of the  lengths of the fibers visited by the orbit, which is
given by
 \begin{equation}\label{eq:snf}
	S_{n}F (x):=\sum_{j=0}^{n-1}F (\sigma^{j}x).
\end{equation}

The term ``symbolic dynamics'' is used loosely to denote a coding of
orbits of a flow by elements $x\in \Sigma$ of a shift of finite type.
In the case of a hyperbolic flow, this coding extends to a H\"{o}lder
continuous\footnote{The implied metric on the suspension space
$\Sigma_{F}$ is the ``taxicab'' metric induced by the flow $\phi_{t}$
and the metric $d$ on $\Sigma$ specified above -- see
\cite{bowen-walters} for details. Roughly, the distance between two
points in the suspension space is the length of the shortest path
between them that consists of finitely many segments along flow lines
and finitely many segments of the form $[(x,s),(x',s)]$. The metric on
$S\Upsilon $ is the metric induced by the Riemannian metric on
$T\Upsilon $ -- see, e.g., \cite{paternain}.} semi-conjugacy with
a suspension flow over a shift of finite type with a H\"{o}lder
continuous height function $F$.  Existence of such semi-conjugacies
was proved in general for Axiom A flows by Bowen \cite{bowen:SymbDyn},
and by Ratner \cite{ratner:markov} for geodesic flows on negatively
curved surfaces.

\begin{proposition}\label{proposition:bowen} (Ratner; Bowen) For the
geodesic flow $\gamma_{t}$ on any compact surface $\Upsilon$ 
with a Riemannian metric of negative curvature there
exists a suspension flow $(\Sigma_{F},\phi_{t})$ over a shift of
finite type and a H\"{o}lder continuous  surjection $\pi
:\Sigma_{F} \rightarrow S\Upsilon $ such that
\begin{equation}\label{eq:semi-conjugacy}
	\pi \circ \phi_{t}=\gamma_{t} \circ \pi \quad \text{for
	all}\;\; t\in \zz{R}.
\end{equation}
The suspension flow $(\Sigma_{F},\phi_{t})$ and the projection $\pi$
can be chosen in such a way that the following properties hold:
\begin{enumerate}
\item [(A)] For some $N<\infty$, the mapping $\pi$ is at most
$N-$to$-1$.
\item [(B)] The suspension
flow and the geodesic flow have the same  topological entropy $\theta>0$.
\item [(C)] For some $\varepsilon >0$, 
\begin{itemize}
\item [(i)] the projection $\pi$ is $\mu -$almost surely
one-to-one for every Gibbs state $\mu$ with entropy larger than
$\theta -\varepsilon$, and
\item [(ii)] The number $M (t)$ of closed geodesics with
more than one $\pi -$pre-image and prime period $\leq t$ satisfies 
\[
	\limsup_{t \rightarrow \infty} t^{-1}\log M (t)\leq \theta
	-\varepsilon .
\]
\end{itemize}
\end{enumerate}
\end{proposition}

For the definition of a Gibbs state, see section~\ref{sec:gibbs}
below; both the Liouville measure and the maximum entropy invariant
measure are Gibbs states. The conclusions (C)-(i) and (C)-(ii) are not
explicitly stated in \cite{bowen:book}, but both follow from Bowen's
construction. See \cite{parry-pollicott}, sec.~3 for further
discussion of this point. Finally, observe that if $\pi$ is a
semi-conjugacy as in Proposition~\ref{proposition:bowen}, then so is
the mapping $\pi_{s}=\pi \circ \phi_{s}$, for any $s\in \zz{R}$. (See
Remark~\ref{remark:crossSection} for implications of this.)

\subsection{Series' construction}\label{ssec:series} For the geodesic
flow on a compact surface of \emph{constant} negative curvature, a
different symbolic dynamics was constructed by Series
\cite{series:symbDynFlows} (see also \cite{bowen-series}, and for
related constructions \cite{adler-flatto} and \cite{lustig}). This
construction is better suited to enumeration of self-intersections.
In this section we give a resume of some of the important features of
Series' construction.

Assume first that $\Upsilon$ has {constant} curvature $-1$ and
genus $g\geq 2$. Then the universal covering space of $\Upsilon$ is
the hyperbolic plane $\zz{D}$, realized as the unit disk with the
usual (Poincar\'{e}) metric. The fundamental group $\Gamma =\pi_{1}
(\Upsilon)$ is a discrete, finitely generated, co-compact group of
isometries of $\zz{D}$. Thus, $\Upsilon$ can be identified with
$\zz{D}/\Gamma$. This in turn can be identified with a fundamental
polygon $\mathcal{P}$, with compact closure in $\zz{D}$, whose sides
are geodesic segments in $\zz{D}$ that are paired by elements in a
(symmetric) generating set for $\Gamma$.  The polygon $\mathcal{P}$
can be chosen in such a way that the \emph{even corners} condition is
satisfied: that is, each geodesic arc in $\partial \mathcal{P}$
extends to a complete geodesic in $\zz{D}$ that is completely
contained in $\cup_{g\in \Gamma}g (\partial \mathcal{P})$.  The
geodesic lines in $\cup_{g\in \Gamma}g (\partial \mathcal{P})$ project
to closed geodesics in $\Upsilon$; because the polygon $\mathcal{P}$
has only finitely many sides, there are only finitely many such
projections. Call these the \emph{boundary geodesics}.

Except for those vectors tangent to one of the boundary geodesics,
each unit tangent vector $v\in S\Upsilon$ can be uniquely lifted to
the unit tangent bundle $S\zz{D}$ of the hyperbolic plane in such a
way that either the lifted vector $L (v)$ has base point in the
interior of $\mathcal{P}$, or lies on the boundary of $\mathcal{P}$
but points \emph{into} $\mathcal{P}$. The vector $L (v)$
uniquely determines a geodesic line in $D$, with initial tangent
vector $L (v)$, which converges to distinct points on the circle at
infinity as $t \rightarrow \pm \infty$. The mapping $L:S\Upsilon
\rightarrow S\zz{D}$ is smooth except at those vectors that lift to
vectors tangent to one of the boundary geodesics; at these vectors,
$L$ is necessarily discontinuous.  Denote by $\Xi\subset S\Upsilon$
the set of all vectors $v$ such that $L (v)$ is based at a point  on
the boundary of $\mathcal{P}$.

For any shift $(\Sigma^{+},\sigma)$, any $x\in \Sigma$ or
$\Sigma^{+}$, and any subset $J\subset \zz{Z}_{+}$ let
$\Sigma^{+}_{J}(x)$ be the cylinder set consisting of all $y\in
\Sigma^{+}$ such that $x_{j}=y_{j}$ for all $j\in J$. For any sequence
$x\in \Sigma$, denote by 
\begin{equation*}
	x^{+}=x_{0}x_{1}x_{2}\dotsb \quad \text{and}\quad 
	x^{-}=x_{-1}x_{-2}x_{-3}\dotsb 
\end{equation*}
the forward and backward coordinate subsequences. The sequence $x^{-}$
need not be an element of $\Sigma^{+}$, since its coordinates are
reversed.  Let $\Sigma^{-}$ be the set of all $x^{-}$ such
that $x\in \Sigma$; then $(\Sigma^{-},\sigma)$ is a shift of
finite type (whose transition matrix $A^{\dagger}$ is the transpose of $A$).

\begin{proposition}\label{proposition:boundaryCorrespondence}
(Series) Let $\Upsilon$ be a compact surface equipped with a Riemannian metric
of constant curvature $-1$.  There exist a shift $(\Sigma ,\sigma)$ of
finite type, a suspension flow $(\Sigma_{F},\phi_{t})$ over the shift,
and surjective, H\"{o}lder-continuous mappings $\xi_{\pm}:\Sigma^{\pm}
\rightarrow \partial \zz{D}$, and $\pi:\Sigma_{F} \rightarrow
S\Upsilon$ such that $\pi$ is a semi-conjugacy with the geodesic flow
(i.e., equation \eqref{eq:semi-conjugacy} holds), and such that the
following properties hold.
\begin{enumerate}
\item [(A)] $\Xi =\pi (\Sigma \times \{0\})$.
\item [(B)] The endpoints on $\partial \zz{D}$ of the geodesic with
initial tangent vector $L\circ \pi (x,0)$ are $\xi_{\pm} (x^{\pm})$.
\item [(C)] $F (x)$ is the time taken by this geodesic line to cross
$\mathcal{P}$. 
\end{enumerate}
Furthermore, the maps $\xi_{\pm}$ send  cylinder sets $\Sigma^{\pm}_{[0,m]} (x)$
onto  closed arcs $J_{m}^{\pm} (x^{\pm})$ such that for certain constants
$C<\infty$ and $0<\beta_{1}<\beta_{2} <1$ independent of $m$ and $x$,
\begin{enumerate}
\item [(D)] the lengths of $J_{m}^{\pm} (x^{\pm})$ are between $C\beta_{1}^{m}$
and $C\beta_{2}^{m}$, and 
\item [(E)] distinct arcs
$J_{m}^{+} (x^{+})$ and $J_{m}^{+} (y^{+})$ of the same generation $m$ have
disjoint interiors (and similarly when $+$ is replaced by $-$).
\end{enumerate}
Consequently, the semi-conjugacy $\pi$ fails to be one-to-one only for
geodesics whose lifts to $\zz{D}$ have at least one endpoint that is
an endpoint of some arc $J_{m}^{\pm} (x)$. Finally, all but finitely many
closed geodesics (the boundary geodesics) correspond uniquely to
periodic orbits of the suspension flow, and for each nonexceptional
closed geodesic the length of the representative sequence in $\Sigma$
is the word length of the free homotopy class relative to the standard
generators of $\pi_{1} (\Upsilon)$. 
\end{proposition}

See \cite{series:symbDynFlows}, especially Th. 3.1, and also
\cite{bowen-series}. The last point is important because it implies
that the set of geodesics where the semi-conjugacy fails to be
bijective is of first category, and has measure zero under any Gibbs
state (in particular, under the Liouville and maximum entropy
measures). 

 Series' construction relies heavily on the hypothesis that the
underlying metric on $\Upsilon$ is of \emph{constant} negative
curvature.  However, the key features of her construction carry over
to metrics of \emph{variable} negative curvature, by virtue of the
\emph{conformal equivalence theorem} (see, for instance,
\cite{schoen-yau}, ch. V) for negatively curved Riemannian metrics on
surfaces and the \emph{structural stability theorem} for Anosov flows
\cite{anosov}, \cite{moser}, \cite{llm}. Structural stability applies
only to small perturbations of Anosov flows, and only geodesic flows
on {negatively curved} surfaces are Anosov, so to use structural
stability globally for geodesic flows we must be able to show that
there is a deformation (homotopy) taking one Riemannian metric to
another \emph{through metrics of negative curvature}
The following easy proposition states that for surfaces,
conformal equivalence of negatively curved metrics implies the
existence of a smooth deformation. (See \cite{gromov:3!b} for a
generalization to higher dimensions.)

\begin{proposition}\label{proposition:conformalDeformation}
Let $\varrho_{0},\varrho_{1}$ be $C^{\infty}$ Riemannian metrics on $\Upsilon$,
both with everywhere negative scalar curvatures. Then there exists a
$C^{\infty}$ deformation $\{\varrho_{t} \}_{t\in [0,1]}$  through
Riemannian metrics with everywhere negative scalar curvatures.
\end{proposition}

\begin{proof}
The conformal equivalence theorem (\cite{schoen-yau}, ch. V, Th. 1.3) implies
that there exists a strictly positive $C^{\infty}$ function $r
=e^{2u}$ on $\Upsilon$ such that $\varrho_{1}=r \varrho_{0}$. The scalar
curvatures $K_{0},K_{1}$ are related by the equation 
\[
	K_{1}=e^{-2u} (K_{0}-2\Delta u)
\]
where $\Delta$ is the Laplace-Beltrami operator with respect to
$\varrho_{0}$. Since $K_{0}$ and $K_{1}$ are both negative everywhere, it
follows that $K_{0}-t\Delta u<0$ for every $t\in [0,1]$. Thus, if
$\varrho_{t}:=e^{-2tu}\varrho_{0}$ then the curvature $K_{t}=e^{-2tu}
(K_{0}-t\Delta u)$ is everywhere negative,  for every $t$.
\end{proof}

Fix Riemannian metrics $\varrho_{0}$ and $\varrho_{1}$ of negative
curvature on $\Upsilon$ such that $\varrho_{0}$ has constant curvature
-1, and let $\varrho_{s}$ be a smooth deformation as in
Proposition~\ref{proposition:conformalDeformation}.  The geodesic flow
on $S\Upsilon$ with respect to any Riemannian metric $\varrho_{s}$ of
negative curvature is Anosov, and if the metrics $\varrho_{s}$ vary
smoothly with $s$ then so do the vector fields of their geodesic
flows. Hence, by the structural stability theorem, for each $s\in
[0,1]$ there exists a H\"{o}lder continuous homeomorphism
$\Phi_{s}:S\Upsilon \rightarrow S\Upsilon $ that maps
$\varrho_{0}-$geodesics to $\varrho_{s}-$geodesics. The H\"{o}lder
exponent is constant in $s$, and the homeomorphisms $\Phi_{s}$ vary
smoothly with $s$ in the H\"{o}lder topology \cite{llm}. Consequently,
the homotopy $\Phi_{s}$ lifts to a homotopy $\tilde{\Phi}_{s}:S\zz{D}
\rightarrow S\zz{D}$ of H\"{o}lder continuous homeomorphisms of the
universal covering space. Each homeomorphism $\tilde{\Phi}_{s}$ maps
$\varrho_{0}-$geodesics to $\varrho_{s}-$geodesics, and for each
$\varrho_{0}-$geodesic $\gamma$ the corresponding $\varrho_{s}-$
geodesic $\tilde{\Phi}_{s} (\gamma)$ converges to the same endpoints
on the circle at infinity $\partial \zz{D}$ as does $\gamma$.

Series' construction gives a semi-conjugacy $\pi_{0}$ of a suspension
flow $(\Sigma_{F_{0}},\phi_{t})$ with the $\varrho_{0}-$geodesic flow
on $S\Upsilon $ that is nearly one-to-one in the senses described in
Proposition~\ref{proposition:boundaryCorrespondence}.  We have just
seen that there is a homotopy of H\"{o}lder continuous homeomorphisms
$\Phi_{s}:S\Upsilon \rightarrow S\Upsilon$ such that each $\Phi_{s}$
maps $\varrho_{0}-$geodesics to $\varrho_{s}-$geodesics. Set $\Phi
=\Phi_{1}$; because $\Phi$ is H\"{o}lder, it lifts to the suspension
flow: in particular, there exist H\"{o}lder continuous $F_{1}:\Sigma
\rightarrow (0,\infty)$ and $\pi_{1}:\Sigma_{F_{1}} \rightarrow
S\Upsilon$ and a H\"{o}lder continuous homeomorphism $\Psi
:\Sigma_{F_{0}} \rightarrow \Sigma_{F_{1}}$ that maps fibers of
$\Sigma_{F_{0}}$ homeomorphically onto fibers of $\Sigma_{F_{1}}$, and
satisfies the conditions
\begin{gather}\label{eq:crossSectionMatch}
	\Psi (x,0)= (x,0) \quad \text{for every}\;\; x\in \Sigma ,
	\quad \text{and} \\
\label{eq:orbitEquivLift}
	\pi_{1}\circ \Psi = \Phi \circ \pi_{0}.
\end{gather}
Thus, the projection $\pi_{1}:\Sigma_{F_{1}}\rightarrow S\Upsilon$ is a
semi-conjugacy between the suspension flow on $\Sigma_{F_{1}}$ and the
geodesic flow on $S\Upsilon$ relative to the metric $\varrho_{1}$. 

\begin{corollary}\label{corollary:series-variable}
For any negatively curved Riemannian metric $\varrho_{1}$ on a
compact surface $\Upsilon$ the suspension flow
$(\Sigma_{F_{1}},\phi_{t})$ and semi-conjugacy $\pi_{1}
:\Sigma_{F_{1}} \rightarrow S\Upsilon $ in
Proposition~\ref{proposition:bowen} can be chosen in such a way that
$\pi$ is one-to-one except on a set of first category, and only
finitely many closed geodesics have more than one pre-image.
\end{corollary}

\subsection{Symbolic dynamics and self-intersection
counts}\label{ssec:cc-gf} For surfaces of constant curvature $-1$ the
symbolic dynamics has the convenient property that the geodesic
segments $p\circ \pi (\mathcal{F}_{x})$ and $p\circ \pi
(\mathcal{F}_{y})$ corresponding to two distinct fibers of the
suspension flow can intersect at most once, since each is a single
crossing of the fundamental polygon. Since will be easiest for us to
count self-intersections of a long geodesic by counting pairs of
fibers whose images cross, we begin by recording a simple modification
of the symbolic dynamics that guarantees only single crossings.

\begin{lemma}\label{lemma:shortFibers}
Given any sufficiently small $\varepsilon >0$ we can assume, without
loss of generality, that the suspension flow has been chosen in such a
way that the height function $F$ satisfies $0< F \leq
\varepsilon$.
\end{lemma}

\begin{proof}
This can be arranged by a simple refinement of the symbolic dynamics
constructed above.  First, consider
the suspension flow obtained by cutting the sections of the flow space
lying above particular initial symbols $x_{0}=a$ into boxes, refining
the alphabet so that there is one symbol per box, and adjusting the
transition rule and the height function accordingly. In detail, fix an
integer $m\geq 1$, replace the original alphabet $\mathcal{A}$ by the
augmented alphabet $\mathcal{A}^{*}:=\mathcal{A}\times [m]$, where
$[m]=\{1,2,...,m \}$, and replace the transition matrix $A$ by the matrix
$A^{*}$ defined by
\begin{align*}
	A^{*} ((a,j),(a',j'))&= 1 \quad \text{if} \;\; a=a'
	\;\;\text{and}\;\;  j'=j+1\leq m;\\
	&=1 \quad \text{if} \;\; A (a,a')=1 \;\;\text{and}\;\; j'=1, j=m;\\
	&=0 \quad \text{otherwise.}
\end{align*}
Define the shift $(\Sigma^{*},\sigma^{*})$ on the enlarged alphabet,
with transition matrix $A^{*}$, accordingly. Let
$\nu:\mathcal{A}^{*}\rightarrow \mathcal{A}$ be the projection on the
first coordinate, and $\nu^{*}:\Sigma^{*} \rightarrow \Sigma$ the
induced projection of the corresponding sequence spaces. Finally,
define $F^{*}:\Sigma^{*} \rightarrow (0,\infty)$ by $F^{*} (x^{*})=F
(\nu^{*} (x^{*}))/m$. Then the mapping $p^{*}:\Sigma^{*}_{F*}
\rightarrow \Sigma_{F}$ defined by
\[
	p^{*} ((x,j),s) = (x,s+ (j-1)F^{*} (x));
\]
provides a conjugacy between the suspension flow
$(\Sigma^{*}_{F*},\phi^{*}_{t})$ with $(\Sigma_{F},\phi_{t})$.
By choosing $m$ large we can arrange that $F^{*}<\varepsilon$.

Unfortunately, this construction introduces periodicity into the
underlying shift $(\Sigma^{*},\sigma)$. This is a nuisance, because
the basic results of thermodynamic formalism \cite{bowen:book},
\cite{parry-pollicott:asterisque} that we will need later, including
the central limit theorem \cite{ratner:clt}, require that the underlying
shift be topologically mixing. But a simple modification of the
foregoing construction can be used to destroy the periodicity. Choose
one symbol $a^{\clubsuit}\in \mathcal{A} $, and for this symbol only,
cut the section of the flow space $\Sigma_{F}$ over $a^{\clubsuit}$ into $m+1$
boxes, instead of the $m$ used in the construction above. Adjust the
transition rule $A^{*}$, the height function $F$, and the projection
mapping $p^{*}$ in the obvious manner to obtain a suspension flow
conjugate to the original flow. The underlying shift for this modified
suspension will be aperiodic, by virtue of the fact that the original
shift $(\Sigma ,\sigma)$ is aperiodic.

Observe that in both of these constructions, the cylinder sets of the
modified shift $(\Sigma^{*},\sigma^{*})$ are contained in cylinder
sets of $(\Sigma ,\sigma)$ of comparable length (i.e., within a factor
$m+1$). Since in Series' symbolic dynamics cylinder sets of length $n$
correspond to boundary arcs in $\delta \zz{D}$ with lengths
exponentially decaying in $n$, the same will be true for the modified
symbolic dynamics.

\end{proof}

By virtue of this lemma, we can assume without loss of generality that
the suspension flow has been chosen in such a way that the images
(under the projection $p\circ \pi$) of any two distinct fibers
$\mathcal{F}_{x}$ and $\mathcal{F}_{y}$ intersect at
most once in $\Upsilon$. Thus, the number of self-intersections of any
geodesic segment can be computed by partitioning the segment into the
images of successive fibers (the first and last segment will only
represent partial fibers) and counting how many pairs intersect. With
this in mind, define $h:\Sigma \times \Sigma \rightarrow \{0,1 \}$ by
setting $h (x,y)=1$ if the fibers $\mathcal{F}_{x}$ and
$\mathcal{F}_{y}$ over $x$ and $y$ project to geodesic segments on
$\Upsilon$ that intersect (transversally), and $h (x,y)=0$ if
not. Clearly, for any periodic sequence $x\in \Sigma$ with minimal
period $m$ the image (under $p\circ \pi$) of the periodic orbit of the
suspension flow containing the point $(x,0)$ will be a closed geodesic
$\gamma$ with self-intersection count
\begin{equation}\label{eq:u-representation}
	N (\gamma) =\frac{1}{2} \sum_{i=1}^{m} \sum_{j=1}^{m} h
	(\sigma^{i}x,\sigma^{j}x) .
\end{equation}

The function $h$ is not continuous, because a small change in the
endpoints or directions of two intersecting geodesic segments can
destroy the intersection.  Nevertheless, for ``most'' sequences
$x,y\in \Sigma$, a ``small'' number of coordinates $x_{j},y_{j}$ will
determine whether or not the geodesic segments corresponding to the
fibers $\mathcal{F}_{x}$ and $\mathcal{F}_{y}$ intersect. Here is a way
to make this precise. For each $m\geq 1$ and each sequence $x\in
\Sigma$, denote by $\Sigma_{[-m,m]} (x)$ the cylinder set consisting of
all $y\in \Sigma$ that agree with $x$ in all coordinates $-m\leq j\leq
m$.  For any two sequences $x,y\in \Sigma$ such that the geodesic
segments  $p\circ \pi
(\mathcal{F}_{x})$ and $p\circ \pi (\mathcal{F}_{y})$ intersect, let
$m (x,y)$ be the least nonnegative 
integer with the following property: for every pair of sequences
$x',y'$ such that $x'\in \Sigma_{[-m,m]} (x)$ and $y'\in
\Sigma_{[-m,m]} (y)$ the geodesic segments $p\circ \pi
(\mathcal{F}_{x'})$ and $p\circ \pi (\mathcal{F}_{y'})$ intersect,
where $\mathcal{F}_{x'}$ and $\mathcal{F}_{y'}$ are the fibers of the
suspension space over $x'$ and $y'$, respectively. For sequences $x,y$
such that the segments  $p\circ \pi
(\mathcal{F}_{x})$ and $p\circ \pi (\mathcal{F}_{y})$ do not
intersect, set $m (x,y)=-1$. Define
\begin{align}\label{eq:hm-def}
	h_{m} (x,y) &=1 \quad \text{if}\;\; m (x,y)=m\geq 0 ;\\
\notag 	      &=0 \quad \text{otherwise}.
\end{align}

\begin{lemma}\label{lemma:h-decomposition}
The function $h$ decomposes as
\begin{equation}\label{eq:h-decomposition}
	h (x,y) =\sum_{m=0}^{\infty}h_{m} (x,y) +h_{\infty} (x,y),
\end{equation}
where $h_{\infty} (x,y)\not =0$ (in which case $h_{\infty} (x,y)=1$) only if the
geodesic segments $p\circ \pi (\mathcal{F}_{x})$ and $p\circ \pi
(\mathcal{F}_{y})$ intersect at an endpoint of one of the two
segments. The functions $h_{m}$ satisfy the following properties:

\begin{enumerate}
\item [(A1)] $h_{m} (x,y)$ depends only on the coordinates $x_{i},y_{i}$
with $|i|\leq m$;
\item [(A2)] $h_{m} (x,y)\not =0$ for at most one $m$; and
\item [(A3)] for some $0<\varrho <1$ and $C<\infty$ not depending on $n$, if
$\sum_{m\geq n}h_{m} (x,y)+h_{\infty} (x,y)\not =0$ then the geodesics
corresponding to the orbits of the suspension flow through $(x,0)$ and
$(y,0)$ intersect
either
\begin{itemize}
\item [(a)]  at an angle less than $C\varrho^{n}$, or
\item [(b)] at a point at distance less than $C\varrho^{n}$ from one
of the endpoints of one of the segments $p\circ \pi (\mathcal{F}_{x})$
or $p\circ \pi (\mathcal{F}_{y})$.
\end{itemize}
\end{enumerate}
\end{lemma}

\begin{proof}
Only statement (A3) is nontrivial. Since the projection $\pi
:\Sigma_{F} \rightarrow S\Upsilon$ is H\"{o}lder continuous, it
suffices to prove that for all large $n$, if two geodesic segments
intersect at an angle larger than $C\varrho^{n}$ and at a point at
distance greater than $C\varrho^{n}$ from any of the endpoints, then
so will two geodesic segments of the same lengths whose initial
points/directions are within distance $C\varrho^{2n}$ of the initial
points/directions of the original pair of geodesic segments. This
holds because at small distances, geodesic segments on $\Upsilon$ look
like straight line segments in the tangent space (to the intersection
point).
\end{proof}

A formula similar to \eqref{eq:u-representation} holds for the
self-intersection count $N (t)=N (t;\gamma)$ of an arbitrary geodesic
segment $\gamma [0,t]$ (which, in general, will not be a closed
curve.) As in the case of closed geodesics, the self-intersection
count can be computed by partitioning the segment into the images of
successive fibers and counting how many pairs intersect. However, for
arbitrary geodesic segments, the first and last segment will only
represent partial fibers, and so intersections with these must be
counted accordingly.

For $x,y\in \Sigma$ and $0\leq s<F (x) $, define $g_{0} (s,x,y)$ to be
$1$ if the geodesic segment corresponding to the fiber
$\mathcal{F}_{y}$ intersects the segment corresponding to the partial fiber
\[
	\{(x,t)\,:\, 0\leq t<s\},
\]
and $0$ if not. Similarly, define $g_{1} (s,x,y)$ to be $1$ if the
geodesic  segment 
corresponding to the fiber $\mathcal{F}_{y}$ intersects the segment
corresponding to the partial fiber
\[
	\{(x,t)\,:\, s\leq t<F (x)\}
\]
 and $0$ if not.   If $\gamma$ is the geodesic ray whose initial
 tangent vector is $p \circ \pi (x,s)$, then the self-intersection
 count for the geodesic segment $\gamma [0,t]$ is given by 
\begin{equation}\label{eq:self-intersection-Count-random}
	N (t;\gamma) = \frac{1}{2} \sum_{i=0}^{\tau} \sum_{j=1}^{\tau} h
	(\sigma^{i}x,\sigma^{j}x) - \sum_{i=1}^{\tau}g_{0}
	(s,x,\sigma^{i}x) -\sum_{i=0}^{\tau}g_{1}
	(S_{\tau +1}F (x) -t+s,\sigma^{\tau}x,\sigma^{i}x) \pm \text{error}
\end{equation}
where 
\begin{equation}\label{eq:first-passage-time-def} 
	\tau =\tau_{t} (x)=\min \{n \geq 0\,:\, S_{n+1}F (x)\geq t \} .
\end{equation}
The error term accounts for possible intersections between the initial
and final segments, and hence is bounded in magnitude by $2$.  Because
it is bounded, it has no effect on the distribution of $( N (t)-\kappa
t^{2})/t$ in the large$-t$ limit. Note that whereas the first double
sum in \eqref{eq:self-intersection-Count-random} will have magnitude
$O (t^{2})$ for large $t$, each of the single sums will have magnitude
$O (t)$. Since the order of magnitude of the fluctuations in
\eqref{eq:random} is $O (T)$, it follows that the single sums in
\eqref{eq:self-intersection-Count-random} will have an appreciable
effect on the  distribution of the normalized self-intersection counts
in \eqref{eq:random}.

The following proposition summarizes the key features of the
construction.

\begin{proposition}\label{proposition:symbD-per}
For any compact surface of constant curvature $-1$ there exists a
topologically mixing shift $(\Sigma ,\sigma)$ of finite type, a
H\"{o}lder continuous height function $F:\Sigma \rightarrow
\zz{R}_{+}$,  and functions $h_{m}:\Sigma \times \Sigma \rightarrow
[0,1]$ satisfying (A1)--(A3) of Lemma~\ref{lemma:h-decomposition}
such
that
\begin{enumerate}
\item [(SD-1)] with only finitely many exceptions, each prime closed geodesic
corresponds uniquely to  a necklace; 
\item [(SD-2)] the length of each such closed geodesic is $S_{n}F
(x)$, where $n$ is the minimal period of the necklace $x$; and
\item [(SD-3)] the number of self-intersections of any such closed geodesic
is given by \eqref{eq:u-representation}, with $h$ defined by
\eqref{eq:h-decomposition}.
\end{enumerate}
(Here a \emph{necklace} is an equivalence class of periodic sequences,
where two periodic sequences are considered equivalent if one is a
shift of the other.)  Furthermore, there exists a semi-conjugacy of
the symbolic flow $\phi_{t}$ on $\Sigma_{F}$ with the geodesic flow
that is one-to-one except on a set of first category. For any geodesic
$\gamma$, the number $N (t;\gamma)$ of self-intersections of the
segment $\gamma [0,t]$ is given by
\eqref{eq:self-intersection-Count-random}, with the error bounded by
$2$.
\end{proposition}

\begin{remark}\label{remark:crossSection}
The choice of the Poincar\'{e} section in the construction of the
suspension flow is important because it determines the locations of
the discontinuities of the function $h$ in the representation
\eqref{eq:u-representation}, which in turn determines how well $h
(x,y)$ can be represented by functions that depend on only finitely
many coordinates of $x$ and $y$ (cf. property (A3) in
Lemma~\ref{lemma:h-decomposition}). This choice is somewhat arbitrary;
other Poincar\'{e} sections can be obtained in a number of ways, the
simplest of which is by moving points of the original section forward
a distance $r$ along the flow lines (equivalently, replacing the
semi-conjugacy $\pi$ of Proposition~\ref{proposition:bowen} by
$\pi_{r}=\pi \circ \phi_{r}$). This has the effect of changing the
function $h$ as follows: define $h^{r}:\Sigma \times \Sigma
\rightarrow \{0,1 \}$ by setting $h^{r} (x,y)=1$ if the $p\circ \pi -$
projections of the suspension flow segments
\begin{equation}\label{eq:f-epsilon}
	\mathcal{F}^{r}_{x}:=\{\phi_{s} (x) \}_{r
	\leq s<F (x)+r} \quad \text{and} \quad
	\mathcal{F}^{r}_{y}:=\{\phi_{s} (y) 
	\}_{r \leq s<F (y)+r}
\end{equation}
intersect transversally on $\Upsilon$, and $h^{r} (x,y)=0$
if not. Clearly, the representation \eqref{eq:u-representation}
for the self-intersection counts of closed geodesics remains valid
with $h$ replaced by any $h^{r}$. Similarly, the representation
\eqref{eq:random} for the self-intersection count of an arbitrary
geodesic segment will  hold when $h$, $g_{0}$, and $g_{1}$ are  replaced by
$h^{r}$, $g^{r}_{0}$, and $g^{r}_{1}$,
where $g^{r}_{i}$ are the obvious modifications of $g_{i}$.
The function $h^{r}$ can be decomposed as
\begin{equation}\label{eq:h-epsilon}
	h^{r}=\sum_{m=0}^{\infty}h^{r}_{m} +h^{r}_{\infty}
\end{equation}
where $h^{r}_{m} (x,y)$ and $h_{\infty}^{r} (x,y)$ are
defined in analogous fashion to the functions $h_{m}$ and $h_{\infty}$ above,
in particular, $h^{r}_{m} (x,y)=1$ if $m$ is the smallest
positive integer such that $h^{r} (x',y')=1$ for all pairs
$x',y'$ that agree with $x,y$ in coordinates $|j|\leq m$, and
$h^{r}_{m} (x,y)=0$ otherwise. These functions once again
satisfy (A1)-(A2) of Lemma~\ref{lemma:h-decomposition}. Property (A3)
is replaced by the following (A3)': if $\sum_{n\geq m}h^{\varepsilon
}_{n} (x,y)+h_{\infty}^{\varepsilon } (x,y)\not =0$ then the geodesic segments
corresponding to the suspension flow segments
$\mathcal{F}^{r}_{x}$ and $\mathcal{F}^{r}_{y}$
intersect either
\begin{itemize}
\item [(a)]  at an angle less than $C\varrho^{m}$, or
\item [(b)] at a point at distance $<C\varrho^{m}$ from one of the
endpoints of $p \circ \pi (\mathcal{F}^{r}_{x})$ or $p\circ
\pi (\mathcal{F}^{r}_{y})$. 
\end{itemize}
\end{remark}

\section{Gibbs states and thermodynamic formalism}\label{sec:gibbs}

\subsection{Standing Conventions}\label{ssec:conventions} We shall
adhere (mostly) to the notation and terminology of Bowen
\cite{bowen:book}. However, we shall suppress the dependence of various
objects on the transition matrix $A$ of the underlying subshift of
finite type, since this will be fixed throughout the paper: thus, the
spaces of one-sided and two-sided sequences will be denoted by
$\Sigma^{+}$ and $\Sigma$, respectively, and the spaces of
$\alpha-$H\"{o}lder continuous real-valued functions on these sequence
spaces by $\mathcal{F}^{+}$ and $\mathcal{F}$.  (With
one exception [sec.~\ref{ssec:gibbs}] the H\"{o}lder exponent $\alpha$
will also be fixed throughout the paper, so henceforth we shall refer
to $\alpha-$H\"{o}lder functions as H\"{o}lder functions.)
The spaces $\mathcal{F}=\mathcal{F}_{\alpha}$ and
$\mathcal{F}^{+}=\mathcal{F}^{+}_{\alpha}$ are Banach spaces with
norm
\begin{gather*}
	\xnorm{f}=\xnorm{f}_{\alpha}=|f|_{\alpha}+\xnorm{f}_{\infty}
		\quad \text{where} \\
	|f|_{\alpha}=\sup_{n\geq 0} \sup_{x,y\,:\, x_{j}=y_{j}
	\;\forall \,|j|\leq n} |f (x)-f (y)|/\alpha^{n}.
\end{gather*}

For any sequence $x \in \Sigma$ or $x \in \Sigma^{+}$ and any interval
$J=\{k,k+1 ,\dotsc ,l \}$ of $\zz{Z}$ or $\zz{N}$ denote by $x_{J}$ or
$x(J)$ the subsequence $x_{k}x_{k+1}\dotsb x_{l}$, and let
$\Sigma_{J}(x)$ (or $\Sigma^{+}_{J} (x)$) be the cylinder set
consisting of all sequences $y\in \Sigma$ such that $y(J)=x(J)$.  For
any $n\in \zz{N}$ let $[n]=\{1,2,\dotsc ,n \}$. For any continuous,
real-valued function $\varphi$ and any probability measure $\lambda$
on $\Sigma$ or $\Sigma^{+}$ denote by $E_{\lambda}\varphi =\int
\varphi \,d\lambda$ the expectation of $\varphi$ with respect to
$\lambda$, by $\text{Pr} (\varphi)$ the topological pressure of
$\varphi$ (cf. \cite{bowen:book}, Lemma 1.20 and Sec. 2C), and for each
interval $J\subset \zz{Z}$ write
\[
	S_{J}\varphi =\sum_{i\in J}\varphi \circ \sigma^{i}.
\]
Following Bowen we shall also write $S_{n}\varphi =S_{[n]}\varphi
\circ \sigma^{-1}=\sum_{i=0}^{n-1}\varphi \circ \sigma^{i}$ for any integer
$n\geq 1$.

\subsection{Gibbs states}\label{ssec:gibbs} For each real-valued
function $\varphi \in \mathcal{F}$ there is a unique \emph{Gibbs
state} $\mu_{\varphi }$, which is by definition a shift-invariant
probability measure $\mu_{\varphi}$ on $\Sigma $ for which there are constants
$0<C_{1}<C_{2}<\infty $ such that for every finite interval $J\subset
\zz{Z}$,
\begin{equation}\label{eq:gibbs}
	C_{1}\leq \frac{\mu_{\varphi} (\Sigma _{J}(x))}{\exp\{S_{J}\varphi
	(x)-|J| \text{\rm Pr} (\varphi )\}} \leq C_{2} 
\end{equation}
for all $x \in \Sigma $.  When the  potential function $\varphi$ is
fixed, we shall delete the subscript $\varphi$ and write $\mu$ in
place of $\mu_{\varphi}$.

If the underlying shift $(\Sigma ,\sigma)$ is topologically mixing --
as we shall assume throughout -- every Gibbs state $\mu
=\mu_{\varphi}$ is mixing (and therefore ergodic), and has positive
entropy. Moreover, there exists $\alpha =\alpha_{\varphi} <1$ such
that for every $n\geq 1$, all cylinder sets $\Sigma_{[0,n]} (x)$ of
generation $n$ have $\mu -$probabilities less than $\alpha^{n}$. In
addition, correlations decay exponentially, in the following
sense. For any subset $J\subset \zz{Z}$, let $\mathcal{B}_{J}$ be the
$\sigma -$algebra of Borel subsets $G$ of $\Sigma$ whose indicator
functions $\mathbf{1}_{G}$ depend only on the coordinates $n\in
J$. Then for each Gibbs state $\mu=\mu_{\varphi }$ there exist
constants $C=C_{\varphi }<\infty$ and $0<\beta=\beta_{\varphi} <1$
such that for each $n\geq 1$,
\begin{equation}\label{eq:expMixing}
	|\mu (G\cap G') -\mu (G)\mu (G')|\leq C \beta^{n}\mu (G)\mu
	(G') 
	\quad\forall \; G\in \mathcal{B}_{(-\infty ,0]},
	\; \; G'\in \mathcal{B}_{[n,\infty )}.
\end{equation}
This implies that for any specification of the ``past'' $ \dots
\omega_{-1}\omega_{0}$, the conditional distribution of the ``future''
$\omega_{n}\omega_{n+1}\dotsb$ differs from the unconditional
distribution by at most $2C\beta^{n}$ in total variation norm. The
exponential mixing property can be expressed in the following
equivalent form (see \cite{parry-pollicott:asterisque},
pp. 29--30): for any two $\alpha -$H\"{o}lder functions $v,w$ such
that $E_{\mu}w=0$,
\begin{equation}\label{eq:expMixingFunctionalForm}
	|E_{\mu} v (w\circ \sigma^{n})|\leq C
	\beta^{n}\xnorm{v}_{\infty}\xnorm{w}_{\alpha} 
\end{equation}
where $\xnorm{w}_{\alpha}$ is the H\"{o}lder norm of $w$.

The uniform mixing property \eqref{eq:expMixing} implies that it is 
unlikely that a random sequence $x\in \Sigma$ chosen according to the
law of a Gibbs state will have long repeating strings at fixed
locations. This is made precise in the next lemma; it will be used in
section~\ref{sec:verification} below (cf. Lemma~\ref{lemma:angle}) to
show that self-intersections at very small angles are highly unlikely.

\begin{lemma}\label{lemma:noRepeats}
For any Gibbs state $\mu$ there exists $0<\beta <1$ such that for all
$k\not =0$ and all sufficiently large $m\geq 1$,
\begin{equation}\label{eq:noRepeats}
	\mu \{x\in \Sigma \,:\, x_{i}=x_{i+k} \;\;\text{for all}\;\;
	0\leq i\leq m\} \leq \beta^{m}
\end{equation}
\end{lemma}

\begin{proof}
The mixing inequality \eqref{eq:expMixing} and $\sigma -$invariance of $\mu$
imply that there exist an integer $L\geq 1$ and $0<\beta <1$ such that
for every $L'\geq L$ and every symbol $a\in \mathcal{A}$ of the
underlying alphabet,
\[
	\mu (x_{L'+n}=a\,|\, \mathcal{B}_{(-\infty ,n]})\leq \beta
	\quad \text{for all} \;\; n\in \zz{Z}
\]
We shall consider two separate cases: first, $k\geq L$; and second, $1\leq
k< L$. (Since every Gibbs state is $\sigma -$invariant, it suffices
to consider only positive values of $k$.) In the first case,
\begin{align*}
	\mu (x_{i}=x_{k+i} \;\; \forall \; 1\leq i\leq nL)
	&\leq \mu (x_{iL}=x_{k+iL} \;\; \forall \; 1\leq i\leq n)\\
	& =E_{\mu} \prod_{i=1}^{n} \mu (x_{iL}=x_{k+iL} \,|\,
	\mathcal{B}_{(-\infty , k+ (i-1)L]}) \\
	&\leq \beta^{n}.
\end{align*}
In the  case $1\leq k< L$, there will be a multiple of $k$ in every
interval of length $L$, so we  can choose $m_{1}<m_{2}<\dotsb <m_{n}$
such that $jL\leq m_{j}k<jL+L$ for each $j\leq n$. Now the requirement
that $x_{i}=x_{i+k}$ for all $1\leq i\leq nL$ forces
$x_{0}=x_{m_{j}k}$ for every $j\leq n$; consequently,
\begin{align*}
	\mu (x_{i}=x_{k+i} \;\; \forall \; 1\leq i\leq nL)
	&\leq \mu (x_{0}=x_{m_{j}k} \;\; \forall\; 1\leq j\leq n) \\
	&=E_{\mu}\prod_{j=1}^{[n/2]} \mu (x_{m_{2j}k}=x_{0} \,|\,
	\mathcal{B}_{(-\infty ,m_{2j-2}k]}) \\
	&\leq \beta^{[n/2]}.
\end{align*}

\end{proof}

Two functions $\varphi ,\psi\in \mathcal{F}$ are said to be
\emph{cohomologous} if their difference is a cocycle $u-u\circ
\sigma$, with $u \in \mathcal{F}$. If $\varphi$ and $\psi$ are
cohomologous then $\mu_{\varphi}=\mu_{\psi}$ and
$\text{Pr}(\varphi)=\text{Pr} (\psi)$. According to a theorem of
Livsic (\cite{bowen:book}, Lemma 1.6), for every $\alpha -$H\"{o}lder
function $\varphi$ there exist $\sqrt{\alpha}-$H\"{o}lder functions
$\varphi^{+},\varphi^{-}$ both cohomologous to $\varphi$ (and
therefore mutually cohomologous) such that $\varphi^{+} (x)$ depends
only on the forward coordinates $x_{1},x_{2},\dotsc$ of $x$ and
$\varphi^{-} (x)$ depends only on the backward coordinates
$x_{0},x_{-1},\dotsc$.

For any function $\varphi \in \mathcal{F}^{+}$, the Gibbs state
$\mu_{\varphi }$ is related to the Perron-Frobenius eigenfunction
$h_{\varphi }$ and eigenmeasure $\nu_{\varphi }$ of the Ruelle
operator $\mathcal{L}_{\varphi}:\mathcal{F}^{+} \rightarrow
\mathcal{F}^{+}$ associated with $\varphi $ (cf. \cite{bowen:book}, ch.~1,
sec.~C). In particular, if $h_{\varphi }$ and $\nu_{\varphi }$ are
normalized so that $\nu_{\varphi}$ and $h_{\varphi}\nu_{\varphi }$
both have total mass $1$, and if $\lambda_{\varphi}$ is the
Perron-Frobenius eigenvalue, then
\begin{equation}\label{eq:gibbsStateDecomp}
	d\mu_{\varphi }= h_{\varphi }d\nu_{\varphi }
	\quad \text{and} \quad \lambda_{\varphi}=\exp \{ \text{Pr}
	(\varphi )\}. 
\end{equation}

\subsection{Suspensions}\label{ssec:suspensions}
Say that a function $f\in \mathcal{F}$ (or $\mathcal{F}^{+}$) is
\emph{nonarithmetic} if there is no function $g\in \mathcal{F}$ valued
in a discrete additive subgroup of $\zz{R}$ to which $f$ is
cohomologous. If $F\in \mathcal{F}$ is strictly positive then the
suspension flow with height function $F$ is topologically mixing if
and only if $F$ is nonarithmetic (see, for instance,
\cite{parry-pollicott}). This is the case, in particular, for
the suspension flow discussed in section~\ref{sec:symbolicDynamics}. 

Assume henceforth that $F$ is a strictly positive, nonarithmetic,
H\"{o}lder-continuous function on $\Sigma$, and let $\Sigma_{F} $ be
the corresponding suspension space. For each $\sigma -$invariant
probability measure $\mu$ on $\Sigma$ define the \emph{suspension} of
$\mu$ relative to $F$ to be the flow-invariant probability measure
$\mu^{*}$ on $\Sigma_{F}$ with cylinder probabilities
\begin{equation}\label{eq:suspensionMeasures}
	\mu^{*} (\Sigma_{[n]}(x)\times [0,a])=\frac{a \mu
	(\Sigma_{[n]} (x))}{\int_{\Sigma}F\,d\mu} 
\end{equation}
for any cylinder set $\Sigma_{[n]} (x)$ and any $a\geq 0$ such that
$a\leq F$ on $\Sigma_{[n]} (x)$. (Here and elsewhere we use the
notation $[n]$ to denote the set of integers $\{1,2,...,n\}$.)
For the geodesic flow on a compact, negatively curved surface, both
the Liouville measure and the maximum entropy measure lift to the
suspensions of Gibbs states; for the maximum entropy measure, the
corresponding Gibbs state is $\mu_{-\theta F}$ where $\theta >0$ is
the unique real number such that $\text{Pr} (-\theta F)=0$, and this
value of $\theta$ is the topological entropy of the flow
(cf. \cite{abramov}, also  \cite{lalley:acta}). If the
surface has constant negative curvature then the Liouville and maximum
entropy measures are the same, but if the surface has \emph{variable}
negative curvature then the Liouville measure is mutually singular
with the maximum entropy measure, and so the potential function for
the corresponding Gibbs state is \emph{not} cohomologous to $-\delta
F$.  This is what accounts for the difference between constant and
variable negative curvature in Theorem~\ref{theorem:closed}.

If the suspension flow is topologically mixing then the suspension of
any Gibbs state is mixing for the flow. This fact is equivalent to a
\emph{renewal theorem}, which can be formulated as follows. For each
$T\in \zz{R}_{+}$ and $x \in \Sigma$ define
\begin{equation}\label{eq:firstPassageTime}
	\tau (x) =\tau_{T} (x)=\min \{n\geq 1 \,:\, S_{n}F (x)> T\} \quad
	\text{and} \quad 
	R_{T} (x)=S_{\tau (x)}F (x)-T.
\end{equation}

\begin{proposition}\label{proposition:renewal}
Assume that the shift $(\Sigma ,\sigma)$ is topologically mixing, and
that  $F\in \mathcal{F}$ is positive and nonarithmetic. Then for any
Gibbs state $\mu$ and all bounded, continuous functions 
$f,g:\Sigma  \rightarrow \zz{R}$ and $h:\zz{R \rightarrow \zz{R}}$,
\begin{equation}\label{eq:renewal}
	\lim_{ T \rightarrow \infty} \int_{\Sigma} f (x) g (\sigma^{\tau (x)} (x))
	 h (R_{T} (x)) \,d\mu (x) 
	= \int f (x)\,d\mu (x)\times \int_{\Sigma_{F}} g
	(x)h (t) \,d\mu^{*} (x,t) 
\end{equation}
where $\mu^{*}$ is the suspension of $\mu$.
\end{proposition}

The special case where $f\equiv g \equiv 1$ is of particular interest,
because it yields estimates of the probability that $R_{T}$ falls in
an interval. In particular, it implies that there exists $C<\infty$
such that for all $\varepsilon >0$, all $a\geq 0$, and all sufficiently large $T$
(i.e., all $T>t_{\varepsilon}$),
\begin{equation}\label{eq:renewalEstimate}
	\mu\{x\in \Sigma \,:\, a\leq R_{T} (x)\leq
	a+\varepsilon\}\leq C\varepsilon  .
\end{equation}  

We will say that two (or more) weakly convergent sequences $X_{T},Y_{T}$ of
random variables, vectors, or sequences are \emph{asymptotically independent}  as
$T \rightarrow \infty$ if the joint distribution of $(X_{T},Y_{T})$
converges weakly to the product of the limit distributions of $X_{T}$
and $Y_{T}$, that is, if for all bounded, continuous real-valued
functions $u,v$, 
\[
	\lim_{T \rightarrow \infty} Eu (X_{T})v (Y_{T})=
	(\lim_{T \rightarrow \infty}Eu (X_{T}))(\lim_{T \rightarrow
	\infty}Ev (Y_{T})) .
\]
In this terminology, Proposition~\ref{proposition:renewal} asserts that the
``overshoot'' random variable $R_{T} (x)$ is asymptotically
independent of the state variables $x$ and $\sigma^{\tau (x)} (x)$.

If $x\in \Sigma$ is chosen randomly according to an ergodic, shift-invariant
probability measure $\mu $ then $\tau_{T} (x)$ will be
random. However, when $T$ is large the random variable is to first
order of approximation ``predictable'' in that the error in the approximation
$\tau_{T}\approx T/E_{\mu}F$ is of order $O_{P} (1/\sqrt{T})$. More precisely:

\begin{proposition}\label{proposition:clt}
Assume that $F:\Sigma  \rightarrow \zz{R}$ and $g:\Sigma
\rightarrow \zz{R}^{k}$ are H\"{o}lder continuous
functions, with $F>0$, and let $\mu$ be a Gibbs state for the shift
$(\Sigma ,\sigma)$. Then there exist a constant
$b>0$ (depending on $\mu$ and $F$) and a $k\times k$ positive
semi-definite matrix $\mathbf{M}$ (depending on $\mu$, $F$, and $g$)
such that as $T \rightarrow \infty$,
\begin{gather}\label{eq:CLT-tau}
	\frac{\tau_{T}-T/E_{\mu}F}{b\sqrt{T}} \Longrightarrow
	\text{Normal}  (0,1) \quad
	\text{and}\\
\label{eq:CLT-g}
	\frac{S_{\tau_{T}}g-TE_{\mu}g/E_{\mu}F}{\sqrt{T}}
	\Longrightarrow  \text{Normal}_{k} (\mathbf{0},\mathbf{M}).
\end{gather}
Moreover, the limiting covariance matrix $\mathbf{M}$ is
\emph{strictly} positive definite unless some linear combination of
the coordinate functions $g_{i}$ is cohomologous to $F+c$ for some
constant $c$. Finally, the random vector $(S_{\tau}g
-TE_{\mu}g/ E_{\mu}F)/T^{1/2}$ and the random variable
$(\tau_{T}-T/E_{\mu}F)/T^{1/2}$ are asymptotically independent
of the overshoot $R_{T} (x)= S_{\tau (x)}F (x)-T$ and the state
variables $x$ and $\sigma^{\tau (x)}x$.
\end{proposition}

Both \eqref{eq:CLT-tau} and \eqref{eq:CLT-g} are 
consequences of Ratner's \cite{ratner:clt} central limit theorem. (See
in particular the proof of Theorem~2.1 in \cite{ratner:clt}. The
vector-valued central limit theorem follows from the scalar central
limit theorem by the so-called \emph{Cramer-Wold device} -- see, for
instance, \cite{billingsley:wc}, ch.~1.)  The last assertion (regarding
asymptotic independence) can be proved by standard methods in renewal
theory (see for instance \cite{siegmund}); roughly speaking, it
holds because the values of the random variables $R_{T} (x)$ and
$\sigma^{\tau (x)}x$ are mainly determined by the last $O (1)$ steps
before time $\tau (x)$, whereas the values of $(S_{\tau}g
-T^{2}E_{\mu}g/ (E_{\mu}F)^{2})/T^{3/2}$ and other ``bulk'' random
variables are mostly determined long before time $\tau (x)$.

\section{U-statistics}\label{sec:u-statistics}

\subsection{$U-$statistics with random limits of
summation}\label{ssec:randLimits} Let $(\Sigma ,\sigma)$ be a
two-sided shift of finite type and $F:\Sigma \rightarrow (0,\infty)$ a
H\"{o}lder-continuous function. \emph{Assume that $F$ is
nonarithmetic:} this ensures that the conclusions of
Propositions~\ref{proposition:renewal} and \ref{proposition:clt} are
valid. As in section~\ref{ssec:suspensions}, define $\tau
=\tau_{T}:\Sigma \rightarrow \zz{Z}_{+}$ to be the first passage time
to the level $T>0$ by the sequence $S_{n}F$ (see equation
\eqref{eq:firstPassageTime}). Let $h:\Sigma \times \Sigma \rightarrow
\zz{R}$ be a symmetric, Borel measurable function. Our interest in
this section is the distribution of the random variable
\begin{equation}\label{eq:u-statistic}
	U_{T} (x):=\sum_{i=1}^{\tau (x)}\sum_{j=1}^{\tau (x)} h
	(\sigma^{i}x,\sigma^{j}x), \quad \text{for} \;\; x\in \Sigma,
\end{equation}
under a Gibbs state $\mu$ or, more generally, under a
probability measure that is absolutely continuous with respect to a
Gibbs state.  Random variables of this form --- but with the random
time $\tau (x)$ replaced by a constant $n$ --- are known in
probability theory as $U-$\emph{statistics}, and have a well-developed
limit theory (cf. \cite{hoeffding},
\cite{denker-keller}). Unfortunately, the standard results of this
literature do not apply here, for three reasons: (a) because here the
limits of summation in \eqref{eq:u-statistic} are themselves random
variables, (b) because no continuity requirements have been imposed on
the function $h$, and (c) because of the peculiar nature of the dependence in
the sequence $\{\sigma^{n}x\}_{n\in \zz{Z}}$.

\subsection{Convergence in law under Gibbs
states}\label{ssec:convergenceGibbs}
Fix a probability measure $\lambda$ on $\Sigma$.

\begin{hypothesis}\label{hypothesis:short-memory}
The kernel $h$ admits a decomposition $h=\sum_{m=1}^{\infty}h_{m}$
such that 
\begin{itemize}
\item [(H0)] each $h_{m}:\Sigma \rightarrow \zz{R}$ is a symmetric
function of its arguments;
\item [(H1)] there exists $C<\infty$ such that $\sum_{m\geq
1}|h_{m}|\leq C$ on $\Sigma \times \Sigma $;
\item [(H2)]  $h_{m} (x,y)$ depends only on the
coordinates $x_{j},y_{j}$ such that  $|j|\leq m$; and
\item [(H3)] there exist $C<\infty$ and $0<\beta <1$ such that for
all $m\geq 1$ and $j\in \zz{Z}$,
\begin{equation}\label{eq:H4}
	\int_{\Sigma} |h_{m} (x,\sigma^{j}x)|\, d\lambda (x) \leq C\beta^{m}.
\end{equation}
\end{itemize}
\end{hypothesis}

\begin{definition}\label{definition:Hoeffding}
For any bounded, symmetric, measurable function $h:\Sigma \times \Sigma
\rightarrow \zz{R}$ and any Borel probability measure $\lambda$ on
$\Sigma$ define the \emph{Hoeffding projection} $h_{+}:\Sigma
\rightarrow \zz{R}$ of $u$ relative to $\lambda$ by
\[
	h_{+} (x) =\int_{\Sigma} h (x,y) \, d\lambda (y).
\]
Say that the kernel $h$ is \emph{centered} relative to $\lambda$ if
its Hoeffding projection is identically $0$.
\end{definition}

Our interest is in the large-$T$ limiting behavior of the random
variable $U_{T}$ defined by \eqref{eq:u-statistic} (more precisely,
its distribution) under a Gibbs state  or a probability
measure that is absolutely continuous with respect to a
Gibbs state. Observe that if $\mu $ is a Gibbs state and if $h$
satisfies Hypothesis~\ref{hypothesis:short-memory} relative to $\mu$
then the corresponding Hoeffding projection $h_{+} (x)$ is H\"{o}lder
continuous on $\Sigma$, even though $h$ itself might not be
continuous. The following theorem will show that under
Hypothesis~\ref{hypothesis:short-memory} two types of limit behavior
are possible, depending on whether or not $h_{+}$ is cohomologous to a
scalar multiple $aF$ of $F$.   Set
\begin{equation}\label{eq:tau-tilde}
	\tilde{\tau}_{T}=\frac{\tau -T/E_{\mu}F}{\sqrt{T}}.
\end{equation}

\begin{theorem}\label{theorem:u-statistic}
Let $\mu =\mu_{\varphi }$ be a Gibbs state, and let $h:\Sigma \times
\Sigma \rightarrow \zz{R}$ be a function that satisfies Hypothesis
\ref{hypothesis:short-memory}, for $\lambda =\mu$.  If the Hoeffding
projection $h_{+}$ of $h$ relative to $\mu$ is cohomologous to $aF$
for some scalar $a\in \zz{R}$, then as $T \rightarrow \infty$,
\begin{equation}\label{eq:caseA}
	\tilde{U}_{T}=\frac{U_{T}- (a/E_{\mu}{F}) T^{2}}{T}
	\stackrel{\mathcal{D}}{\longrightarrow} G
\end{equation}
for some probability distribution $G$ on $\zz{R}$. Otherwise,
\begin{equation}\label{eq:caseB}
	\tilde{U}_{T}=\frac{U_{T}- (E_{\mu}h_{+}/E_{\mu}{F}^{2})
	T^{2}}{T^{3/2}} \stackrel{\mathcal{D}}{\longrightarrow} \text{Gaussian}
\end{equation}
for a proper Gaussian distribution on $\zz{R}$. Furthermore, in either
case the random vector $(\tilde{U}_{T},\tilde{\tau}_{T})$,
the state variables $x,\sigma^{\tau (x)}$, and the overshoot random
variable $R_{T}$ are asymptotically independent as $T \rightarrow
\infty$.
\end{theorem}

\begin{proof}
[Proof strategy] The logic of the proof is as follows. First, we will
show that the theorem is true for centered kernels $h$ such that $h
(x,y)$ depends only on finitely many coordinates of the arguments
$x,y$. This will use Proposition~\ref{proposition:clt}. Second, we
will prove by an approximation argument that the truth of the theorem
for centered kernels can be deduced from the special case of centered
kernels that depend on only finitely many coordinates. This step will
use moment estimates that depend on
Hypothesis~\ref{hypothesis:short-memory} (in particular, on the
critical assumption (H3)). Third, we will show that to prove the
theorem in the general case it suffices to consider the case where the
kernel $h$ is centered. For ease of exposition, we will present the
third step before the second; however, this step will rely on the
other two.

Observe that the validity of \eqref{eq:caseA}--\eqref{eq:caseB} is not
affected by rescaling of either $T$ or $h$. Consequently, there is no
loss of generality in assuming that $E_{\mu}F=1$ and that the
constants $C,C'$ in Hypothesis~\ref{hypothesis:short-memory} are
$C=C'=1$. 
\end{proof}

\begin{proof}
[Step 1] Assume first that $h $ is centered (thus, $h_{+}$ is
cohomologous to $aF$ with $a=0$, and so case \eqref{eq:caseA} of
Theorem~\ref{theorem:u-statistic} applies), and that $h(x,y)$ depends
only on the coordinates $x_{1}x_{2}\dotsb x_{m}$ and
$y_{1}y_{2},,,y_{m}$.  Then the function $h$ assumes only finitely
many different values, and these are given by a symmetric, real,
square matrix $h (\xi,\zeta)$, where $\xi$ and $\zeta$ range over the
set $\Sigma_{m}$ of all length-$m$ words occurring in infinite
sequences $x\in \Sigma$. This matrix induces a real, Hermitian
operator $L$ on the finite-dimensional subspace of $L^{2}
(\Sigma,\mu)$ consisting of functions that depend only on the
coordinates $x_{1}x_{2}\dotsb x_{m}$. Let $D_{m}$ be the dimension of
this subspace. Because $h$ is centered, the operator $L$ contains the
constants in its null space. Consequently, all other eigenfunctions
$\varphi_{k}$ are orthogonal to the constant function $1$, and thus,
in particular, have mean $0$. It follows by the spectral theorem for
symmetric matrices that the $U-$statistic \eqref{eq:u-statistic} can
be written as
\[
	U_{T} (x) =\sum_{i=1}^{\tau (x)} \sum_{j=1}^{\tau (x)}
	      \sum_{k=2}^{D_{m}}\lambda_{k}\varphi_{k}
	      (\sigma^{i}x)\varphi_{k} (\sigma^{j}x)
	      =\sum_{k=2}^{D_{m}}\lambda_{k}\left(\sum_{i=1}^{\tau
	      (x)}\varphi_{k} (\sigma^{i}x) \right)^{2} .
\]
Therefore, Proposition~\ref{proposition:clt}
implies that as $T \rightarrow \infty$,
\begin{equation}\label{eq:vclt}
	\left(\tilde{\tau}_{T},
\left(\frac{1}{\sqrt{T}}\sum_{i=1}^{\tau} \varphi_{k}\circ \sigma^{i}
\right)_{2\leq k\leq D (m)}\right) \stackrel{\mathcal{D}}{\longrightarrow}
\text{N} (\mathbf{0},\mathbf{A})
\end{equation}
for some possibly degenerate $( D_{m} -1)-$dimensional multivariate
normal distribution $N (0,A)$. The convergence \eqref{eq:caseA}
follows, with $a=0$ and $G$ the distribution of the quadratic form of
the multivariate normal.  Proposition~\ref{proposition:clt} also
implies that the random vector \eqref{eq:vclt} is asymptotically
independent of $x,\sigma^{\tau (x)}x$, and $R_{T} (x)$; consequently,
so is the random vector with components $U_{T}/T$ and
$\tilde{\tau}_{T}$.
\end{proof}

\begin{proof}
[Step 3] Assume that the result is true for all
\emph{centered} kernels. We will show that the theorem  then holds for any
\emph{non-}centered kernel satisfying
Hypothesis~\ref{hypothesis:short-memory}.
Recall that if $h$ satisfies Hypothesis~\ref{hypothesis:short-memory}
then its Hoeffding projection $h_{+}$ is H\"{o}lder continuous.
There are two cases to consider, according to whether or not
$h_{+}$ is cohomologous to a scalar multiple of $F$. Consider first
the case where $h_{+}$ is cohomologous to $aF$ for some $a\in \zz{R}$.
Thus, $E_{\mu}h_{+}=a$ (since $E_{\mu}F=1$), and so for some
coboundary $w-w\circ \sigma$, 
\begin{align*}
	h (x,y)&=h_{0} (x,y)+h_{+} (x)+h_{+} (y)-a\\
	  &=h_{0} (x,y)+aF (x)+aF (y)+w (x)-w (\sigma x)-w (y)-w (\sigma y)-a 
\end{align*}
where $h_{0} (x,y)$ is centered. This implies that
\begin{align*}
	U_{T}&=\sum_{i=1}^{\tau (x)} \sum_{j=1}^{\tau (x)}h (\sigma^{i}x,\sigma^{j}x)\\
	&=\sum_{i=1}^{\tau (x)} \sum_{j=1}^{\tau (x)}h_{0}
	(\sigma^{i}x,\sigma^{j}x) +2a\tau (x)S_{\tau (x)}F (x) -a\tau
	(x)^{2} +\tau (x) (w (x)-w (\sigma^{\tau (x)}x))\\
	&=\sum_{i=1}^{\tau (x)} \sum_{j=1}^{\tau (x)}h_{0}
	(\sigma^{i}x,\sigma^{j}x) +aT^{2}+a (\tau (x) -T)^{2}+a\tau
	(x)R_{T} (x)+\tau (x)  (w (x)-w (\sigma^{\tau (x)}x))\\
	&=:V_{T} +aT^{2}+a (\tau (x) -T)^{2}+a\tau
	(x)R_{T} (x)+\tau (x)  (w (x)-w (\sigma^{\tau (x)}x))
\end{align*}
where $V_{T}$ is the $U-$statistic \eqref{eq:u-statistic} with the
kernel $h$ replaced by the centered kernel $h_{0}$.  Now as $T
\rightarrow \infty$, $\tau (x)/T \rightarrow 1$ a.s., by the ergodic
theorem, and both $R_{T}$ and $\tilde{\tau}_{T}$ converge in
distribution, by Proposition~\ref{proposition:renewal} and
Proposition~\ref{proposition:clt}. Consequently, the convergence
\eqref{eq:caseA} and the joint asymptotic independence assertions hold
because by assumption they hold for centered kernels.

Next, consider the case where $h_{+}$ is \emph{not} cohomologous to a
scalar multiple of $F$; we must prove \eqref{eq:caseB}. As above, we
assume without loss of generality that $E_{\mu}F=1$. Let $h_{0}$ be
the centered kernel defined by
\[
	h (x,y)=h_{0} (x,y)+h_{+} (x)+h_{+} (y)-b
\]
where
$b=E_{\mu}h_{+}$; then
\begin{align*}
	U_{T}	&=\sum_{i=1}^{\tau (x)} \sum_{j=1}^{\tau (x)}h_{0}
	(\sigma^{i}x,\sigma^{j}x) +2\tau (x)S_{\tau (x)}h_{+} (x)-b\tau (x)^{2}\\
	&=\sum_{i=1}^{\tau (x)} \sum_{j=1}^{\tau (x)}h_{0}
	(\sigma^{i}x,\sigma^{j}x) +2\tau (x) ( S_{\tau (x)}h_{+}
	(x)-b\tau (x))+bT^{2} + b(\tau (x)^{2}-T^{2}).
\end{align*}
Now consider the effect of dividing this quantity by $T^{3/2}$. Since
the kernel $h_{0}$ is centered, the double sum divided by $T$
converges in distribution (by our hypothesis that the theorem is true
for centered kernels); hence, if it is divided by $T^{3/2}$ it
will converge to $0$. Thus, asymptotically as $T \rightarrow \infty$
the distribution of $(U_{T-bT^{2}})/T^{3/2}$ is determined by the
remaining terms $2\tau (S_{\tau}-b\tau)/T^{3/2}$ and $b
(\tau^{2}-T^{2})/T^{3/2}$. The ergodic theorem implies that $\tau /T
\rightarrow 1$, and the central limit theorem
(Proposition~\ref{proposition:clt}) implies that
$(S_{\tau}h_{+}-bT)/T^{1/2}$ and $\tilde{\tau}_{T}=(\tau -T)/T^{1/2}$ converge jointly
in distribution to a two-dimensional Gaussian distribution; consequently,
\[
	\frac{2\tau (x) ( S_{\tau (x)}h_{+}
	(x)-b\tau (x))+ b(\tau (x)^{2}-T^{2})}{T^{3/2}}
	\stackrel{\mathcal{D}}{\longrightarrow} \text{Gaussian}.
\]
This proves \eqref{eq:caseB}. The asymptotic independence assertions
follow directly from Proposition~\ref{proposition:clt}.
\end{proof}

\begin{proof}
[Step 2]
Assume, finally, that $h$ is a \emph{centered} kernel which satisfies
Hypothesis~\ref{hypothesis:short-memory}. Without loss of generality,
we can assume that the functions $h_{m}$ in the decomposition
$h=\sum_{m=1}^{\infty}h_{m}$ are themselves centered, because
replacing each $h_{m}(x,y)$ by $h_{m} (x,y)-h^{+}_{m} (x)-h^{+}_{m}
(y)$ does not change the validity of Hypothesis~\ref{hypothesis:short-memory}. 
Set
\[
	v_{m}=\sum_{k=1}^{m}h_{k} \quad \text{and} \quad 
	w_{m}=\sum_{k=m+1}^{\infty}h_{k}.
\]
Then each $v_{m}$ is centered, so by Step 1, the result is true if
$v_{m}$ is substituted for $h$ in the definition of $U_{T}$ (but of
course the limit distribution $G=G_{m}$ will depend on
$m$). Consequently, to prove the result for $h$ it suffices to show
that for any $\varepsilon >0$ there exists $m$ sufficiently large that
\begin{equation}\label{eq:OPe}
	\mu \left\{x\,:\,\bigg|\sum_{i=1}^{\tau (x)} \sum_{j=1}^{\tau (x)} w_{m}
	(\sigma^{i}x,\sigma^{j}x)\bigg| >\varepsilon T\right\}<\varepsilon 
\end{equation}
for all sufficiently large $T$.

Fix $0<\delta < 1/6$, and set
\begin{align*}
	n_{-}&=n_{-} (T)=\lfloor T-T^{1/2 +\delta} \rfloor \quad \text{and}\\
	n_{+}&=n_{+} (T)=\lfloor T+T^{1/2+\delta} \rfloor.
\end{align*}
By the central limit theorem (Proposition~\ref{proposition:clt}),
$n_{-}<\tau <n_{+}$ with $\mu -$probability approaching $1$ as $T
\rightarrow \infty$; thus $\tau (x)$ is essentially limited to one of
$T^{1/2+\delta}$ different possible values. Therefore, by the
Chebyshev inequality and a crude union bound, to establish
\eqref{eq:OPe} it suffices to prove the following.

\begin{lemma}\label{lemma:momentBound}
For each $\varepsilon >0$ there exists $m$ sufficiently large that for all
large $T$,
\begin{gather}\label{eq:momentBoundA}
	E_{\mu} \left( \sum_{i=1}^{n_{-}} \sum_{j=1}^{n_{-}} w_{m}
	(\sigma^{i}x,\sigma^{j}x)\right)^{2}<\varepsilon T^{2} \quad \text{and}\\
\label{eq:momentBoundB}
	\max_{n_{-}\leq n\leq n_{+}} E_{\mu} \left( \sum_{i=1}^{n}
	\sum_{j=1}^{n} w_{m}
	(\sigma^{i}x,\sigma^{j}x)-\sum_{i=1}^{n_{-}}
	\sum_{j=1}^{n_{-}} w_{m}
	(\sigma^{i}x,\sigma^{j}x)\right)^{4}<\varepsilon T^{3+2\delta} .
\end{gather}
\end{lemma}

\end{proof}

\begin{proof}
[Proof of \eqref{eq:momentBoundA}] This will use Hypothesis (H3) and
also the fact that Gibbs states have exponentially decaying
correlations (equation \eqref{eq:expMixing}). Since $h_{k} (x,y)$
depends only on the coordinates $x_{i},y_{i}$ with $|i|\leq k$, and
since $|h_{k}|\leq 1$ (see the earlier remark on scaling), exponential
correlation decay implies that for all $k,r\geq 1$,
\begin{equation}\label{eq:corrDecay}
	|E (h_{k} (x,\sigma^{k+r}x)\, |\, \mathcal{B}_{(-\infty
	,k]\cup [2k+2r,\infty )})| \leq C\beta^{r}.
\end{equation}
For convenience, we shall assume  that the constants $0<\beta <1$ in
\eqref{eq:expMixing} and in Hypothesis (H3) are the same (if the two
constants are different, replace the smaller with the larger).

When the square in \eqref{eq:momentBoundA} is expanded the resulting
terms have the form 
\[
	 Eh_{k} (\sigma^{i}x,\sigma^{j}x)h_{k'}(\sigma^{i'}x,\sigma^{j'}x),
\]
with $k,k'\geq m$ and $i,i',j,j'\leq n_{-}\leq T$.  Let $\Delta$ be
the largest integer such that one of the four indices $i,i',j,j'$ is
separated from \emph{all of} the other three by a gap of size $\Delta$, and let
$k_{*}=\max (k,k')$. Then by the exponential correlation decay
inequality \eqref{eq:corrDecay} and Hypothesis (H3) (using the fact
that $|h_{k}h_{k'}|\leq |h_{k_{*}}|$),
\[
	|Eh_{k}
	(\sigma^{i}x,\sigma^{j}x)h_{k'}(\sigma^{i'}x,\sigma^{j'}x)|\leq  
	C\min (\beta^{\Delta -2k_{*}},\beta^{k_{*}}).
\]
For any given value of $\Delta\geq 1$, the number of quadruples $i,i',j,j'\leq T$
with maximal gap size $\Delta $ is bounded above by $24 T^{2}
(2\Delta +1)^{2}$. (There are roughly $T^{2}$ choices for two of the
indices; once such a choice $(l,l')$ is made then one of the remaining
indices must be located within the interval of radius $\Delta$
centered at $l$, and the other within the corresponding interval
centered at $l'$. The factor 24=4! accounts for the possible
permutations of the indices.) Furthermore, for each $k_{*}\geq m$ the number of pairs
$k,k'\geq m$ such that $\max(k,k')=k_{*}$ is less than  $2k_{*}$. Consequently,
\[
	E_{\mu} \left( \sum_{i=1}^{n_{-}} \sum_{j=1}^{n_{-}} w_{m}
	(\sigma^{i}x,\sigma^{j}x)\right)^{2}
	\leq C' T^{2} \sum_{k_{*}=m}^{\infty}\sum_{\Delta =0}^{\infty
	} (2\Delta+1)^{2}k_{*} \min (\beta^{\Delta -2k_{*}},\beta^{k_{*}})
\]
where $C'=48$.  By choosing $m$ sufficiently large one can make this
bound smaller than $\varepsilon T^{2}$.
\end{proof}

\begin{proof}
[Proof of \eqref{eq:momentBoundB}] This is similar to the proof of
\eqref{eq:momentBoundA}, the difference being that here it is
necessary to count octuples instead of quadruples. The key once again
is the exponential correlation decay inequality \eqref{eq:corrDecay}:
this implies that for any 4 triples $i_{r},j_{r},k_{r}$,
\[
	|E_{\mu}\prod_{r=1}^{4} h_{k_{r}}
	(\sigma^{i_{r}}x,\sigma^{j_{r}}x)|
	\leq C \min (\beta^{\Delta -2k_{*}},\beta^{k_{*}})
\]
where $k_{*}=\max_{1\leq r\leq 4}k_{r}$ and $\Delta$ is the maximal
gap separating one of the indices $i_{r},j_{r}$ from the remaining 7.
For each $r\leq 4$ the indices $i_{r},j_{r}$ that occur in
\eqref{eq:momentBoundB} are constrained as follows (taking $i_{r}$ to
be the smaller of the two): either
\[
	1\leq i_{r}\leq n_{-}\leq j_{r} \leq n
	\quad \text{or} \quad n_{-}\leq i_{r}\leq j_{r}\leq n.
\]
Consequently, for each $\Delta \geq 1$, the total number of octuples
$(i_{r},j_{r})_{1\leq r\leq 4}$ with maximal gap $\Delta$ that occur
when the fourth power in \eqref{eq:momentBoundB} is expanded is
bounded above by $C'\Delta^{3}T^{3+2\delta}$ for some constant
$C'<\infty$ independent of $T$ and $\Delta$. For each $k_{*}\geq m$,
the number of quadruples $k_{1},k_{2},k_{3},k_{4}$ with maximum value
$k_{*}$ is bounded above by $4k_{*}^{3}$. Therefore, for each $n$ such
that $n_{-}\leq n\leq n_{+}$,
\begin{multline*}
	E_{\mu} \left( \sum_{i=1}^{n}
	\sum_{j=1}^{n} w_{m}
	(\sigma^{i}x,\sigma^{j}x)-\sum_{i=1}^{n_{-}}
	\sum_{j=1}^{n_{-}} w_{m}
	(\sigma^{i}x,\sigma^{j}x)\right)^{4}	\\
	\leq 
	C'' T^{3+2\delta}\sum_{k_{*}=m}^{\infty} \sum_{\Delta
	=0}^{\infty} \Delta^{3}k_{*}^{3} \min (\beta^{\Delta
	-2k_{*}},\beta^{k_{*}}) 	
\end{multline*}
for a constant $C''<\infty$ independent of $T$ and $m$. By choosing
$m$ large one can make this bound smaller than $\varepsilon
T^{3+2\delta}$. 
\end{proof}
 
\subsection{Extensions}\label{ssec:extensions}

\begin{corollary}\label{corollary:acU}
Let $\mu=\mu_{\varphi}$ be a Gibbs state, and let $\lambda $ be a
Borel probability measure on $\Sigma$ that is absolutely continuous
with respect to $\mu$ and such that the likelihood ratio $d\lambda
/d\mu$ is continuous on $\Sigma$. Let $h:\Sigma \times
\Sigma \rightarrow \zz{R}$ be a function that satisfies
Hypothesis~\ref{hypothesis:short-memory} relative to $\mu$. Then all
of the conclusions of Theorem~\ref{theorem:u-statistic} remain valid
under the measure $\lambda$, and the joint limit distribution of
$\tilde{U}_{T}$, $\tilde{\tau}_{T}$, and $R_{T}$ is
the same under $\lambda$ as  under $\mu$.
\end{corollary}

\begin{proof}
This follows from the asymptotic independence assertions of
Theorem~\ref{theorem:u-statistic}.  Consider first the random variable
$\tilde{U}_{T}$: to show that it converges in distribution under
$\lambda$, we must prove that for any bounded, continuous test
function $\psi :\zz{R}\rightarrow \zz{R}$, the expectations 
$ E_{\lambda} \psi (\tilde{U}_{T})$ converge as $T \rightarrow
\infty$. But since $d\lambda /d\mu$ is a bounded, continuous function,
the convergence  \eqref{eq:caseA}--\eqref{eq:caseB}  and the
asymptotic independence of $x$ and $\tilde{U}_{T} (x)$ under $\mu$
imply that 
\begin{align*}
	\lim_{T \rightarrow \infty}E_{\lambda} \psi
	(\tilde{U}_{T})&=\lim_{T \rightarrow \infty}E_{\mu} \psi (\tilde{U}_{T})
	\frac{d\lambda}{d\mu}\\
	& = \lim_{T \gamma  \infty} E_{\mu}\psi
	(\tilde{U}_{T})\lim_{T \rightarrow \infty}E_{\mu}\frac{d\lambda}{d\mu}\\
	&=\lim_{T \gamma  \infty} E_{\mu}\psi
	(\tilde{U}_{T}).
\end{align*}
(Note: This holds in both the case where $h_{+}$ is cohomologous to a
scalar multiple of $F$ and the case where it isn't.) A similar
argument, using Proposition~\ref{proposition:renewal} and
Proposition~\ref{proposition:clt}, proves that the random variables
$R_{T}$ and $\tilde{\tau}_{T}$ converge in distribution
under $\lambda$ to the same limit distributions as under $\mu$, and
that the various random variables, vectors, and sequences are
asymptotically independent.
\end{proof}

This result will suffice to deduce limit results about continuous-time
$U-$statistics in suspension flows under suspensions of Gibbs states
(cf. section~\ref{ssec:suspensions}) from corresponding results about
discrete-time $U-$statistics in shifts of finite type. For dealing
with measures like the uniform distribution on the set of periodic
orbits of minimal period $\leq T$ the following variant of
Corollary~\ref{corollary:acU}  will be needed.

\begin{corollary}\label{corollary:ac2sided}
Let $\mu=\mu$ be a Gibbs state and  $h:\Sigma \times
\Sigma \rightarrow \zz{R}$ be a function that satisfies
Hypothesis~\ref{hypothesis:short-memory} relative to $\mu$. Let $\{
\lambda_{T}\}_{T\geq 1}$ be a family of probability measures on
$\Sigma$ such that as $T \rightarrow \infty$,
\begin{equation}\label{eq:ac2sided}
	\frac{d \lambda_{T}}{d\mu} (x)\sim g_{1} (x)g_{2}
	(\sigma^{\tau_{T} (x)}x)g_{3} (R_{T} (x)) 
\end{equation}
where $g_{1},g_{2}:\Sigma \rightarrow [0,\infty) $ and
$g_{3}:[0,\infty )$ are nonnegative, bounded, continuous functions not
depending on $T$, such that $g_{3}$ is strictly positive on an
interval and $g_{2},g_{3}$ both have positive expectation under
$E_{\mu}$. Then as $T \rightarrow \infty$ the joint distribution of
$\tilde{U}_{T}$, $\tilde{\tau}_{T}$, $R_{T}$, $x$, and $\sigma^{\tau
(x)}x$ under $\lambda_{T}$ converges. Moreover, the limiting joint
distribution of $\tilde{U}_{T}$ and $\tilde{\tau}_{T}$ is the same as
under $\mu$.
\end{corollary}

\begin{proof}
The proof is virtually the same as that of
Corollary~\ref{corollary:acU}: in particular, for 
any bounded, continuous test functions
$\psi_{1},\psi_{2},\psi_{3}:\zz{R}\rightarrow \zz{R}$ and
$\psi_{4},\psi_{5}:\Sigma \rightarrow \zz{R}$,
\begin{align}\label{eq:tiltA}
	&\lim_{T \rightarrow \infty} E_{\lambda_{T}}\psi_{1}
	(R_{T})\psi_{2} (\tilde{U}_{T})\psi_{3} (\tilde{\tau}_{T})
	\psi_{4} (\psi_{5}\circ \sigma^{\tau})\\
\notag 	=&\lim_{T \rightarrow \infty} E_{\mu }\psi_{1}
	(R_{T})\psi_{2} (\tilde{U}_{T})\psi_{3} (\tilde{\tau}_{T})
	\psi_{4} (\psi_{5}\circ \sigma^{\tau}) g_{1} ( g_{2}\circ
	\sigma^{\tau})g_{3} (R_{T} (x)) 
\end{align}
by Theorem~\ref{theorem:u-statistic}, since $g_{1}$, $g_{2}$, and
$g_{3}$ are continuous. Moreover, because the random variables
$x,\sigma^{\tau (x)} (x)$, and $R_{T}$ are asymptotically independent
of $(\tilde{U}_{T},\tilde{\tau}_{T})$ under $\mu$, they will also be
asymptotically independent under $\lambda_{T}$, and the limit
distribution of $(\tilde{U}_{T},\tilde{\tau}_{T})$ will be the same as
under state $\mu$. However, the limit distributions of $R_{T}$ and
$\sigma^{\tau (x)}x$ will in general be different, because unlike the
bulk variables $(\tilde{U}_{T},\tilde{\tau}_{T})$ t hese random
variables are highly dependent on the last few coordinates of $x$
before $\tau (x)$. In particular, the limit distribution of
$\sigma^{\tau (x)} (x)$ will be ``tilted'' by the likelihood ratio
$g_{2}$:
\[
	\lim_{T \rightarrow \infty} E_{\lambda} (\psi \circ
	\sigma^{\tau})= \lim_{T \rightarrow \infty }E_{\mu} (\psi \circ
	\sigma^{\tau}) ( g_{2}\circ \sigma^{\tau})
	=E_{\mu}\psi g_{2}.
\]
\end{proof}

\begin{remark}\label{remark:strengthening}
The corollary remains true if it is only assumed that $g_{3}$ is
\emph{piecewise} continuous, in particular, if $g_{3}$ is the
indicator function of a bounded interval $[a,b]$. This can be proved
by a routine sandwiching argument, using the fact that $\mu \{x\,:\,
R_{T} (x)\in (b-\varepsilon ,b+\varepsilon)\}=O (\varepsilon)$, by the
renewal theorem (Proposition\ref{proposition:renewal}). Observe that
for this the standing assumption that $F$ is nonarithmetic is essential.
\end{remark}

\section{Cohomology and Hoeffding Projections in Constant and Variable
Negative Curvature}\label{sec:constant-variable-curvature}

Theorem~\ref{theorem:u-statistic} shows that in general the large-time
behavior of the distribution of a $U-$statistic is dependent on the
Hoeffding projection $h_{+} (x)$ of the kernel $h (x,y)$. In
particular, if $h_{+}$ is cohomologous to a scalar multiple of the
height function $F$ then the fluctuations of the $U-$statistic are of
order $T$, but otherwise are of order $T^{3/2}$. Thus, in any
application of the theorem it will be necessary to determine whether or
not the Hoeffding projection is cohomologous to a scalar multiple of
$F$. In this section we will show that for the function $h (x,y)$ in
the representation \eqref{eq:u-representation} of the
self-intersection count for geodesics on a surface of negative
curvature, the  factor that determines this is whether or not the
curvature is constant.

\subsection{Crossing intensities}\label{ssec:intensities} 
First we make a simple observation about the asymptotic frequencies of
intersections of a random geodesic with a fixed geodesic segment. 
(See \cite{bonahon}, \cite{bonahon:2} for far-reaching extensions and
consequences of this observation.) Fix a (compact) geodesic segment $\alpha$ on
$\Upsilon$ (for instance, a [prime] closed geodesic), and for any
geodesic ray $\gamma (t)=\gamma (t;x,\theta)$ let $N_{t} (\alpha
;\gamma)$ be the number of transversal intersections of $\alpha$ with
the segment $\gamma ([0,t])$.

\begin{proposition}\label{proposition:intensities}
Assume that $\Upsilon$ is a compact surface with a Riemannian metric
of (possibly variable) negative curvature, and let $\nu$ be any
ergodic, invariant probability measure for the geodesic flow on
$S\Upsilon$. Then for each geodesic segment $\alpha$ there is a
positive constant $\kappa (\alpha;\nu)$ such that for $\nu
-$a.e. initial vector $(x,\theta)$ the geodesic ray $\gamma$ with
initial tangent vector $(x,\theta)$ satisfies
\begin{equation}\label{eq:intensities}
	\lim_{t \rightarrow \theta} \frac{N_{t} (\alpha
	;\gamma)}{t}=\kappa (\alpha ;\nu) .
\end{equation} 
\end{proposition}

\begin{proof}
This  is a straightforward application of Birkhoff's ergodic
theorem. Fix $\varepsilon >0$ sufficiently small that any geodesic
segment of length $2\varepsilon$ can intersect $\alpha$ transversally
at most once, and denote by $G$ the set of all unit vectors
$(x,\theta)\in S\Upsilon$ such that the geodesic  segment $\gamma
([-\varepsilon ,\varepsilon],(x,\theta))$ crosses $\alpha$
(transversally). Define $g= (2\varepsilon)^{-1}I_{G}$ where $I_{G}$ is
the indicator function of $G$. Then 
\[
	\bigl| \int_{0}^{t} g (\gamma_{s}) \,ds -N_{t} (\alpha
	;\gamma)\bigr| \leq 2 ;
\]
the error $\pm 2$ enters only because the first and last crossing might
be incorrectly counted. Thus, the result follows from Birkhoff's
theorem.
\end{proof}

Lemma~\ref{lemma:H1} (see also \cite{bonahon:2}) implies that
if $\nu =\nu_{L}$ is normalized Liouville measure then for every
geodesic segment $\alpha$,
\begin{equation}\label{eq:charLiouville}
	\kappa (\alpha ;\nu_{L})= \kappa_{\Upsilon} |\alpha |,
\end{equation}
where $|\alpha |$ is the length of $\alpha$ and $\kappa_{\Upsilon}$ is
as in relation~\eqref{eq:lln}. On the other hand, Theorem~2 of
\cite{otal} implies that if two ergodic invariant probability measures
$\nu ,\nu'$ have the same intersection statistics, that is, if
\begin{equation}\label{eq:otal}
	\kappa (\alpha ;\nu)=\kappa (\alpha;\nu') \quad \text{for all
	closed geodesics} \;\; \alpha ,
\end{equation}
then $\nu =\nu'$. Now let $\nu_{\max}$ be the maximum entropy
invariant measure for the geodesic flow. If $\Upsilon$ has
\emph{constant} negative curvature then $\nu_{\max}=\nu_{L}$, but if
$\Upsilon$ has \emph{variable} negative curvature then $\nu_{\max}\not
=\nu_{L}$ (and in fact $\nu_{\max}$ and $\nu_{L}$ are mutually
singular). This proves the following corollary.

\begin{corollary}\label{corollary:constant-versus-variable}
If $\nu$ is an ergodic, invariant probability measure for the geodesic
flow such that the ratio $\kappa (\alpha ;\nu)/|\alpha |$ has the same
value for all {closed geodesics} $\alpha$, then $\nu
=\nu_{L}$. Consequently, in order that the ratio $\kappa (\alpha
;\nu_{\max})/|\alpha |$ has the same value for all {closed geodesics}
$\alpha$ it is necessary and sufficient that the surface $\Upsilon$
have {constant} negative curvature.
\end{corollary}

\subsection{Hoeffding projection of the self-intersection
kernel}\label{ssec:hoeffding-cohomology} According to the results of
section~\ref{sec:symbolicDynamics}, the geodesic flow on $S\Upsilon$
is semi-conjugate to the suspension flow $(\Sigma_{F},\phi_{t})$ over
a shift $(\Sigma ,\sigma)$ with H\"{o}lder continuous height function
$F$. By Proposition By Proposition~\ref{proposition:symbD-per}, the
self-intersection counts for closed geodesics and geodesic segments
are given by equations \eqref{eq:u-representation} and
\eqref{eq:self-intersection-Count-random}, with $h=h^{r}$ for any
$0\leq r <\min F$.

Fix a Gibbs state $\mu$ on $\Sigma$, and let $\lambda$ be the
projection to $\Sigma$ of the suspension measure $\mu^{*}$, that is,
the absolutely continuous probability measure defined by
\begin{equation}\label{eq:projectSuspension}
	\lambda (A)= E_{\mu_{L}} (I_{A}F) /E_{\mu_{L}}F.
\end{equation}

\begin{proposition}\label{proposition:cohomology}
The Hoeffding projection $h_{+}$ of the  function $h$ relative to the
measure $\lambda$ is cohomologous to a scalar multiple of the height
function $F$ if and only if the suspension measure $\mu^{*}$ is the
pullback of the Liouville measure on $S\Upsilon$.
\end{proposition}

\begin{proof}
For any $x\in \Sigma$, the value $h_{+} (x)$ is the probability that
the geodesic segments  $p\circ \pi (\mathcal{F}_{x})$
and $p\circ \pi (\mathcal{F}_{y})$ of the suspension flow will intersect when $y$
is randomly chosen according to the law $\lambda$. (By
Lemma~\ref{lemma:shortFibers}, we can assume that the symbolic
dynamics has been refined so that any two such segments can intersect
at most once.) In order that $h_{+}$ be cohomologous to $cF$, it is
necessary and sufficient (see \cite{bowen:book}, Theorem~1.28) that these
two functions sum to the same value over every periodic orbit of the
shift, that is, if for every periodic sequence $x\in \Sigma$ with
period (say) $n=n (x)$,
\begin{equation}\label{eq:cohomology-nsc}
	\sum_{j=0}^{n-1} h_{+} (\sigma^{j}x)=\sum_{j=0}^{n-1} cF (\sigma^{j}x).
\end{equation}
The left side is the expected number of intersections of the closed
geodesic $p\circ \pi (\mathcal{F}_{x})$ with a geodesic segment
$p\circ \pi (\mathcal{F}_{y})$ gotten by projecting a fiber
$\mathcal{F}_{y}$ of the suspension flow chosen randomly according to
the law $y\sim \lambda$. By the ergodic theorem, this expectation is
the long-run frequency (per time $E_{\mu}F$) of intersections of a
randomly chosen geodesic with the closed geodesic $\alpha =p\circ \pi
(\mathcal{F}_{x})$, and therefore coincides with (a scalar multiple
of) $\kappa (\alpha ;p_{*}\lambda)$. On the other hand, the right side
of \eqref{eq:cohomology-nsc} is just $c$ times the length of
$\alpha$. By Corollary~\ref{corollary:constant-versus-variable}, the
two sides of \eqref{eq:cohomology-nsc} coincide for all $\alpha$ if
and only if the projection $p_{*}\lambda$ is the Liouville measure.
\end{proof}

\section{Verification of Hypothesis
\ref{hypothesis:short-memory}}\label{sec:verification} 

To deduce Theorem~\ref{theorem:random} from the results of section
\ref{sec:u-statistics} it will be necessary to show that the relevant
function $h$ in the representation \eqref{eq:u-representation}
satisfies Hypothesis \ref{hypothesis:short-memory}. For the functions
$h$ and $h^{r}$ defined in section~\ref{sec:symbolicDynamics}, the
properties (H0)-- (H2) hold trivially, so only statement (H3) of
Hypothesis~\ref{hypothesis:short-memory} warrants consideration. The
purpose of this section is to prove that for any Gibbs state $\mu$
there exist values of $r$ such that the function $h^{r}$ defined in
Remark~\ref{remark:crossSection} meets the requirement (H3).

By Proposition \ref{proposition:boundaryCorrespondence} and
Corollary~\ref{corollary:series-variable}, for any compact, negatively
curved surface $\Upsilon$ the geodesic flow on $S\Upsilon$ is
semi-conjugate to a suspension flow $(\Sigma_{F},\phi_{t})$ over a
topologically mixing shift of finite type $(\Sigma ,\sigma)$ with a
H\"{o}lder continuous height function $F$.  This semi-conjugacy is
one-to-one except on a set of first category, and both the Liouville
measure and the maximum entropy measure for the geodesic flow pull
back to suspensions of Gibbs states on $\Sigma$. Furthermore, points
$x\in \Sigma$ of the underlying shift are mapped to pairs of points
$\xi_{+} (x^{+}),\xi_{-} (x^{-})$ on $\partial \zz{D}$ in such a way that the
suspension-flow orbit through $(x,0)$ is mapped to the geodesic whose
$L-$lift to the Poincar\'{e} plane has endpoints $\xi_{+} (x^{+}),\xi_{-}
(x^{-})$; and this mapping sends cylinder sets to arcs of $\delta
\zz{D}$ satisfying (D)--(E) of
Proposition~\ref{proposition:boundaryCorrespondence}.  By
Lemma~\ref{lemma:shortFibers}, for any small $\varepsilon >0$ the
symbolic dynamics admits a refinement for which the height function
$F$ satisfies $F<\varepsilon$. By
Proposition~\ref{proposition:symbD-per}, the self-intersection counts
for closed geodesics and geodesic segments are given by equations
\eqref{eq:u-representation} and
\eqref{eq:self-intersection-Count-random}, with $h=h^{r}$ for any
$0\leq r <\min F$. The function $h^{r}$ decomposes as in
equation~\eqref{eq:h-epsilon}.

\begin{proposition}\label{proposition:H4}
For any Gibbs state $\mu $, the functions $h^{r}_{m}$ satisfy (H3) of
Hypothesis~\ref{hypothesis:short-memory} relative to $\mu $ for almost
every $r$ in some interval $[0,r_{*}]$ of positive
length $r_{*}$.  
\end{proposition}

The remainder of this section is devoted to the proof of this
proposition. The key is the property (A3)'
(cf. Remark~\ref{remark:crossSection}), which asserts that there
exists $\varrho <1$ such that $h^{r}_{n}\not =0$ for some $n\geq m$
only if the geodesic segments corresponding to the suspension flow
segments $\mathcal{F}^{r}_{x}$ and $\mathcal{F}^{r}_{y}$ (cf. equation
\eqref{eq:f-epsilon}) intersect either at an angle less than
$C'\varrho^{m}$, or at a point within distance $C'\varrho^{m}$ of one
of the endpoints of the two geodesic segments.

For any integers $m,k\geq 1$ define 
\begin{align*}
	A^{r}_{m,k} &=\{x\in \Sigma \,:\, p\circ \pi
	(\mathcal{F}^{r}_{x}) \;\;\text{and}\;\; \, p\circ \pi
	(\mathcal{F}^{r}_{\sigma^{k}x}) \;\;\text{intersect at
	angle}\;<\varrho^{m} \} \quad \text{and}\\
	B^{r}_{m,k} &=\{x\in \Sigma \,:\, p\circ \pi
	(\mathcal{F}^{r}_{x}) \;\;\text{and}\;\; \, p\circ \pi
	(\mathcal{F}^{r}_{\sigma^{k}x}) \;\;\text{intersect at
	distance}\;<\varrho^{m} \;\text{of} \\
	& \;\quad  p (\pi
	(x,0)) \;\text{or}\; p ( \pi (x,F (x)-r))\} .
\end{align*}
To show that  (H3) of
Hypothesis~\ref{hypothesis:short-memory} holds relative to a Gibbs
state $\mu$ it is enough to show that there exist $C<\infty$ and $\beta <1$
such that for all sufficiently large $m$ and all $k\not =0$,
\begin{equation}\label{eq:H4Obj}
	\mu (A^{r}_{m,k}) +\mu (B^{r}_{m,k})\leq C\beta^{m}.
\end{equation}
We will show in Lemmas \ref{lemma:bmkr} and \ref{lemma:angle} that
each of the probabilities $\mu (A^{r}_{m,k})$ and $\mu (B^{r}_{m,k})$
is exponentially decaying in $m$, uniformly in $k$, for almost every
$r$ in a small interval $[0,r_{*}]$ of positive length.

\subsection{Intersections in small balls}\label{ssec:smallBalls}
We begin with $\mu (B^{r}_{m,k})$. The strategy for bounding this will
be to first handle the case $|k|\leq \exp \{\varepsilon  m\}$ for small $\varepsilon  >0$
by a density argument, and then the case $|k|>\exp \{\varepsilon  m\}$ by using
the exponential mixing property \eqref{eq:expMixingFunctionalForm} of
Gibbs states.

\begin{lemma}\label{lemma:densityArgument}
If  $0<\varrho <\alpha <1$, then for any Gibbs
state $\mu $ and almost every $r <\min F /3$, if $m$ is sufficiently
large then 
\begin{equation}\label{eq:densityArgument}
	\mu (B^{r}_{m,k})\leq \alpha ^{m} \;\; \text{for
	all}\; |k|\leq (\alpha /\varrho )^{m/2}, \;\;k\not =0.
\end{equation}
\end{lemma}

\begin{proof}
Without loss of generality we can assume
(cf. Lemma~\ref{lemma:shortFibers}) that no two geodesic segments of
length less than $2\max F$ intersect transversally more than once. For
$x\in \Sigma$ let $B_{m,k} (x)$ be the set of $r \in [0, F (x)+F
(\sigma x)]$ such that the geodesic segments $p\circ \pi
(\mathcal{F}_{x}\cup \mathcal{F}_{\sigma x}) $ and $p\circ \pi
(\mathcal{F}_{\sigma^{k}x}\cup \mathcal{F}_{\sigma^{k+1}x})$ intersect
at distance less than $\varrho^{m}$ of $p\circ \pi (x,r)$. Because
there is at most one intersection, the Lebesgue measure of $B_{m,k}
(x)$ is less than $2\varrho^{m}$.  Since $x\in B^{r}_{m,k}$ implies
that $r\in B_{m,k} (x)$, it follows by Fubini's theorem that for any
$\alpha \in (\varrho ,1)$,
\begin{align*}
	m_{Leb}&\{r\in [0,\min F/3]\,:\, \mu ( B^{r}_{m,k})\geq \alpha
	^{m} \}\leq 2 
	(\varrho /\alpha )^{m} \quad \Longrightarrow \\
	m_{Leb}&\{r\in [0,\min F/3]\,:\, \mu ( B^{r}_{m,k})\geq \alpha ^{m} \;\;
	\text{for some}\; |k|\leq (\alpha /\varrho)^{1/2}\}\leq 2
	(\varrho /\alpha )^{m/2}. 
\end{align*}
Since $\sum_{m} (\varrho /\alpha )^{m/2} <\infty$, it follows by the
Borel-Cantelli lemma that for almost every $r\in [0,\min F/3]$ 
the inequality \eqref{eq:densityArgument} holds for all sufficiently
large $m$.
\end{proof}

\begin{lemma}\label{lemma:smallBalls}
For any Gibbs state $\mu$ on $\Sigma$ and any $T<\infty$, there exist
$\delta =\delta (\mu,T) >0$ and $C=C_{T,\mu}>0$ with the following property: for 
any ball $B$ in $\Upsilon$ 
of sufficiently small diameter $\varepsilon >0$,
\begin{equation}\label{eq:smallBalls}
	\mu \{x \in \Sigma \,:\, p\circ \pi
	(\{\phi_{t} (x,0) \}_{0\leq t\leq T})
	\;\;\text{intersects}\;\;B\} 
	\leq C\varepsilon ^{\delta}.
\end{equation}
\end{lemma}

\begin{proof}
In the special case where the suspension $\mu_{*}$ of the Gibbs state
$\mu$ (see equation \eqref{eq:suspensionMeasures}) is the pullback of
the Liouville measure this is apparent from purely geometric
considerations, as we now show. We may assume that for $\varepsilon$ sufficiently
small the image $p\circ \pi (\mathcal{F}_{x})$ of a fiber can
intersect a ball of radius $2\varepsilon$ at most once.  Since the
surface area measure of a ball $B (x,\varepsilon)$ of radius
$\varepsilon$ is $\asymp \varepsilon^{2}$, the ergodic theorem implies
that the long-run fraction of time spent in $B (x,2\varepsilon)$ is
almost surely $K\varepsilon^{2}$, for a constant $K$ independent of
$\varepsilon$. On each visit to $B
(x,\varepsilon )$ a geodesic must spend time at least $K'\varepsilon$
in $B (x,2\varepsilon)$. Consequently, the long run fraction of the
sequence of fibers $p\circ \pi (\mathcal{F}_{\sigma^{n}x})$ on a
geodesic that visit $B (x,\varepsilon)$ is less than $K\varepsilon
/K'$; thus, by the ergodic theorem, \eqref{eq:smallBalls} holds with $\delta =1$.

Unfortunately, for arbitrary Gibbs states there is no simple relation
between the surface area and the Gibbs measure, so a different
argument is needed. Consider first the case of a Riemannian metric
with constant curvature $-1$.  Recall that in this case the surface
$\Upsilon$ can be identified with a compact polygon $\mathcal{P}$ in
the Poincar\'{e} disk $\zz{D}$ whose edges are pasted together in
pairs. With this identification, any ball $B$ in $\Upsilon$
corresponds to a ball of the same radius in the interior of
$\mathcal{P}$, provided this ball does not intersect $\partial
\mathcal{P}$, or otherwise a union of at most $4g$ sectors of balls of
the same radius, where $4g$ is the number of sides of
$\mathcal{P}$. Thus, a geodesic segment of length less than $\min F$
that intersects $B$ will lift to a geodesic segment in $\zz{D}$ that
intersects one of up to $4g$ balls of the same radius, all with
centers in the closure of $\mathcal{P}$. Consequently, a geodesic
segment of length $T$ in $\Upsilon$ that intersects a ball of radius
$\varepsilon$ in $\Upsilon$ lifts to a geodesic segment in $\zz{D}$
that intersects one of up to $4g T$ balls of the same radius, all with
centers at distance no more than $T$ from $\mathcal{P}$.

Fix a point $\zeta_{-} \in \partial \zz{D}$ on the circle at infinity,
and consider the set of all geodesics in
$\zz{D}$ with $\zeta_{-}$ as an endpoint (as $t \rightarrow -\infty$)
that intersect a ball $B$ of radius $\varepsilon$ with center in
$\mathcal{P}\cup \partial \mathcal{P}$. For any such geodesic, the
second endpoint $\zeta_{+}$ on $\partial \zz{D}$ is
constrained to lie in an arc $J (\zeta_{-},B)$ of length $\leq K
\varepsilon$, where $K$ is a constant that does not depend on
$\zeta_{-}$ or on the center of $B$. Recall
(Proposition~\ref{proposition:boundaryCorrespondence}) that
specification of the endpoint $\zeta_{-}$ of a geodesic is equivalent
(except on a set of first category) to specification of the backward
coordinates $x^{-}$ of the sequence $x\in \Sigma$ that represents the
geodesic; and similarly, specification of the endpoint $\zeta_{+}$ is
equivalent to specification of the forward coordinates $x^{+}$. By
(D)--(E) of Proposition~\ref{proposition:boundaryCorrespondence}, it
follows that constraining $\zeta_{+}$ to lie in an arc of length $\leq
K\varepsilon$ has the effect of constraining its forward itinerary
$x^{+}$ to lie in a union of one or two cylinder sets
$\Sigma^{+}_{[0,m]} (y)$ with $m=K' \log \varepsilon^{-1}$. Now for
any Gibbs state $\mu$ there exists $\beta <1$ such that the $\mu
-$measure of any cylinder set $\Sigma_{[1,m]} (x)$ is less than
$\beta^{m}$. Moreover, by inequality \eqref{eq:expMixing}, the
conditional measure $\mu (\cdot | x^{-})$ given the past is dominated
by a constant multiple of the unconditional measure $\mu $. Thus, if
$G^{\varepsilon} (B)$ denotes the set of all $x\in \Sigma$ such that
the suspension flow segment $\{\phi_{t} (x,0) \}_{0\leq t\leq T}$
lifts to a geodesic segment that intersects $B$, then
\[
	\mu (G^{\varepsilon} (B)) =E_{\mu} ( E_{\mu}
	(I_{G^{\varepsilon} (B)}\,|\,\mathcal{B}_{(-\infty,0]}))\leq
	\beta^{m}. 
\]
This implies \eqref{eq:smallBalls} in the constant curvature case.

This argument extends to metrics of variable negative curvature, with
the aid of the structural stability results of
section~\ref{ssec:series}. Let $\varrho_{1}$ be a metric of variable
negative curvature and $\varrho_{0}$ a metric of curvature -1.  Recall
that the $\varrho_{1}-$geodesic flow is orbit-equivalent, by a
H\"{o}lder continuous mapping $\Phi :S\Upsilon \rightarrow S\Upsilon$,
to the $\varrho_{0}-$geodesic flow, and that the homeomorphism $\Phi$
lifts to a homeomorphism $\tilde{\Phi}:S\zz{D} \rightarrow S\zz{D}$ of
the universal cover. Each $\varrho_{0}-$geodesic in $\zz{D}$
corresponds under $\tilde{\Phi}$ to a $\varrho_{1}-$ geodesic, and
these have the same endpoints on $\partial \zz{D}$ and the same
symbolic representation $x\in \Sigma$. Because
$\tilde{\Phi}$ is H\"{o}lder, constraining a $\varrho_{1}-$geodesic to
pass through a $\varrho_{1}-$ball of radius $\varepsilon$ forces the
corresponding $\varrho_{0}-$geodesic to pass through a $\varrho_{0}-$
ball of radius $\varepsilon^{\alpha}$, for some $\alpha >0$ depending
on the H\"{o}lder exponent and all $\varepsilon$ sufficiently small.
Therefore, the problem  reduces to the constant curvature case.
\end{proof}

\begin{lemma}\label{lemma:bmkr}
If $\varrho <1$ then for any Gibbs state $\mu$ and for almost every
$0\leq r<\min F/3$, there exists $\alpha <1$ such that if $m$ is
sufficiently large then
\begin{equation}\label{eq:mubmkr}
	\mu (B^{r}_{m,k})\leq \alpha ^{m} \;\; \text{for
	all}\; k\not =0.
\end{equation}
\end{lemma}

\begin{proof}
Lemma \ref{lemma:densityArgument} implies that if $\varrho \leq \alpha
<1$ then for almost every $r<\min F/3$ the inequality \eqref{eq:mubmkr} holds
for all $|k|< (\alpha /\varrho)^{m/2}$. We will show that for
\emph{every} $0\leq r<\min F /3$ the inequality \eqref{eq:mubmkr}  also holds
for $|k|\geq  (\alpha /\varrho)^{m/2}$; for this we shall appeal to the
exponential mixing inequality \eqref{eq:expMixing}, using
Lemma~\ref{lemma:smallBalls} to control the first moment. 
The proof will rely on the following elementary geometric fact:
\emph{For any compact Riemannian manifold $\mathcal{M}$ there exists $\kappa
<\infty$ (depending on the metric) such that for every
sufficiently small $\varepsilon >0$ there is a finite set of points
$z_{1},z_{2}, \dotsc ,z_{n}$ such that every $x\in \mathcal{M}$ is
within distance $\varepsilon$ of some $z_{i}$, but is within distance
$6\varepsilon$ of at most $\kappa$ distinct points $z_{i}$.} Call
such a collection of points $z_{i}$ an \emph{efficient $\varepsilon -$net.}

In order that $x\in B^{r}_{m,k}$ it is necessary that the geodesic
segment $p\circ \pi (\mathcal{F}^{r}_{\sigma^{k}x})$ intersects either
the ball of radius $\varrho^{m}$ centered at $p\circ \pi (\phi_{r}
(x))$, or the ball of radius $\varrho^{m}$ centered at $p (\phi_{F
(x)- r} (x))$, or both. Let $z_{1},..,z_{n}$ be an efficient
$\varrho^{m}-$net and let $B(z_{i},3\varrho^{m})$ be the ball of
radius $3\varrho^{m}$ centered at $z_{i}$.  Then
\[
	B^{r}_{m,k} \subset \bigcup_{i=1}^{n} H^{r}_{i,m}\cap G^{r}_{i,m,k}
\]
where $H^{r}_{i,m}$ is the set of all $x\in \Sigma$ such that $p\circ
\pi (x,r)\in B (z_{i},3\varrho^{m})$ and $G_{i,m,k}$ is the set of all
$x\in \Sigma $ such that the geodesic segment $p\circ \pi
(\mathcal{F}_{\sigma^{k}x}\cup \mathcal{F}_{\sigma^{k+1}x})$
intersects $B(z_{i},3\varrho^{m})$. (This is because
$\mathcal{F}^{r}_{y}\subset \mathcal{F}_{y}\cup \mathcal{F}_{\sigma
y}$.) Since $z_{1},..,z_{n}$ is an efficient $\varrho^{m}-$net, at
most $\kappa$ of the events $H^{r}_{i,m}$ can occur together;
consequently,
\begin{align*}
	\mu (B^{r}_{m,k})&\leq \sum_{i=1}^{n} \mu (H^{r}_{i,m}\cap
	G_{i,m,k}) \\
	&= \sum_{i=1}^{n} \mu (H^{r}_{i,m}) \mu
	(G_{i,m,k}\,|\, H^{r}_{i,m}) \\
	&\leq \kappa \max_{i\leq n} \mu
	(G_{i,m,k}\,|\, H^{r}_{i,m}).
\end{align*}
Thus, it remains to bound the conditional probabilities $\mu
(G_{i,m,k}\,|\, H^{r}_{i,m})$ for  $|k|> (\alpha /\varrho)^{m/2}$.

For each $i$ let $0\leq \psi _{i}\leq 1$ be a smooth function on
$\Upsilon$ with Lipschitz norm less than $6\varrho ^{-m}$
that takes the value $1$ on $B (z_{i},3\varrho^{m})$ and $0$ on the
complement of $B (z_{i},6\varrho^{m})$. 
For each $x\in \Sigma$ and $0\leq r\leq \min F/3$ define 
\begin{align*}
	g_{i,m} (x)&=\max_{0\leq s\leq F (x)+F (\sigma x)}
		\psi_{i} (p\circ \pi (x,s))  \quad \text{and}\\ 
	h^{r}_{i,m} (x)&=\psi_{i} (p\circ \pi (x,r)).
\end{align*}
Since the projection $p\circ \pi$ is $\delta  -$H\"{o}lder continuous
for some exponent $\delta  $, both $g_{i,m}$ and $h^{r}_{i,m}$ have
$\delta  -$H\"{o}lder norms bounded by $6\xnorm{p\circ \pi}_{\delta 
}\varrho  ^{-m}$. Therefore, the exponential mixing inequality
\eqref{eq:expMixingFunctionalForm} implies that for some $C<\infty$
and $0<\beta <1$
independent of $i,m,k$ and $r$,
\begin{align*}
	\mu (G_{i,m,k}\cap  H^{r}_{i,m})&\leq 
	E_{\mu} (g_{i,m}\circ \sigma^{k})h^{r}_{i,m} \\
	&\leq E_{\mu}g_{i,m} E_{\mu} h^{r}_{i,m} +C\beta^{k}\varrho ^{-m}
\end{align*}
For $|k|> (\alpha /\varrho)^{m/2}$ the second term is
super-exponentially decaying in $m$. But
Lemma~\ref{lemma:smallBalls}  implies that the
expectation $E_{\mu}g_{i,m}$ is bounded by $(6\varrho^{m})^{q}$ for
some $q>0$, and so the result now follows.

\end{proof}

\subsection{Intersections at small angles}\label{ssec:smallAngles}
Next we must show that the events $A^{r}_{m,k}$ have uniformly
exponentially decaying probabilities, in the sense
\eqref{eq:H4Obj}. The strategy here will be to show that if two
geodesic segments corresponding to distinct fibers of the suspension
flow cross at a small angle, then it will be impossible for their
successors to cross for a long time. This fact, coupled with the
ergodic theorem, will imply that the probability of a crossing at a
small angle must be small. The key geometric fact is as follows (see
also \cite{birman-series:1}).

\begin{lemma}\label{lemma:smallAngles}
For any $\kappa >0$ sufficiently small and any $\varrho <1$ there
exists $C<\infty$ such that for all large $m\geq 1$ the following
holds. If two geodesic segments $\gamma ([0,2\kappa],x)$ and $\gamma
([0,2\kappa],y)$ of length $2\kappa$ cross transversally at an angle
less than $\varrho^{m}$ then for every $1\leq j\leq Cm$ the geodesic
segments $\gamma ([j\kappa ,j\kappa +2\kappa],x)$ and $\gamma
([j\kappa ,j\kappa +2\kappa],y)$ do not cross.
\end{lemma}

\begin{proof}
Let $\kappa >0$ be sufficiently small that no two geodesic segments of
length $3\kappa$ on $\Upsilon$ can cross transversally more than once.
Consider lifts $\tilde{\gamma} (t,x)$ and $\tilde{\gamma} (t,y)$ of
the geodesic rays $\gamma (t,x)$ and $\gamma (t,y)$ to the universal
covering surface $\tilde{\Upsilon}$ whose initial segments
$\tilde{\gamma} ([0,2\kappa],\tilde{x})$ and $\tilde{\gamma}
([0,2\kappa],\tilde{y})$ cross transversally at angle
$<\varrho^{m}$. These geodesic rays cannot cross again, because for any
two points in a Cartan-Hadamard manifold there is only one connecting
geodesic. Consequently, if for some $j$ the geodesic segments $\gamma
([j\kappa ,j\kappa +2\kappa],x)$ and $\gamma ([j\kappa ,j\kappa
+2\kappa],y)$ were to cross, then their lifts $\tilde{\gamma}
([j\kappa ,j\kappa +2\kappa],x)$ and $\tilde{\gamma} ([j\kappa
,j\kappa +2\kappa],y)$ would contain points $\tilde{w},\tilde{z}$,
respectively, such that $\tilde{z}=g\tilde{w}$ for some element $g\not
=1$ of the group of deck transformations. However, if the initial
angle of intersection is less than $\varrho^{m}$ then the geodesic rays
$\tilde{\gamma} (t,x)$ and $\tilde{\gamma} (t,y)$ cannot diverge by
more than $\varepsilon$ for time $Cm$, where $C$ is a constant
determined by $\varepsilon$ and the curvature of $\Upsilon$ (which is
bounded, since $\Upsilon$ is compact).  If $\varepsilon >0$ and
$\kappa >0$ are sufficiently small then this would preclude the
existence of points $\tilde{w},\tilde{z}$ such that
$\tilde{w}=g\tilde{z}$ for some $g\not =1$.
\end{proof}

Because the semi-conjugacy $\pi :\Sigma_{F}\rightarrow S\Upsilon$ is
not one-to-one, two orbits of the geodesic flow can remain close for a
long time but have symbolic representations that are not close. The next
lemma shows that, at least for the symbolic dynamics constructed by
Series (cf. section~\ref{ssec:series}) and refinements such as that
described in the proof of Lemma~\ref{lemma:shortFibers}, this event has
small probability under any Gibbs state. 

Fix $\alpha >0$ and $\varepsilon >0$, and for each $m\geq 1$ let
$D_{m}=D^{\alpha,\varepsilon}_{m}$ be the set of all sequences $x\in
\Sigma$ such that there exists $(y,s)\in \Sigma_{F}$ satisfying
\begin{gather*}
	\text{distance} (\pi (\phi_{t} (x,0)),\pi (\phi_{t}
	(y,s)))\leq \varepsilon  \;\;\text{for all} \;\; |t|\leq
	e^{\alpha m} \quad \text{and}\\
	x_{i}\not =y_{i} \quad \text{for some}\;\;  |i|\leq m.
\end{gather*}

\begin{lemma}\label{lemma:nonUniqueReps}
Let $\mu$ be any Gibbs state. Then
for all sufficiently small $\varepsilon >0$ and all sufficiently large
$\alpha$ there exist $\beta <1$ and $C<\infty$ such that 
\begin{equation}\label{eq:nonUniqueReps}
	\mu (D_{m})\leq C\beta^{m} \quad \text{for all} \;\; m\geq 1.
\end{equation}
\end{lemma}

\begin{proof}
Recall (Proposition~\ref{proposition:conformalDeformation} and
following) that the geodesic flow with respect to a
Riemannian metric  of variable negative
curvature is orbit-equivalent to the geodesic flow on the same surface
but with a Riemannian metric of constant curvature $-1$, and that the
orbit equivalence is given by a H\"{o}lder-continuous mapping
$S\Upsilon  \rightarrow S\Upsilon$. Therefore, it suffices to prove
the lemma for the geodesic flow on a surface of
constant curvature $-1$. (The conformal deformation of metric might
change the values of $\beta$ and $\varepsilon$, but this is
irrelevant.) 

Suppose, then, that the Riemannian metric has curvature $-1$, and that
$\pi \circ \phi_{t} (x,0)$ and $\pi \circ \phi_{t} (y,s)$ are two
geodesics on the unit tangent bundle that stay within distance
$\varepsilon$ for all $|t|\leq e^{\alpha m}$, for some small
$\varepsilon$ and large $\alpha$. Because distinct orbits of the
geodesic flow separate exponentially fast (at exponential rate $1$,
since the curvature is $-1$), the initial vectors $\pi (x,0)$ and $\pi
(y,s)$ must be within distance $\kappa e^{- \alpha m}$, for some
constant $\kappa =\kappa (\varepsilon) >0$ independent of $\alpha$ and
$m$ (provided $m$ is sufficiently large).  

Recall that geodesics can be lifted to $S\zz{D}$ via the mapping $L$
described in section~\ref{ssec:series}. This mapping has
discontinuities only at vectors  tangent to one of the sides
of the fundamental polygon $\mathcal{P}$, but everywhere else is
smooth; consequently, either
\begin{enumerate}
\item [(a)] $L\circ \pi (x,0)$ and $L\circ \pi (y,s)$ are  within
distance $C\kappa e^{- \alpha m}$, or
\item [(b)] $L\circ \pi (x,0)$ is within distance
$C\kappa e^{- \alpha  m}$ of a vector tangent to one of the sides of
$\mathcal{P}$.
\end{enumerate}
In case (a), the lifted geodesics must have endpoints on $\partial
\zz{D}$ that are within distance $C'\kappa e^{- \alpha m}$; in case
(b) the lifted geodesics must have endpoints within distance $C'\kappa
e^{-\alpha m}$ of the endpoints on $\partial \zz{D}$ of one of the
geodesics that bound $\mathcal{P}$ (recall that the sides of
$\mathcal{P}$ are geodesic arcs). In either case, if $x$ and $y$
disagree in some coordinate $|i|\leq m$ then by
Proposition~\ref{proposition:boundaryCorrespondence} at least one of
the endpoints of the geodesic  $L\circ \pi \circ \phi_{t} (x,0)$ must
be within distance $C'\kappa e^{- \alpha m}$ of 
one of the endpoints of an arc $J_{k} (z^{+})$ of some generation
$k\leq m$. (Recall that the arcs $J_{k} (z^{+})$ correspond to cylinder
sets $\Sigma^{+}_{[0,m]} (z^{+})$.) 
There are at most $e^{2Am}$ such endpoints, where $A$ is
the number of sides of $\mathcal{P}$. 

If $\zeta$ is one of the endpoints of an arc $J_{k} (z^{+})$ of
generation $k$, then $\zeta$ has two symbolic expansions (i.e., there
are two sequences $z^{+},z_{*}^{+}$ that are mapped to $\zeta$ by
$\xi$). Since the arcs $J_{k} (\cdot)$ do not shrink faster than
exponentially (Proposition~\ref{proposition:boundaryCorrespondence}
part (D)), the forward endpoint $\xi (x^{+})$ of the geodesic  $L\circ
\pi \circ \phi_{t} (x,0)$ will lie within distance  $C'\kappa e^{-
\alpha m}$ of $\xi (z^{+})$ only if either 
\[
	x_{i}=z_{i} \;\; \forall \; 0\leq i\leq  C'' \alpha m \quad \text{or}
	\quad
	x_{i}= (z_{*})_{i} \;\; \forall \; 0\leq i\leq C''\alpha m,
\]
for a suitable constant $C''>0$ not depending on $m$ or
$\alpha$. Hence, since a Gibbs states $\mu$ will attach mass at most
$e^{-bm}$ to cylinder sets of generation $m$, for some $b=b (\mu) >0$,
it follows that 
\[
	\mu (D_{m})\leq C''' \exp \{-b \alpha m \} \exp \{2Am \}.
\]
By choosing $\alpha >0$ such that $b\alpha >2A$ we can arrange that
\eqref{eq:nonUniqueReps} holds.
\end{proof}

\begin{lemma}\label{lemma:angle}
For any Gibbs
state $\mu$
there exists $\beta =\beta (\varrho)<1$ such that for all
sufficiently large $m$ and all $k\not =0$,
\[
	\mu (A^{r}_{m,k})\leq \beta^{m}.
\] 
\end{lemma}

 \begin{proof} 
By Lemma~\ref{lemma:smallAngles}, it suffices to show
that there exist $\alpha >0$, $\varepsilon >0$ and $\beta <1$ such
that for all large $m$,
\[
	\mu (A^{r}_{m,k}\setminus D^{\alpha,\varepsilon}_{m})\leq \beta^{m}. 
\]
Suppose that $x\in A^{r}_{m,k}\setminus D^{\alpha,\varepsilon}_{m}$;
then the geodesic segments $p\circ \pi (\mathcal{F}^{r}_{x})$ and
$p\circ \pi (\mathcal{F})^{r}_{\sigma^{k}x}$ cross at angle less than
$\varrho^{m}$; in particular, there exist $r\leq s_{1}\leq F (x)+F
(\sigma x)$ and $r\leq s_{2}\leq F (\sigma^{k}x)+F (\sigma^{k+1}x)$
such that $p\circ p (x,s_{1})=p\circ \pi
(\sigma^{k},s_{2})$. Consequently, for some $\alpha >0$ depending on
the curvature of the underlying Riemannian metric,
\[
\text{distance} (\pi (\phi_{t} (x,s_{1})),\pi (\phi_{t}
	(\sigma^{k},s_{2})))\leq \varepsilon  \;\;\text{for all} \;\; |t|\leq
	e^{\alpha m} 	.
\]
Since $x\not \in  D^{\alpha,\varepsilon}_{m}$, it follows that 
$x_{i}=x_{i+{k}}$ for all $|i|\leq m$. The lemma now follows from
Lemma~\ref{lemma:noRepeats}.

\end{proof}

\section{Proof of Theorem \ref{theorem:random}}\label{sec:proof}

In this section we deduce Theorem~\ref{theorem:random} from the
results of section~\ref{sec:u-statistics}, using the symbolic dynamics
for the geodesic flow outlined in section~\ref{sec:symbolicDynamics}.
For this symbolic dynamics, the normalized Liouville measure $\nu_{L}$
pulls back to a measure $\mu^{*}$ on the suspension space $\Sigma$
that is the suspension (cf. equation~\eqref{eq:suspensionMeasures}) of
a Gibbs state $\mu=\mu_{L}$ on $\Sigma$. 
Proposition~\ref{proposition:H4} implies that for any Gibbs state
$\mu$ there exist values of $r$ such that the functions $h^{r}$ and
$h^{r}_{m}$ in equation~\eqref{eq:h-epsilon} satisfy
Hypothesis~\ref{hypothesis:short-memory} for $\lambda =\mu_{L}$, and
therefore also for any probability measure $\lambda$ on $\Sigma$ that
is absolutely continuous with respect to $\mu_{L}$. Recall from
Remark~\ref{remark:crossSection} that replacing the functions
$h,h_{m}$ by $h^{r},h^{r}_{m}$ is equivalent to moving the
Poincar\'{e} section of the suspension flow.  For notational ease, we
shall assume henceforth that the cross section has been adjusted in
such a way that (H3) holds for $r =0$, and drop the superscript from
the functions $h,h_{m}$.

Let $\lambda $ be the projection to
$\Sigma$ of the suspension measure $\mu^{*}$, that is, the absolutely
continuous probability measure defined by
\begin{equation}\label{eq:projectSuspension}
	\lambda (A)= E_{\mu_{L}} (I_{A}F) /E_{\mu_{L}}F.
\end{equation}
By Proposition~\ref{proposition:cohomology}, the  Hoeffding projection
$h_{+}$ of the function $h$ relative to the measure $\lambda$ is a scalar
multiple of $F$, in particular,
\begin{equation}\label{eq:globalHoeffding}
	h_{+}=\kappa F
\end{equation}
where $\kappa =1/ (4\pi |\Upsilon |)$. Consequently, case
\eqref{eq:caseA} of Theorem~\ref{theorem:u-statistic} applies.
Theorem~\ref{theorem:random} would  follow immediately from
Theorem~\ref{theorem:u-statistic} if not for
the presence of the ``boundary terms''sum $\sum_{1}^{\tau} (
g_{0}+g_{1})$ in \eqref{eq:self-intersection-Count-random}, since this
is of order $O(T)$. The following lemma will show that this sum,
normalized by $T$, converges in distribution as $T \rightarrow
\infty$, and that the limits depend only on the initial and final
points of the flow segment.

\begin{lemma}\label{lemma:ergodic}
For $\lambda -$almost
every $x\in \Sigma $ , every $0\leq s\leq F (x)$, and every $0\leq
r\leq F (\sigma^{\tau_{T}}x)$
\begin{align}\label{eq:initial-final}
	\lim_{T \rightarrow \infty}
	{T}^{-1}&\sum_{i=1}^{\tau_{T}}g_{0}(s,x,\sigma^{i}x)  =
	s \kappa \quad \text{and}\\
\label{eq:IF2}
 	\lim_{T	\rightarrow
	\infty}{T}^{-1}&\sum_{i=0}^{\tau_{T}}g_{1} (S_{\tau_{T} +1}F
	(x)-r,x,\sigma^{i}x) = r \kappa ,
\end{align}
where $\kappa =\kappa_{\Upsilon}=1/ (4\pi |\Upsilon |)$.
\end{lemma}

\begin{proof}
The relation \eqref{eq:initial-final} follows from the results of
section~\ref{sec:kernel} and the ergodic theorem.  The sum
$\sum_{i=1}^{\tau_{T}}g_{0}(s,x,\sigma^{i}x) $ counts the number of
intersections of the geodesic segment $p\circ \pi (\{\phi_{t} (x,s)
\})_{-s\leq t\leq 0}$ with the union of the segments $p\circ p
(\mathcal{F}_{\sigma^{j}x})$ for $0\leq j\leq \tau_{T} (x)$; this sum
can be re-expressed in terms of the intersection kernel $H_{\delta}$
(cf. section~\ref{sec:kernel}), yielding
\[
	\sum_{i=1}^{\tau_{T}}g_{0}(s,x,\sigma^{i}x)=\lim_{\delta
	\rightarrow 0}\sum_{i=1}^{[s/\delta
	]}\sum_{j=1}^{[T/\delta ]} H_{\delta} (\tilde{\gamma}
	(-s+i\delta),\tilde{\gamma} (j\delta)) +O (1) .
\]
(The error term accounts for the possibility of an intersection with
the geodesic segment corresponding to the final partial fiber, and
therefore is either 0 or 1.)
For each fixed point $\tilde{\gamma} (-s+i\delta)$, the ergodic
theorem and Lemma~\ref{lemma:H1} imply that 
\[
	\lim_{T \rightarrow \infty}T^{-1} \sum_{j=1}^{[T/\delta]}
	H_{\delta} (\tilde{\gamma} 
	(-s+i\delta),\tilde{\gamma} (j\delta)) =
	\delta^{2}\kappa 
\]
almost surely. Letting $\delta  \rightarrow 0$ one obtains the first
limit in \eqref{eq:initial-final}. The second limit is obtained in a
similar fashion.
\end{proof}

\begin{proof}
[Proof of Theorem~\ref{theorem:random}]

The measure $\lambda$ is the projection of the suspension measure
$\mu^{*}_{L}$, which in turn is the pullback to the suspension space
$\Sigma_{F}$ of the Liouville measure on $S\Upsilon$.  Thus, if
$(x,s)\in \Sigma_{F}$ is randomly chosen with distribution
$\mu^{*}_{L}$ then $x$ has distribution $\lambda$ and the normalized
vertical coordinate $s/F (x)$ is uniformly distributed on the unit
interval $[0,1]$, and is independent of $x$. Hence,
Lemma~\ref{lemma:ergodic} implies that if $(x,s)$ has distribution
$\mu^{*}_{L}$ then the first normalized boundary sum
\eqref{eq:initial-final} will converge to $YF (x)\kappa$, where $Y$ is
a uniform [0,1]  random variable independent of $x$. Since the
double sum $\sum_{0}^{\tau}\sum_{0}^{\tau}$ in the representation
\eqref{eq:u-representation} depends only on $x$, it follows from
Corollary~\ref{corollary:acU} that the random variables
\[
	\frac{\sum_{0}^{\tau}\sum_{0}^{\tau}h
	(\sigma^{i}x,\sigma^{j}x)-E_{\mu^{*}_{L}}h_{+}}{T}
	\quad \text{and} \quad 
	\frac{\sum_{i=1}^{\tau_{T}}g_{0}(s,x,\sigma^{i}x)}{T} 
\]
are asymptotically independent as $T \rightarrow \infty$. Similarly,
by Corollary~\ref{corollary:acU} and the renewal theorem,
if $(x,s)$ has distribution $\mu^{*}_{L}$ then the overshoot $R_{T}$
and the terminal state $\sigma^{\tau}x$ are asymptotically independent
of  $x,Y$ and of the double sum $\sum_{0}^{\tau}\sum_{0}^{\tau}$, and
so by Lemma~\ref{lemma:ergodic}, the second boundary sum
\[
	{T}^{-1}\sum_{i=0}^{\tau_{T}}g_{1} (S_{\tau_{T} +1}F
	(x)-t+s,x,\sigma^{i}x)
\]
is asymptotically independent of the other two sums in
\eqref{eq:self-intersection-Count-random}. Theorem~\ref{theorem:random}
now follows, as Lemma~\ref{lemma:ergodic} implies that the normalized
boundary-term sum converges almost surely and
Corollary~\ref{corollary:acU} implies that the normalized sum
$\sum_{1}^{\tau}\sum_{1}^{\tau}$ converges in distribution.
\end{proof}

\section{Proof of Theorem \ref{theorem:local}}\label{sec:local}

\subsection{Local self-intersection counts}\label{ssec:localSI}
Recall that for any smooth function $\varphi:\Upsilon  \rightarrow
\zz{R}_{+}$ the $\varphi -$\emph{localized} self-intersection counts of a
geodesic $\gamma$ are defined by 
\[
	N_{\varphi} (T)=N_{\varphi} (T;\gamma)=\sum_{i=1}^{N (T)}\varphi (x_{i})
\]
where $N (T)=N (T;\gamma)$ is the number of self-intersections of the
geodesic segment $\gamma [0,T]$ and $x_{i}$ are the locations of the
self-intersections on $\Upsilon$. Like the global self-intersection
counts, these can be expressed as sums of suitable functions defined
on a shift of finite type. Let $(\Sigma_{F},\phi_{t})$ and $(\Sigma
,\sigma)$ be the suspension flow and shift of finite type,
respectively, provided by
Proposition~\ref{proposition:symbD-per}. Define
a function $h_{\varphi}:\Sigma \times \Sigma \rightarrow \zz{R}_{+}$
by setting
\begin{equation}\label{eq:h-phi}
	h_{\varphi} (x,y)=\varphi  (z (x,y))h (x,y)
\end{equation}
where $h=1$ if the geodesic segments corresponding to the suspension
flow segments $\mathcal{F}_{x}$ and $\mathcal{F}_{y}$ intersect at a
point $z=z (x,y)\in \Upsilon$, and $h=0$ if these segments do not
intersect. By the same reasoning as in
equation~\eqref{eq:self-intersection-Count-random},
\begin{equation}\label{eq:localSIRepresentation}
	N_{\varphi} (T;\gamma)=\frac{1}{2}
	\sum_{i=1}^{\tau_{T}}\sum_{i=1}^{\tau_{T}} h_{\varphi}
	(\sigma^{i}x,\sigma^{j}x) +O (T).
\end{equation}
The error term accounts for intersections with the geodesic segments
corresponding to the first and last partial fibers (cf. equation
\eqref{eq:self-intersection-Count-random}), of which there are at most
$O (T)$.  Because the normalization in Theorem~\ref{theorem:local}
(cf. relation \eqref{eq:local}) entails division by $T^{3/2}$, the
error term in \eqref{eq:localSIRepresentation} can be ignored.

The proof of Theorem~\ref{theorem:local}, like that of
Theorem~\ref{theorem:random} in section~\ref{sec:proof}, will rely on
Corollary~\ref{corollary:acU}. Once again, let $\mu^{*}$ be the
pullback of the Liouville measure to $\Sigma_{F}$; recall that this is
the suspension of a Gibbs state $\mu$ for the shift. Let $\lambda$ be
the projection of $\mu^{*}$ to $\Sigma$, as defined by
\eqref{eq:projectSuspension}. This is absolutely continuous with
respect to $\mu$, so by Corollary~\ref{corollary:acU} the conclusions
of Theorem~\ref{theorem:u-statistic} remain valid for $\lambda$.  We
must show (1) that the function $h_{\varphi}$ satisfies Hypothesis
\ref{hypothesis:short-memory} with respect to $\lambda$, and (2) that
it is the \emph{second} case of Theorem~\ref{theorem:u-statistic} that
applies when the support of $f$ has small diameter, that is, that the
Hoeffding projection 
\begin{equation}\label{eq:hoeffdingF}
	h_{\varphi}^{+} (x):= \int_{\Sigma}h_{\varphi} (x,y)\,d\lambda (y)
\end{equation}
is not cohomologous to a scalar multiple of $F$. The first of these
tasks will be carried out in section~\ref{ssec:HH4} by an argument
similar to that carried out in section~\ref{sec:verification}
above for the function $h$. The second will be addressed in
sections~\ref{ssec:hoeffdingRep}--\ref{ssec:coboundaries}. 

\subsection{Representation of the Hoeffding
projection}\label{ssec:hoeffdingRep} For each small $\delta >0$ define
a function $H^{\varphi}_{\delta}:S\Upsilon\times S\Upsilon$ by setting
$H^{\varphi}_{\delta} (u,v)=\varphi (z (u,v))$ if the geodesic
segments of length $\delta$ based at $u$ and $v$ intersect at a point
$z (u,v)\in \Upsilon$, and setting $H^{\varphi}_{\delta} (u,v)=0$
otherwise. This is the obvious analogue of the intersection kernel
$H_{\delta}$ defined in section~\ref{sec:kernel}. The primary
difference between the self-intersection kernel $H_{\delta}$ and the
localized kernel $H^{\varphi}_{\delta}$ is that the constant function
$1$ is not, in general, an eigenfunction of $H^{\varphi}_{\delta}$. To
see this, define
\[
	k^{\varphi}_{\delta} (u)=\frac{1}{\delta^{2}\kappa}\int_{S\Upsilon}
	 H^{\varphi}_{\delta} (u,v) L (dv) 
\]
where $L$ is the normalized Liouville measure on $S\Upsilon$.

\begin{lemma}\label{lemma:localHoeffding} If $f:\Upsilon  \rightarrow
\zz{R}$ is continuous then
\[
	\lim_{ \delta  \rightarrow 0}
	\xnorm{k^{\varphi}_{\delta}-\varphi \circ p}_{\infty}=0. 
\]
\end{lemma}

\begin{proof}
Because $\varphi$ is continuous and $H^{\varphi}_{\delta} (u,v)$ is
nonzero only for pairs $u,v$ at distance $<\delta$, the value of
$\varphi (z (u,v))$ will be close to $\varphi (pu)$ when $\delta >0$
is small, uniformly for $u\in S\Upsilon$. Hence,
\[
	|\varphi  (pu) H_{\delta} (u,v) -\kappa
	\delta^{2}H^{\varphi}_{\delta} (u,v)| 
	\leq \max_{d (u,v)\leq \delta} |\varphi  (pu)-\varphi  (pv)|.
\]
Since the constant function $1$ is an eigenfunction of $H_{\delta}$,
with eigenvalue $\delta^{2}\kappa$ (by Lemma~\ref{lemma:H1}), the
result follows.
\end{proof}

The relevance of the kernel $H^{\varphi}_{\delta}$ is that the
Hoeffding projection $h_{\varphi}^{+}$ defined by
\eqref{eq:hoeffdingF} can be expressed approximately in terms of
$k^{\varphi}_{\delta}$. Both $h_{\varphi}^{+}$ and
$k^{\varphi}_{\delta}$ are defined as expectations of $\varphi
-$values at intersection points of geodesic segments: (i)
$h^{+}_{\varphi} (x)$ is the expected value of $f$ at the intersection
point (if there is one) of the geodesic segments corresponding to the
fibers $\mathcal{F}_{x}$ and $\mathcal{F}_{y}$ of the suspension flow
when $y$ is chosen according to the distribution $\lambda$; and (ii)
$k^{\varphi}_{\delta}$ is the corresponding expectation for the
geodesic segments of fixed length $\delta $. Hence, for small $\delta$
the value of $h^{\varphi}_{+} (x)$ can be obtained approximately by
integrating along the fiber $\mathcal{F}_{x}$. Together with
Lemma~\ref{lemma:localHoeffding}, this implies that if ${\gamma
} (t)$ is the orbit of the geodesic flow corresponding to the orbit
$\phi_{t} (x,0)$ of the suspension flow then
\begin{equation}\label{eq:approx}
	h^{\varphi}_{+} (x)=\lim_{\delta  \rightarrow 0}\int_{0}^{F
	(x)-\delta}k_{\delta} ({\gamma} (s)) \,ds =
	\int_{0}^{F (x)} \varphi  (p ({\gamma} (s)))
	\,ds .
\end{equation}

\subsection{Coboundaries of the geodesic
flow}\label{ssec:coboundaries} If $f,g :\Sigma \rightarrow \zz{R}$ are
H\"{o}lder continuous functions, then a {necessary} and sufficient
condition for $f$ and $g$ to be cohomologous is that they sum to the
same values on all periodic sequences (cf.\cite{bowen:book}, Theorem~1.28 for
the sufficiency). In particular, for every periodic sequence $x\in
\Sigma$, if $x$ has period $n=n (x)$ ten
\begin{equation}\label{eq:necessaryForCohomology}
	S_{n}f  (x)=g\psi  (x).
\end{equation}
In the case of interest, the relevant functions are integrals
over fibers of the suspension space $\Sigma_{F}$. For the function $F$
this is obvious: 
\[
	F (x)=\int_{0}^{F (x)}1\,ds,
\]
while for the Hoeffding projection $h^{\varphi}_{+}$ it follows from
formula \eqref{eq:approx}. Consequently, for $F$ and $ah_{+}^{\varphi}$ to
be cohomologous it is necessary that the function $a\varphi \circ p-1$
integrate to $0$ on every periodic orbit of the suspension flow. Since
both $\varphi \circ p$ and the constant $1$ are pullbacks of smooth
functions on $\Upsilon$, this implies that $a\varphi -1$ must
integrate to $0$ on every closed geodesic.

Call a continuous function $\psi :\Upsilon \rightarrow \zz{R}_{+}$ a
\emph{coboundary} for the geodesic flow if it integrates to zero along
every closed geodesic, and say that two functions are
\emph{cohomologous} if they differ by a coboundary.  It is quite easy
to construct a function $g:\Upsilon \rightarrow \zz{R}$ that is not
cohomologous to a constant. Take two closed geodesics $\alpha$ and
$\beta$ that do not intersect on $\Upsilon$, and let $g:\Upsilon
\rightarrow \zz{R}$ be any $C^{\infty}$, nonnegative function that is
identically $1$ along $\alpha$ but vanishes in a neighborhood of
$\beta$; then by the criterion established above, $g$ cannot be
cohomologous to a constant. In fact, the existence of non-intersecting
closed geodesics yields the existence of a large class of functions
that are not cohomologous to constants:

\begin{proposition}\label{proposition:plethora}
Let $\varepsilon >0$ be the distance in $\Upsilon$ between two
non-intersecting closed geodesics $\alpha$ and $\beta$. Then no
$C^{\infty}$, nonnegative function $g:\Upsilon \rightarrow \zz{R}$ that is
not identically zero and whose support has diameter less than
$\varepsilon$ is cohomologous to a constant.
\end{proposition}

\begin{proof}
In order that $g$ be cohomologous to a constant $c$ it must be the
case that the average value of $g$ along any closed geodesic is $c$.
If $g$ is continuous and nonnegative, and not identically zero, then
there is an open set $U$ in which $g$ is strictly positive. Because
closed geodesics are dense in $S\Upsilon$, their projections are dense
in $\Upsilon$, and so there is at least one closed geodesic $\xi$ that
enters $U$. Since $g$ is continuous, its integral -- and hence its
average -- along $\xi$ must be positive. However, by hypothesis, $g$
vanishes on at least one of the geodesics $\alpha,\beta$, and so there
is at least one closed geodesic on which the average value of $g$ is
$0$.
\end{proof}

\begin{remark}\label{remark:nonintersectingClosedGeodesics}
That there exist pairs of non-intersecting closed geodesics on any
negatively curved surface can be proved using the conformal
equivalence of Riemannian metrics discussed in
section~\ref{ssec:series}  above.  First, elementary arguments in
hyperbolic geometry show that there are non-intersecting closed
geodesics on any surface of constant curved -1. Next,
Proposition~\ref{proposition:conformalDeformation}  implies that for
any Riemannian metric $\varrho_{1}$ of variable negative curvature on
a compact surface $\Upsilon$ there is a smooth deformation of
$\varrho_{1}$ to a constant-curvature metric $\varrho_{0}$ through
metrics $\varrho_{s}$ of negative curvature. In this deformation, the
closed geodesic in a given free homotopy class deforms smoothly;
moreover, transversal intersections can be neither created nor
destroyed. Therefore, if $\gamma_{0},\gamma_{1}$ are non-intersecting
closed geodesics relative to $\varrho_{0}$, then the closed geodesics
$\gamma_{0}',\gamma_{1}'$ in the corresponding free homotopy classes
are also non-intersecting.
\end{remark}

\subsection{Verification of
Hypothesis~\ref{hypothesis:short-memory}}\label{ssec:HH4} 
It remains to show that the function $h_{\varphi}:\Sigma \times \Sigma
\rightarrow \zz{R}_{+}$ defined by \eqref{eq:h-phi} satisfies  Hypothesis
\ref{hypothesis:short-memory} relative to the measure $\lambda$, or
to some equivalent (mutually a.c.) probability measure. For the same
reason as in section~\ref{sec:verification} (see in particular
Lemma~\ref{lemma:densityArgument}) we must allow for  adjustment
of the Poincar\'{e} section of the suspension. Thus, for small $r\geq
0$ define
\[
	h^{r}_{\varphi} (x,y) =\varphi (z_{r} (x,y))h^{r} (x,y)
\]
where $h^{r}=1$ if the geodesic segments corresponding to the suspension
flow segments $\mathcal{F}^{r}_{x}$ and $\mathcal{F}^{r}_{y}$ intersect at a
point $z=z_{r} (x,y)\in \Upsilon$, and $h^{r}=0$ if these segments do not
intersect. Then the  representation~\eqref{eq:localSIRepresentation}
holds with $h_{\varphi}$ replaced by $h^{r}_{\varphi}$, and the
Hoeffding projection of $h^{r}_{\varphi}$ will again be given by
\eqref{eq:approx}, but with 
\[
	\int_{0}^{F (x)} \quad \text{replaced by} \;\; \int_{r}^{F (x)+r}.
\]

For the verification of Hypothesis~\ref{hypothesis:short-memory}, we
use the decomposition
\begin{gather}\label{eq:h-phi-decomp}
	h^{r}_{\varphi} (x,y)=\sum_{m=1}^{\infty} \psi^{r}_{m} (x,y)
	+\psi^{r}_{\infty} (x,y)\quad 
		\text{where}\\			
	\notag	\sum_{j=1}^{m}\psi^{r}_{j} (x,y)=\min \{
 	h^{r}_{\varphi} (x',y')\,:\, x'_{i}=x_{i} \;\;\text{and}\;\;
	y'_{i}=y_{i} \;\; \forall \; |i|\leq m\}.
\end{gather}
The functions $\psi^{r}_{m}$ are obviously symmetric, since
$h_{\varphi}$ is, and $\psi^{r}_{m}$ depends only on the coordinates
$|i|\leq m$. Moreover, each $\psi^{r}_{m}$ is nonnegative, and
$\sum_{m}\psi^{r}_{m}+\psi^{r}_{\infty} =h^{r}_{\varphi}$ is bounded
by $\xnorm{\varphi}_{\infty}$. Hence, (H0), (H1), and (H2) of
Hypothesis~\ref{hypothesis:short-memory} all hold, leaving only (H3).

\begin{lemma}\label{lemma:wiggles}
There exist constants $\varrho <\beta <1$ such that for all large $m$,
\begin{equation}\label{eq:psi-small}
	\psi^{r}_{m} (x,y)\leq \beta^{m}	
\end{equation}
unless the geodesic segments $p\circ \pi (\mathcal{F}^{r}_{x})$ and
$p\circ \pi (\mathcal{F}^{r}_{y})$ intersect at angle less than
$\varrho^{m}$ or at a point $z_{r}(x,y)$ within distance $\varrho^{m}$ of
one of the endpoints of one of the geodesic segments. 
\end{lemma}

\begin{proof}
For ease of exposition we shall discuss only the case $r=0$; the
general case can be handled in the same manner.
The semi-conjugacy $\pi :\Sigma_{F}  \rightarrow S\Upsilon$ is
H\"{o}lder continuous, so there exists $\alpha <1$ such that if two
sequences $x,x'\in \Sigma$ agree in coordinates $|i|\leq m$ then $\pi
(x,s)$ and $\pi (x',s)$ are within distance $\alpha^{m}$ for all $s\in
[0,F (x)\wedge F (x')]$, and $|F (x)-F (x')|<\alpha^{m}$, at least for
sufficiently large $m$. 

Suppose now that $x_{i}=x'_{i}$ and $y_{i}=y'_{i}$ for all $|i|\leq
m$, and that the geodesic segments $p\circ \pi (\mathcal{F}_{x})$ and
$\mathcal{F}_{y}$ intersect at an angle not smaller than $\varrho^{m}$
and at a point $z (x,y)$ not within distance $\varrho^{m}$ of one of
the endpoints.  Then by (A3)' of Remark~\ref{remark:crossSection}
(section~\ref{ssec:cc-gf}), the geodesic segments $p\circ \pi
(\mathcal{F}_{x'})$ and $\mathcal{F}_{y'}$ will also
intersect. Furthermore, if $\alpha <\varrho$ (as we may assume without
loss of generality) then the intersection point $z (x',y')$ will lie
within distance $\beta  ^{m}$ of $z (x,y)$, for some $\beta
<1$. Since $\varphi$ is smooth, it follows that for some $C<\infty$
depending on the $C_{1}-$norm of $\varphi$,
\[
	|\varphi (z (x,y))-\varphi (z (x',y'))|<C\beta^{m}.
\]
\end{proof}

Lemma~\ref{lemma:wiggles} implies that to prove condition (H3) of
Hypothesis~\ref{hypothesis:short-memory}, translation invariant
suffices to establish the inequality~\eqref{eq:H4Obj}. But this has
already been done, in
Lemmas~\ref{lemma:densityArgument}--\ref{lemma:angle}. This yields the
following result.

\begin{proposition}\label{proposition:Verification-2}
For any Gibbs state $\mu $, the functions $\psi ^{r}_{m}$ satisfy (H3) of
Hypothesis~\ref{hypothesis:short-memory} relative to $\mu $ for almost
every $r$ in some interval $[0,r_{*}]$ of positive
length $r_{*}$.  
\end{proposition}

\subsection{Proof of
Theorem~\ref{theorem:local}}\label{ssec:localProof}
Proposition~\ref{proposition:Verification-2} implies that after
appropriate modification of the Poincar\'{e} section of the suspension
flow, the function $h_{\varphi}$ in the representation
\eqref{eq:localSIRepresentation} of the localized self-intersection
count meets the requirements of Corollary~\ref{corollary:acU}.
Proposition~\ref{proposition:plethora} implies that if the support of
$\varphi:\Upsilon \rightarrow \zz{R}_{+}$ is less than the distance
between two non-intersecting closed geodesics, and if $\varphi \geq 0$
is smooth and not identically $0$, then the Hoeffding projection
$h^{\varphi}_{+}$ of $h_{\varphi}$ relative to the measure $\lambda$
is not cohomologous to a constant multiple of $F$. By the argument of
section~\ref{ssec:localSI} it follows that the \emph{second} case of
Theorem~\ref{theorem:u-statistic} (cf. relation~\eqref{eq:caseB})
applies.

\qed

\section{$U-$statistics and randomly chosen periodic
orbits}\label{sec:perOrbits}

The remainder of the paper will be devoted to the proof of
Theorem~\ref{theorem:closed}. The main technical tool,
Theorem~\ref{theorem:u-periodic}, will be an extension of
Theorem~\ref{theorem:u-statistic} to measures concentrated on finite
sets of periodic sequences.  As in Theorem~\ref{theorem:u-statistic}
there will be two cases, one leading to fluctuations of size $T$, the
other to fluctuations of size $T^{3/2}$. Which of these two cases will
apply will once again be determined by whether or not the relevant
Hoeffding projection is cohomologous to a scalar multiple of the
height function $F$ in the suspension
flow. Proposition~\ref{proposition:cohomology}  of
section~\ref{sec:constant-variable-curvature} gives  necessary and
sufficient conditions for this: the Hoeffding projection is
cohomologous to a scalar multiple of $F$ if and only if the suspension
measure $\mu^{*}$ associated with the Gibbs state $\mu$ is the
pullback of the Liouville measure.

\subsection{An Extension of
Theorem~\ref{theorem:u-statistic}}\label{ssec:perOrbitSampling}  

It is well understood that the distribution of periodic orbits in a
hyperbolic dynamical system is, in a certain sense, governed by the
invariant measure of maximal entropy. There are two aspects of the
connection. First, according to Margulis' \emph{prime orbit theorem}
(and its generalization in \cite{parry-pollicott}) the number of
periodic orbits of minimal period less than $T$ grows like $e^{\theta
T}/ (\theta T)$, where $\theta $ is the entropy of the max-entropy
measure \cite{margulis}, \cite{parry-pollicott}. Second, the empirical
distribution of a random chosen periodic orbit (from among those with
minimal period less than $T$) is, with high probability, close to the
max-entropy measure in the weak topology on measures (see
\cite{lalley:acta}, Theorem~7). It is the latter connection that is primarily
responsible for Theorem~\ref{theorem:closed}. 

In this section we will formulate and prove an extension of
Theorem~\ref{theorem:u-statistic} for $U-$statistics of randomly
chosen periodic orbits of a suspension flow. This result will be
combined with the results of section~\ref{sec:symbolicDynamics} on
symbolic dynamics for geodesic flows to prove
Theorem~\ref{theorem:closed} in section~\ref{sec:closed}.

Fix a topologically mixing suspension flow  on a suspension
space $\Sigma_{F}$ over a shift $(\Sigma ,\sigma )$ of finite type
with a H\"{o}lder continuous height function $F>0$.  The invariant
probability measure of maximal entropy for the suspension flow is the
suspension $\mu^{*}_{-\theta F}$ (cf. sec.~\ref{ssec:suspensions}) of
the Gibbs state $\mu_{-\theta F}$, where $\theta$ is the unique
positive number such that $\text{Pressure} (-\theta F)=0$. The value
$\theta$ is the topological entropy of the suspension flow (see,
e.g., \cite{parry-pollicott}, \cite{lalley:acta}).

Say that a sequence $x\in \Sigma$ \emph{represents} an orbit $\gamma$
of the suspension flow if the point $(x,0)\in \Sigma_{F}$ lies on the
path $\gamma$. If $x\in \Sigma$ is a periodic sequence then all of its
cyclic shifts represent the same periodic orbit $p=p_{x}$ of the
suspension flow, and these are the only representatives of
$p$. Theorem 7 of \cite{lalley:acta} implies that for large $T$ nearly
100\% of the periodic orbits of the suspension flow with minimal
period approximately $T$ have the property that their representative
periodic sequences have minimal period $T/E_{\mu_{-\theta F}}F +o
(T)$. Hence, for large $T$ nearly all of the periodic orbits of the
flow have $T/E_{\mu_{-\theta F}}F +o (T)$ representative
sequences. This implies that the uniform distribution on periodic
orbits of the flow with period $\approx T$ is nearly identical to the
image (under the natural correspondence) of the uniform distribution
on periodic sequences $x$ such that $S_{\tau_{T} (x)}F (x)\approx
T$. (Recall that $\tau (x)=\tau_{T} (x)$ is the smallest integer $n$
such that $S_{n}F (x)\geq T$, and $R_{T} (x)=S_{\tau (x)}F (x)-T$ is the
overshoot.)  Thus, we now change our focus from periodic orbits of
the flow to periodic sequences.

A periodic sequence $x\in \Sigma$ represents a periodic
orbit of the suspension flow with minimal period between $T$ and
$T+\varepsilon$ if and only if (a)  the period of the sequence $x$ is
$\tau (x)$, and (b) $R_{T} (x)<\varepsilon$. 
Denote by $B_{T,\varepsilon}$ the set of all periodic sequences
satisfying these conditions. This set is finite: in fact, Margulis'
prime orbit theorem and the law of large numbers cited above (or,
alternatively, Theorem 1 of \cite{lalley:acta}) imply that
\begin{equation}\label{eq:margulisEstimate}
	|B_{T,\varepsilon}|\sim C e^{\theta T} (e^{\theta \varepsilon}-1)
	\quad \text{as}\;\; T \rightarrow \infty 
\end{equation}
for a constant $C>0$ independent of $T$ and $\varepsilon$. Define
\begin{equation}\label{eq:nu-t-e}
	\nu_{T,\varepsilon}=\,\text{ uniform probability distribution
	on} \;\;B_{T,\varepsilon}.
\end{equation}
Our objective in this section is to extend the
results of Theorem~\ref{theorem:u-statistic} to the family of measures
$\nu_{T,\varepsilon}$. These results concern the large-$T$ limiting
behavior of the distribution of the $U-$statistics $U_{T}$ defined by
\[
		U_{T} (x):=\sum_{i=1}^{\tau (x)}\sum_{j=1}^{\tau (x)} h
	(\sigma^{i}x,\sigma^{j}x).
\]
Since ultimately we will want to use these results to prove that the
distribution of self-intersection counts of closed geodesics
converges, it is important that they should hold for functions
$h:\Sigma \times \Sigma \rightarrow \zz{R}$ that are not necessarily
continuous. The following relatively weak hypothesis on the function
$h$ is tailored to the particular case of self-intersection
counts. For any periodic sequence $x$ with minimal period $\tau
(x)=\tau_{T} (x)$ and any integer $m\geq 1$ define
\begin{equation}\label{eq:Delta-T-U}
	\Delta^{m}_{T}U (x)= \max \left| \sum_{i=1}^{\tau (x)}
	\sum_{j=1}^{\tau (x)} (h (\sigma^{i}x,\sigma^{j}x)-h
	(\sigma^{i}x',\sigma^{j}x'))\right|  
\end{equation}
where the maximum is over all sequences $x'\in \Sigma$ (not necessarily
periodic) such that $x'_{i}=x_{i}$ for all $-m\leq i\leq \tau (x)+m$.

\begin{hypothesis}\label{hypothesis:periodic}
For each $\varepsilon >0$ there exist positive constants
$\varepsilon_{m}\rightarrow 0$ as $m \rightarrow \infty$ such that for
all  sufficiently large $m\geq 1$ and  $T$ (i.e., for $T\geq
t_{m,\varepsilon}$),
\begin{equation}\label{eq:hypo-per}
	\nu_{T,\varepsilon}\{\Delta^{m}_{T} U (x) \geq
	\varepsilon_{m}T\} <\varepsilon_{m} .
\end{equation}
\end{hypothesis}
 
If $h$ is  H\"{o}lder continuous on $\Sigma \times \Sigma$ then
Hypothesis~\ref{hypothesis:periodic} is trivially satisfied, because
in this case $\Delta^{m}_{T}U$ will be uniformly bounded by
$\beta^{m}T$, for some $0<\beta <1$. In section~\ref{sec:proof} we will
show that the hypothesis holds for the function $h$ in the
representation \eqref{eq:u-representation} of self-intersection
counts.

\begin{theorem}\label{theorem:u-periodic}
Assume that $h:\Sigma \times \Sigma \rightarrow \zz{R}$ satisfies
Hypothesis~\ref{hypothesis:periodic} and also
Hypothesis~\ref{hypothesis:short-memory} for the measure $\lambda
=\mu_{-\theta F}$. Let $h_{+}$ be the Hoeffding projection of $h$
relative to $\mu_{-\theta F}$ (cf. equation \eqref{eq:hoeffdingF}),
and let $\tilde{U}_{T}$ and $\tilde{\tau}_{T}$ be the renormalizations
of $U_{T}$ and $\tau_{T}$ defined in \eqref{eq:caseA},
\eqref{eq:caseB}, and \eqref{eq:tau-tilde}. (In particular, if $h_{+}$
is cohomologous to a scalar multiple $aF$ of the height function $F$
then $\tilde{U}_{T}$ is defined by \eqref{eq:caseA}, but otherwise it
is defined by \eqref{eq:caseB}.)  Then as $T \rightarrow \infty$ the
joint distribution of $\tilde{U}_{T}$, $\tilde{\tau}_{T}$, $R_{T}$,
$x$, and $\sigma^{\tau_{T} (x)} (x)$ under $\nu_{T,\varepsilon}$
converges, and the limiting joint distribution of $\tilde{U}_{T}$ and
$\tilde{\tau}_{T}$ is the same as under $\mu_{-\theta F}$ (that is, by
\eqref{eq:caseA} and \eqref{eq:caseB} of Theorem~\ref{theorem:u-statistic}).
\end{theorem}


The remainder of this section is devoted to the proof of this
theorem. The strategy will be to deduce
Theorem~\ref{theorem:u-periodic} from
Corollary~\ref{corollary:ac2sided}; however, since
$\nu_{T,\varepsilon}$ is supported by a finite set of periodic
sequences, it is mutually singular with respect to any Gibbs state,
and so Corollary~\ref{corollary:ac2sided} does not apply directly.
Consequently, we will be forced to show that $\nu_{T,\varepsilon}$ is
close in the weak topology to a probability measure absolutely
continuous relative the Gibbs state $\mu_{-\theta F}$, to which
Corollary~\ref{corollary:ac2sided} \emph{does} apply.
Hypothesis~\ref{hypothesis:periodic} will ensure that the distribution
of $\tilde{U}_{T}$ under this approximating measure is close to its
distribution under $\nu_{T,\varepsilon}$, and so
Theorem~\ref{theorem:u-periodic} will follow.

Observe that to prove Theorem~\ref{theorem:u-periodic} it suffices to
prove the result for small values of $\varepsilon >0$, because for
each $T>0$ and each integer $M\geq 1$ the measure
$\nu_{T,\varepsilon}$ is a convex combination of the measures
$\nu_{T+i\varepsilon /M,\varepsilon /M}$. 

\subsection{Skeleton of the Proof}\label{ssec:skeleton}
To show that $\nu_{T,\varepsilon}$ is weakly close  to a
probability measure absolutely continuous with respect to
$\mu_{-\theta F}$, we will partition the support $B_{T,\varepsilon}$
of $\nu_{T,\varepsilon}$ into subsets on which the ``likelihood
function'' of the measure $\mu_{-\theta F}$
(cf. equation~\eqref{eq:gibbs}) is nearly constant. See
Proposition~\ref{proposition:near-uniform} below for a precise
statement. This will imply
that the uniform distribution on each set $A$ of the partition is
weakly close to the normalized restriction of $\mu_{-\theta F}$ to the
cylinder set of all sequences in $\Sigma$ that agree with some element
$x\in A$ in coordinates $-m\leq j\leq \tau (x)+j$, for some large
$m$. 

To define the partition, for each $x\in \Sigma$ and $m\geq 1$, let
$B_{T,\varepsilon ,m,x} $  be the set of all $y\in
B_{T,\varepsilon}$ such that
\begin{equation}\label{eq:m-agreement}
	y_{j}= y_{\tau (y)+j}=x_{j} 
	\quad \text{for all} \;\; j\in [-m,m] .
\end{equation}
Note that since $y\in B_{T,\varepsilon}$ it is periodic with period
$\tau (y)=\tau_{T} (y)$, so the restriction $y_{\tau (y)+j}=x_{j}$ is
redundant. Define $\nu_{T,\varepsilon ,m,x}$ to be the uniform distribution on
$B_{T,\varepsilon ,m,x} $. Clearly, $B_{T,\varepsilon ,m,x}$ depends
on $x$ only by way of the coordinates $x_{[-m,m]}$, and the sets
$B_{T,\varepsilon ,m,x}$ are pairwise disjoint, so they partition
$B_{T,\varepsilon}$. Thus, for each $m\geq 1$,
\begin{equation}\label{eq:cvx}
	\nu_{T,\varepsilon}=\sum_{x_{[-m,m]}} \frac{|B_{T,\varepsilon
	,m,x}|}{|B_{T,\varepsilon}|} \nu_{T,\varepsilon ,m,x},
\end{equation}
where the sum is over all \emph{admissible} sequences $x_{[-m,m]}$,
that is, sequences obtained by restricting sequences $x\in \Sigma$.
Theorem 1 of \cite{lalley:acta} implies that for each $x\in \Sigma$
and $m\geq 1$, as $T \rightarrow \infty$,
\begin{equation}\label{eq:lalley-est}
	|B_{T,\varepsilon ,m,x} |\sim 	\mu_{-\theta F}
	(\Sigma_{[-m,m]} (x)) |B_{T,\varepsilon}|.
\end{equation}
Since $\mu (\Sigma_{[-m,m]} (x))>0 $ for any admissible $x_{[-m,m]}$, this
estimate implies that for each $x_{[-m,m]}$ the set $|B_{T,\varepsilon
,m,x} |$ grows exponentially with $T$. In particular, for all
sufficiently large $T$ the set $B_{T,\varepsilon ,m,x} $ is nonempty,
and so $\nu_{T,\varepsilon ,m,x}$ is well-defined.

\begin{proposition}\label{proposition:near-uniform}
There exist constants $C=C_{\varepsilon},t_{m}<\infty$ and $\beta \in (0,1)$  such
that for all sufficiently large $m$,  all
$T\geq t_{m}$, and any two periodic sequences $x,y\in
B_{T,\varepsilon}$ for which \eqref{eq:m-agreement} holds,
\begin{equation}\label{eq:near-uniform}
	1-C\beta^{m}\leq \frac{\mu_{-\theta F} (\Sigma_{[-m,\tau (x)+m]}
	(x))}{\mu_{-\theta 	F} (\Sigma_{[-m,\tau (y)+m]}
	(y)) } \frac{e^{-\theta S_{\tau (y)}F (y)}}{e^{-\theta S_{\tau
	(x)}F (x)}} \leq  1+C\beta^{m}. 
\end{equation}
\end{proposition}

Since $\text{Pr} (-\theta F)=0$, it follows from the definition
\eqref{eq:gibbs} of a Gibbs state that the ratio in
\eqref{eq:near-uniform} is bounded above and below. The proof that the
upper and lower bounds are within $C\beta^{m}$ of $1$  will rely on 
the periodicity of the sequences $x,y$ and the condition
\eqref{eq:m-agreement}, together with the spectral theory of the
Ruelle operator. The details of the argument are deferred to
section~\ref{ssec:near-uniform} below.

\begin{remark}\label{remark:tau}
It is not assumed in Proposition~\ref{proposition:near-uniform} that
$\tau (x)=\tau (y)$, so the number of coordinates specified in the two
cylinder sets appearing in \eqref{eq:near-uniform} need not be the
same. Also, by definition of $\tau =\tau_{T}$, the sums $S_{\tau (x)}F
(x)$ and $S_{\tau (y)}F (y)$ both lie in the interval
$[T,T+\varepsilon]$, so the ratio of exponentials in
\eqref{eq:near-uniform} is bounded above and below by $e^{\pm \theta
\varepsilon}$. Thus, for small $\varepsilon$ the measures of the
cylinder sets in \eqref{eq:near-uniform} are nearly equal. 
\end{remark}

For any integer $m\geq 1$ define $\var_{m}F$ to be the
maximum difference $|F (x)-F (y)|$ for sequences $x,y\in \Sigma$ such
that $x_{j}=y_{j}$ for all $|j|\leq m$. Because the function $F$ is
H\"{o}lder continuous, the sequence $\var_{m}F$ decays exponentially
in $m$. Consequently, if two sequences $x,y\in \Sigma$ satisfy
$x_{j}=y_{j}$ for all $-m\leq j\leq n+m$ then
\begin{equation}\label{eq:DeltaM}
	 |S_{n}F (x)-S_{n}F (y)|<
	 \delta_{m}:=3\sum_{k=m}^{\infty}\var_{k}F. 
\end{equation}
The sequence $\delta_{m}$ decays exponentially with $m$.
 
\emph{Assume henceforth that} $\varepsilon < \min F$. This guarantees
that if $T<S_{n}F (x)\leq T+\varepsilon$ then $\tau(x)=n$.  Assume
also that $m$ is large enough that $5\delta_{m}<\varepsilon$.  For
each $m\geq 1$ and $x\in \Sigma$, define $A_{T,\varepsilon ,m,x}$ to
be the set of all sequences $y\in \Sigma$ that satisfy
\eqref{eq:m-agreement} and are such that $0<R_{T} (y)\leq
\varepsilon$, and define $\lambda_{T,\varepsilon ,m,x}$ to be the
probability measure with support $A_{T,\varepsilon ,m,x}$ that is
absolutely continuous relative to $\mu_{-\theta F}$ with Radon-Nikodym
derivative 
\begin{equation}\label{eq:lambdaTemx}
	\frac{d\lambda_{T,\varepsilon ,m,x}}{d\mu_{-\theta F}} (z) =
	 Ce^{\theta R_{T} (z)} I_{A_{T,\varepsilon ,m} (x)} (z),
\end{equation}
where $C=C_{T,\varepsilon ,m,x}$ is the normalizing constant needed to
make $\lambda_{T,\varepsilon ,m,x}$ a probability measure. In
section~\ref{ssec:weakApprox}  below we will show that when $m$ and
$T$ are large the set $A_{T,\varepsilon ,m,x}$ nearly coincides with
\[
	A^{*}_{T,\varepsilon  ,m,x}: =\bigcup_{y\in B_{T,\varepsilon ,m,x}}
	\Sigma_{[-m,\tau (y)+m]} (y)
\]
and so Proposition~\ref{proposition:near-uniform} will imply that the
probability measure $\lambda_{T,\varepsilon ,m,x}$ distributes its
mass nearly uniformly over the cylinder sets in the union
$A^{*}_{T,\varepsilon ,m,x}$.  Using this, we will prove that for
large $m$ the measure $\lambda_{T,\varepsilon ,m,x}$ is close in the
weak (L\'{e}vy) topology to $\nu_{T,\varepsilon ,m,x}$. It will be
most convenient to formulate this statement using the following
\emph{coupling metric} for the weak topology on the space of Borel
probability measures. See \cite{strassen} or \cite{dudley} for a proof
that the coupling metric generates the weak topology.

\begin{definition}\label{definition:coupling}
Let $Q_{A},Q_{B}$ be Borel probability measures on a complete,
separable metric space $(\mathcal{X},d)$. The \emph{coupling distance}
$d_{C} (Q_{A},Q_{B})$ between $Q_{A}$ and $Q_{B}$ is the infimal
$\kappa   \geq 0$ for which there exists a Borel probability
measure $Q$ on $\mathcal{X}\times \mathcal{X}$ with marginals $Q_{A}$
and $Q_{B}$ such that
\begin{equation}\label{eq:couplingDistance}
	Q\{(x,y)\,:\, d (x,y)>\kappa \}<\kappa  .
\end{equation}
\end{definition}

Since this definition requires a metric on $\mathcal{X}$, we must
specify a metric for the case $\mathcal{X}=\Sigma$, so henceforth we
will let $d=d_{\Sigma}$ be the metric $d (x,y)=2^{-n (x,y)}$ where $n
(x,y)$ is the minimum nonnegative integer $n$ such that $x_{j}\not
=y_{j}$ for either $j=n$ or $j=-n$.

\begin{proposition}\label{proposition:weakApprox}
There exist constants $C=C_{\varepsilon},t_{m}<\infty$ and $0<\beta <1$ such that for
all sufficiently large $m$ and $T\geq t_{m}$, and all $x\in \Sigma$,
\begin{equation}\label{eq:weakCloseA}
	d_{C} (\nu_{T,\varepsilon ,m,x},\lambda_{T,\varepsilon
	,m,x})\leq Ce^{\theta e}\beta^{m}
\end{equation}
Consequently,  for
all $\varepsilon >0$, all sufficiently large $m$ and $T$,
\begin{equation}\label{eq:weakCloseB}
	d_{C} (\nu_{T,\varepsilon},\lambda_{T,\varepsilon,m})\leq 2C\beta^{m}
\end{equation}
where 
\begin{equation}\label{eq:lambda}
	\lambda_{T,\varepsilon,m}:=\sum_{x_{-m,m}} \mu_{-\theta F}
	(\Sigma_{[-m,m]} (x)) \nu_{T,\varepsilon ,m,x}.
\end{equation}
\end{proposition}

The proof is given in section~\ref{ssec:weakApprox} below.

\begin{lemma}\label{lemma:corLambda}
Fix $\varepsilon >0$ and $m\geq 1$ large enough that
$5\delta_{m}<\varepsilon$. Then for each $x\in \Sigma$ the conclusions
of Corollary~\ref{corollary:ac2sided} hold for the family of
probability measures $(\lambda_{T,\varepsilon ,m,x})_{T\geq
1}$. Consequently, they hold also for the family
$\{\lambda_{T,\varepsilon ,m} \}_{T\geq 1}$. 
\end{lemma}

\begin{proof}
We must show that the Radon-Nikodym derivatives in equation
\eqref{eq:lambdaTemx} have the form \eqref{eq:ac2sided}. By
definition, the set $A_{T,\varepsilon ,m,x} $ consists of all
sequences $z\in \Sigma$ such that $R_{T} (z)\in (0,\varepsilon ]$ and
such that \eqref{eq:m-agreement} holds (with $y=z$). Hence, the
likelihood ratio \eqref{eq:lambdaTemx} can be factored as
\[
	 e^{-\theta R_{T} (z)}I_{A_{T,\varepsilon ,m} (x)}
	 (z)=e^{-\theta R_{T} (z)}I_{(0,\varepsilon ]} 
	 (R_{T} (z)) I_{\Sigma_{-m,m} (x)} (z) I_{\Sigma_{-m,m} (x)}
	 (\sigma^{\tau (z)} z) .
\]
This is clearly of the form \eqref{eq:ac2sided}. Although the function
$g_{3}$ in this factorization is not continuous (because of the
indicator $I_{(0,\varepsilon ]}$), it is piecewise continuous and
bounded, so by Remark~\ref{remark:strengthening} the conclusions of
Corollary~\ref{corollary:ac2sided} are valid for the family
$(\lambda_{T,\varepsilon ,m,x})_{T\geq 1}$. Since each of the
probability measures $\lambda_{T,\varepsilon ,m}$ is a convex
combination of the measures $\lambda_{T,\varepsilon ,m,x}$ (cf. equation
\eqref{eq:lambda}), it follows that the conclusions of
Corollary~\ref{corollary:ac2sided} hold also for the family
$\{\lambda_{T,\varepsilon ,m} \}_{T\geq 1}$. 
\end{proof}

\begin{proof}
[Proof of Theorem~\ref{theorem:u-periodic}] Given
Proposition~\ref{proposition:weakApprox} and
Lemma~\ref{lemma:corLambda}, Theorem~\ref{theorem:u-periodic} follows
routinely. Lemma~\ref{lemma:corLambda} implies that for all
sufficiently large $m$ the conclusions of
Theorem~\ref{theorem:u-periodic} hold if the measures
$\nu_{T,\varepsilon}$ are replaced by $\lambda_{T,\varepsilon
,m}$. Thus, to complete the proof, it will suffice to show that for
any $\delta >0$ and any continuous, bounded function $\varphi
:\zz{R}^{3}\times \Sigma^{2} \rightarrow \zz{R}$, there exists $m$
sufficiently large such that
\begin{equation}\label{eq:objCoupling}
	\limsup_{T \rightarrow \infty} |E_{\lambda_{T,\varepsilon ,m}}\Phi 
	-E_{\nu_{T,\varepsilon}} \Phi |\leq \delta \quad 
	\text{where} \quad 
	\Phi (x)=\varphi (\tilde{U}_{T},\tilde{\tau}_{T} ,R_{T},x,\sigma^{\tau
	(x)}x).
\end{equation}

By Proposition~\ref{proposition:weakApprox}, for all sufficiently
large $m$ and $T$ there exist Borel probability measures $Q=Q_{T,\varepsilon
,m}$ on $\Sigma^{2}$ with marginals $\nu_{T,\varepsilon}$
and $\lambda_{T,\varepsilon ,m}$ such that \eqref{eq:couplingDistance}
holds with $\kappa  =2Ce^{\theta \varepsilon}\beta^{m}$. Now for
any probability measure $Q$ on $\Sigma \times \Sigma$ with marginals 
$\nu_{T,\varepsilon}$
and $\lambda_{T,\varepsilon ,m}$,
\[
	E_{\lambda_{T,\varepsilon ,m}}\Phi 
	-E_{\nu_{T,\varepsilon}} \Phi =\int_{\Sigma \times \Sigma}
	(\Phi (x)-\Phi (y)) \,dQ (x,y).
\]
By \eqref{eq:couplingDistance},
\[
	\int_{d(x,y)>\kappa_{m}}|\Phi (x)-\Phi (y)|\, dQ (x,y)\leq
	\kappa_{m} \xnorm{\Phi}_{\infty} ,
\]
where $\kappa_{m} =2Ce^{\theta \varepsilon}\beta^{m}$. By choosing $m$
sufficiently large, we may arrange that $\kappa_{m}
\xnorm{\Phi}_{\infty} <\delta /2$. Thus, to prove
\eqref{eq:objCoupling} we must bound the integral of $|\Phi (x)-\Phi
(y)|$ on the set of pairs $(x,y)$ such that $d (x,y)\leq \kappa_{m}$.
This is where Hypothesis~\ref{hypothesis:periodic} will be used.

None of the functions $U_{T} (x)$, $\tau_{T} (x)$, nor $R_{T} (x)$ is
continuous in $x$. However, if $m$ is so large that
$5\delta_{m}<\varepsilon$ then inequality \eqref{eq:DeltaM} implies
that if $d (x,y)<2^{-m}$ then $\tau (x)=\tau (y)$ and $|R_{T}
(x)-R_{T} (y)|<\delta_{m}$ \emph{unless}
\[
	R_{T} (z) \in [0,\delta_{m}]\cup [\varepsilon -\delta_{m}] 
	\quad \text{for either} \;\;z=x \;\;\text{or} \;\;z=y.
\]  
Since $Q$ has marginals $\nu_{T,\varepsilon}$ and
$\lambda_{T,\varepsilon ,m}$, the estimates \eqref{eq:margulisEstimate}
(for the first marginal $\nu_{T,\varepsilon}$) and
\eqref{eq:renewalEstimate}  (for the second marginal
$\lambda_{T,\varepsilon ,m}$, using the fact that this measure is
absolutely continuous relative to $\mu_{-\theta F}$) imply that for
each $\varepsilon >0$ there exist constants $C_{\varepsilon}<\infty$
and $t_{m}<\infty$ such that if $m$ is large enough that
$5\delta_{m}<\varepsilon$ and $T\geq t_{m}$, then
\begin{align*}
	Q\{(x,y) \,:\, R_{T} (x) \in [0,\delta_{m}]\cup [\varepsilon
	-\delta_{m}]\} &< C_{\varepsilon}\delta_{m} \quad \text{and}\\
	Q\{(x,y) \,:\, R_{T} (y) \in [0,\delta_{m}]\cup [\varepsilon
	-\delta_{m}]\} &< C_{\varepsilon}\delta_{m}.
\end{align*}
Now if
$\tau (x)=\tau (y)$ then  depending on whether or not $h_{+}$ is
cohomologous to a scalar multiple of the height function $F$
(cf. equations \eqref{eq:caseA}--\eqref{eq:caseB}), 
\[
	\tilde{U}_{T} (x)-\tilde{U}_{T} (y)= T^{-\alpha }\sum_{i=1}^{\tau (x)}
	\sum_{j=1}^{\tau (x)} (h (\sigma^{i}x,\sigma^{j}x)-h
	(\sigma^{i}y,\sigma^{j}y))
\]
for either $\alpha =1$ or $\alpha =3/2$. In either case,
Hypothesis~\ref{hypothesis:periodic} ensures that there exist positive
constants $\varepsilon_{m}\rightarrow 0$ such that  for all large $m$ and
$T$, 
\[
	Q\{(x,y)\,:\, \tau (x)=\tau (y) \quad \text{and} \quad  |\tilde{U}_{T}
	(x)-\tilde{U}_{T} (y)|>\varepsilon_{m}\} <\varepsilon_{m}.
\]
Since the sequences $\varepsilon_{m},\delta_{m},$ and $\kappa_{m}$ all
converge to $0$ as $m \rightarrow \infty$, it now follows from the
continuity of $\varphi$ that for sufficiently large $m$ and $T$,
\[
		\int_{d(x,y)\leq \kappa_{m}}|\Phi (x)-\Phi (y)|\, dQ
		(x,y)\leq \delta /2.
\]
\end{proof}

\subsection{Proof of Proposition
\ref{proposition:weakApprox}}\label{ssec:weakApprox} 

In this section we show how Proposition~\ref{proposition:weakApprox}
follows from Proposition~\ref{proposition:near-uniform}. The proof of
Proposition~\ref{proposition:near-uniform} is given in
section~\ref{ssec:near-uniform} below. 

Fix $\varepsilon \in (0,\min F)$ and let $m$ be sufficiently large
that $5\delta_{m}<\varepsilon$.
For each $x\in B_{T,\varepsilon}$ define
\begin{equation}\label{eq:ATmx}
	A^{*}_{T,\varepsilon  ,m,x} =\bigcup_{y\in B_{T,\varepsilon ,m,x}}
	\Sigma_{[-m,\tau (y)+m]} (y). 
\end{equation}
Since $B_{T,\varepsilon ,m,x} $ depends only on the finite subsequence
$x_{[-m,m]}$, the same is true of the set $A^{*}_{T,\varepsilon
,m,x}$. The estimate \eqref{eq:lalley-est} implies that for large $T$
there will be many representative sequences $x'\in B_{T,\varepsilon
,m,x}$ for which $2\delta_{m}<R_{T} (x') <\varepsilon -4\delta_{m}$;
assume henceforth that $x$ is such a sequence, that is, that $x\in
B_{T+2\delta_{m},\varepsilon -4\delta_{m}}$. Then by
\eqref{eq:DeltaM}, for any $y\in B_{T,\varepsilon,m,x}$ it must be the
case that $\tau (y)=\tau (x)$; hence, the cylinder sets in the union
\eqref{eq:ATmx} are pairwise disjoint. Therefore, there is a
well-defined mapping $z \mapsto \hat{z}$ from $A^{*}_{T,\varepsilon ,m,x}
$ to $B_{T,\varepsilon ,m,x} $ that sends each $z$ to the element
$y\in B_{T,\varepsilon ,m,x} $ that indexes the cylinder set of the
partition \eqref{eq:ATmx} which contains $z$. This mapping has the
property that no point $z$ is moved a distance more than $2^{-m}$ (in
the usual metric $d=d_{\Sigma}$ on $\Sigma$); in fact,
\begin{equation}\label{eq:movingCost}
	d_{\Sigma} (\sigma^{i}z,\sigma^{i}\hat{z})\leq 2^{-m} \quad
	\text{for all} \;\; -m\leq i\leq \tau (z)+m.
\end{equation}

For each  $x\in B_{T+\delta_{m},\varepsilon -\delta_{m}}$ define
$\lambda^{*}_{T,\varepsilon ,m,x}$ to be the probability measure on
$A^{*}_{T,\varepsilon ,m,x} $ that is absolutely continuous with respect
to $\mu_{-\theta F}$ with Radon-Nikodym derivative
\begin{equation}\label{eq:LTemx}
	\frac{d\lambda^{*}_{T,\varepsilon ,m,x}}{d\mu_{-\theta F}} (z) =
	 C^{*}_{T,\varepsilon ,m,x}e^{-\theta R_{T} (\hat{z})}
	 I_{A^{*}_{T,\varepsilon ,m,x} } (z) , 
\end{equation}
where $C^{*} (T,\varepsilon ,m,x)$ is the normalizing constant needed to
make $\lambda^{*}_{T,\varepsilon ,m,x}$ a probability measure.  The
Radon-Nikodym derivative is constant on each of the cylinder sets in
the partition \eqref{eq:ATmx}, and its values on these cylinder sets
are chosen so as to cancel the exponential factors in
\eqref{eq:near-uniform}.  Consequently, by
Proposition~\ref{proposition:near-uniform}, if $y,z\in
B_{T,\varepsilon,m} (x)$ then for constants $C_{\varepsilon}<\infty$
depending only on $\varepsilon$,
\begin{equation}\label{eq:mu-flattened}
	1-C_{\varepsilon }\beta^{m}\leq
	\frac{\lambda^{*}_{T,\varepsilon ,m,x} (\Sigma_{[-m,\tau
	(x)+m]} 
	(z))}{\lambda^{*}_{T,\varepsilon ,m,x} (\Sigma_{[-m,\tau (y)+m]}
	(y)) }\leq 1+C_{\varepsilon }\beta^{m}
\end{equation}

\begin{corollary}\label{corollary:near-uniform}
Assume that $x\in B_{T+2\delta_{m},\varepsilon -2\delta_{m}}$.
Let ${\lambda^{\dagger}}_{T,\varepsilon ,m,x}$ be the push-forward of the
probability measure $\lambda^{*}_{T,\varepsilon ,m,x}$ under the mapping
$z\mapsto \hat{z}$ (that is, the distribution of $\hat{z}$ when $z$
has distribution $\lambda^{*}_{T,\varepsilon ,m,x}$). Then for suitable
constants $C=C_{\varepsilon}<\infty$ and $\beta \in (0,1)$ not depending on
$T, m$, or $x$,
\begin{equation}\label{eq:RNest}
	\biggr|\frac{d{\lambda^{\dagger}}_{T,\varepsilon
	,m,x}}{d\nu_{T,\varepsilon ,m,x}} -1 \biggr|
	\leq C\beta^{m},
\end{equation}
and consequently,
\begin{equation}\label{eq:near-uniform-coupling}
	d_{C}
	({\lambda^{\dagger}}_{T,\varepsilon,m,x},\nu_{T,\varepsilon
	,m,x})\leq 1-(1-C\beta^{m})^{-1}.
\end{equation}
\end{corollary}

\begin{proof}
The first inequality  is a direct consequence of 
\eqref{eq:mu-flattened}. The second follows from the first by the
following elementary fact: if $\nu ,\mu$ are mutually absolutely
continuous probability measures whose Radon-Nikodym derivatives $d\nu
/d\mu$ and $d\mu /d\nu$ are both bounded below by $\varrho \in (0,1]$,
then their coupling distance is no greater than $1-\varrho$.
\end{proof}

Since the mapping $z\mapsto \hat{z}$ moves each $z$ by a distance at
most $2^{-m}$, coupling distance between the probability measures
${\lambda^{\dagger}}_{T,\varepsilon ,m,x}$ and
${\lambda^{**}}_{T,\varepsilon ,m,x}$ is at most
$2^{-m}$. By Corollary~\ref{corollary:near-uniform}, the
coupling distance between ${\lambda^{\dagger}}_{T,\varepsilon ,m,x}$
and $\nu_{T,\varepsilon ,m,x}$ is at most $C'\beta^{m}$ for a suitable
$C'=C'_{\varepsilon}<\infty$. Therefore, to prove
Proposition~\ref{proposition:weakApprox} it suffices to prove the
following lemma.

\begin{lemma}\label{lemma:lambdaStar}
For each $\varepsilon >0$ there exists a constant
$C_{\varepsilon}<\infty$ such that for all $x\in \Sigma$,  all $m$
large enough that $\delta_{m}<5\varepsilon$,  and all large
$T$,
\begin{equation}\label{eq:lambdaStarTV}
	d_{C} (\lambda^{*}_{T,\varepsilon ,m,x},\lambda_{T,\varepsilon
	,m,x})
	\leq C_{\varepsilon} \delta_{m}.
\end{equation}
\end{lemma}

\begin{proof}
Recall that without loss of generality we may assume (for $T$ large)
that the representative sequence $x\in \Sigma$ is an element of
$B_{T+2\delta_{m},\varepsilon -2\delta_{m}}$. It must then be the
case, by inequality \eqref{eq:DeltaM}, that  every $y\in
B_{T,\varepsilon ,m,x}$ must have period $\tau (y)=\tau (x)$, and that
$R_{T} (y)\in (\delta_{m},\varepsilon -\delta_{m})$. This in turn
implies that  every $z$ in the cylinder $\Sigma_{[-m,\tau
(y)+m]}(y)$  must also satisfy  $\tau (z)=\tau (x)$ and $R_{T} (z)\in
(0,\varepsilon)$. Consequently, the support sets of the measures 
$\lambda^{*}_{T,\varepsilon ,m,x}$ and $\lambda_{T,\varepsilon ,m,x}$
satisfy 
\begin{equation}\label{eq:containment}
	A^{*}_{T,\varepsilon ,m,x} \subset A_{T,\varepsilon ,m,x}.	
\end{equation}

Next, suppose that $z\in A_{T,\varepsilon ,m,x} $ is such that
$R_{T} (z)\in (\delta_{m},\varepsilon -\delta_{m})$.  Let $\tilde{z}$ be
the periodic sequence with period $\tau (z)$ that agrees with $z$ in
coordinates $j\in [-m,\tau (z)+m]$; then by the same argument as
above, using inequality \eqref{eq:DeltaM}, we have  $R_{T}
(\tilde{z})\in (0,\varepsilon)$, and so $\tilde{z}\in B_{T,\varepsilon
,m,x}$. By construction, $z$ is in the cylinder $\Sigma_{[-m,\tau
(z)+m]} (\tilde{z})$, so it follows that $z\in A^{*}_{T,\varepsilon
,m,x}$ and $\tilde{z}=\hat{z}$. This proves that 
\begin{equation}\label{eq:SetDiff}
	 A_{T,\varepsilon ,m} (x)\setminus 
	A^{*}_{T,\varepsilon ,m,x} 
	\subset \{z \,:\, R_{T} (z)\not \in (\delta_{m},\varepsilon
	-\delta_{m})\}.
\end{equation}
These arguments also show that for every $z\in A_{T,\varepsilon ,m,x}$  such that
$R_{T} (z)\in (\delta_{m},\varepsilon -\delta_{m})$, and for every
$z\in A^{*}_{T,\varepsilon ,m,x}$,
\begin{equation}\label{eq:closeLRs}
	\left| \frac{e^{-\theta R_{T} (z)}}{e^{-\theta R_{T}
	(\hat{z})}} -1 \right|< e^{\theta \delta_{m}}-1.
\end{equation}

Relations \eqref{eq:containment}, \eqref{eq:SetDiff}, and
\eqref{eq:closeLRs} imply that the ratio of the Radon-Nikodym
derivatives  \eqref{eq:LTemx} and
\eqref{eq:lambdaTemx}  differs from $1$ by less than $C\delta_{m}$
except on the set
\[
	A^{**}_{T,\varepsilon ,m,x}:=\{z\in A_{T,\varepsilon
	,m,x}\,:\, R_{T} (z)\not \in 
	(\delta_{m},\varepsilon -\delta_{m})\} .
\]
But the renewal theorem (cf. inequality \eqref{eq:renewalEstimate})
implies that for some constant $C'_{\varepsilon}$ independent of $m$,
for all large $T$,
\[
	\frac{\mu_{-\theta F} (A^{**}_{T,\varepsilon
	,m,x})}{\mu_{-\theta F} (A_{T,\varepsilon ,m,x})} \leq
	C'_{\varepsilon}\delta_{m} .
\]
It now follows by routine arguments that for a suitable constant
$C''_{\varepsilon}$, the total variation distance between the measures
$\lambda^{*}_{T,\varepsilon ,m,x}$ and $\lambda_{T,\varepsilon ,m,x}$
is bounded above by $C''_{\varepsilon }\delta_{m}$ for large $m$ and
large $T$. This implies \eqref{eq:lambdaStarTV}.

\end{proof}

\subsection{Proof of
Proposition~\ref{proposition:near-uniform}}\label{ssec:near-uniform} 

Recall that any H\"{o}lder continuous function on $\Sigma$ is
cohomologous to a H\"{o}lder continuous function that depends only on
the forward coordinates. Let $F$ be the height function of the
suspension, and let $F_{+}$ be a function of the forward coordinates
that is cohomologous to $F$. Then $-\theta F$ and $-\theta F_{+}$ have
the same topological pressure (which by choice of $\theta$ is $0$),
and the Gibbs states $\mu_{-\theta F}$ and $\mu_{-\theta F_{+}}$ are
identical. Also, for any periodic sequence $x\in \Sigma$ with period
(say) $n$ it must be the case that $S_{n}F(x)=S_{n}F_{+} (x)$. Since
the assertion \eqref{eq:near-uniform} involves only periodic sequences
and measures of events under $\mu_{-\theta F}$,  to prove
\eqref{eq:near-uniform} it will suffice to prove
\eqref{eq:near-uniform} with $F$ replaced by $F_{+}$. Thus, the
representation \eqref{eq:gibbsStateDecomp} for the Gibbs state
$\mu_{-\theta F_{+}}$ can be used; in particular, since $\text{Pr}
(-\theta F_{+})=0$, 
\[
	\frac{d\mu_{-\theta F_{+}}}{d_{\nu_{-\theta F_{+}}}}
	=h_{-\theta F_{+}} 
\]
where $\nu$ and $h$ are the right and left eigenvectors of the Ruelle
operator $\mathcal{L}=\mathcal{L}_{-\theta F_{+}}$. (For the remainder
of the proof we will drop the subscripts on $h,\nu ,\mu$, and
$\mathcal{L}$.)

The measure $\mu$ is shift-invariant, so the cylinder sets in
\eqref{eq:near-uniform} can be shifted by $\sigma^{-m}$, and hence can
be regarded as cylinder sets in the one-sided sequence space
$\Sigma^{+}$. Fix a periodic sequence $x$ of period $n>2m+1$; then
\begin{align*}
	\mu (\Sigma_{[0,n+2m]} (x))&=\int_{\Sigma^{+}_{[0,n+2m]} (x)} h\,d\nu \\
	    &=\int_{\Sigma_{[0,n+2m]} (x)} h\,d ((\mathcal{L}^{*})^{n+m}\nu)\\
	    &=\int_{\Sigma^{+}} (\mathcal{L}^{n+m} (hI_{\Sigma^{+}_{[0,n+2m]}
	    (x)}) )\,d\nu  .
\end{align*}
Here we have used the fact that $\nu$ is an eigenmeasure of the
adjoint $\mathcal{L}^{*}$ of the Ruelle operator with eigenvalue
$1$. Next we use the definition of the  Ruelle operator
(\cite{bowen:book}, ch.1 sec. B) to write, for
any $z\in \Sigma^{+}$,
\[
	\mathcal{L}^{n+m} (hI_{\Sigma^{+}_{[0,n+2m]}
	    (x)}) (z)= \sum_{y\in \Sigma^{+}:\sigma^{n+m}y=z}
	    e^{-\theta S_{n+m}F (y)} h (y)I_{\Sigma^{+}_{[0,n+2m]}
	    (x)} (y) .
\]
The indicator function in this expression guarantees that the only
$z\in \Sigma^{+}$ for which there is a nonzero term in the sum are
those sequences such that $z_{j}=x_{j}$ for $j\in [0,2m]$. (Keep in
mind that $x$ is periodic with period $n$.)  For each such $z$ there
is exactly one $y\in \Sigma^{+}$ for which the summand is
nonvanishing, to wit, the sequence $(x|z)_{2m+1}:= x_{0}x_{1}\dotsb
x_{2m}z$ obtained by prefixing to $z$  the first $2m+1$ letters of
$x$. Consequently,
\begin{align*}
	\mu (\Sigma_{[0,n+2m]} (x))&=\int_{\Sigma^{+}_{2m+1} (x)}
	\exp \{-\theta S_{n+m}F ((x|z)_{2m+1}) \} h ((x|z)_{2m+1})) \, d\nu (z)\\
	&= e^{-\theta S_{n+m}F (x)}\mu (\Sigma_{[0,2m]} (x)) (1\pm C\beta^{m})
\end{align*}
for suitable constants $C<\infty$ and $0<\beta <1$ independent of $x$. The final
approximate equality follows from the H\"{o}lder continuity of $F$ and
$h$, together with inequality \eqref{eq:DeltaM}. It now follows that
if $x,x'$ are any periodic sequences with periods $n=\tau
(\sigma^{m}x)>2m+1$ and $n'=\tau
(\sigma^{m}x')>2m+1$ such that $x_{[0,2m]}=x'_{[0,2m]}$ then
\begin{align*}
	\frac{\mu (\Sigma_{[0,n +2m]} (x))}{\mu
	(\Sigma_{[0,n' +2m]} (x'))} &=\frac{e^{-\theta S_{m+n}F
	(x)}}{e^{-\theta S_{m+n'}F (x')}} (1\pm C'\beta^{m})\\
	&=\frac{e^{-\theta S_{n}F
	(\sigma^{m}x)}}{e^{-\theta S_{n'}F (\sigma^{m}x')}} (1\pm C''\beta^{m})
\end{align*}
for suitable $C',C''<\infty$. This implies relation \eqref{eq:near-uniform}.

\qed

\section{Proof of Theorem~\ref{theorem:closed}}\label{sec:closed}

The results of section~\ref{sec:symbolicDynamics} imply that for any
compact surface $\Upsilon$ equipped with a smooth Riemannian metric of
negative curvature the geodesic flow on $S\Upsilon$ is semi-conjugate
(by a H\"{o}lder continuous mapping) to a suspension flow
$(\Sigma_{F},\phi_{t})$ over a shift of finite type.  All but finitely
many closed geodesics correspond uniquely to periodic orbits of this
suspension flow, and for each of these the self-intersection count is
given by equation~\eqref{eq:u-representation}, or by
equation~\eqref{eq:u-representation} with $h$ replaced by $h^{r}$, for
some small $r\geq 0$. By Proposition~\ref{proposition:H4}, there exist
values of $r$ such that the function $h^{r}$ satisfies the hypotheses
of Theorem~\ref{theorem:u-statistic} relative to any Gibbs state; for
simplicity we will assume that the Poincar\'{e} section of the
suspension has been adjusted so that $r=0$. If the Riemannian metric
on $\Upsilon$ has constant curvature then the normalized Liouville
measure for the geodesic flow coincides with the maximum entropy
invariant measure, and so in this case
Proposition~\ref{proposition:cohomology} implies that the Hoeffding
projection of $h$ relative to the Gibbs state $\lambda =\mu_{-\theta
F}$ is a scalar multiple of $F$.  On the other hand, if the Riemannian
metric has variable negative curvature then the maximum entropy
measure is singular relative to Liouville measure, and so in this
case, by Proposition~\ref{proposition:cohomology}, the Hoeffding
projection of $h$ relative to the Gibbs state $\lambda =\mu_{-\theta
F}$ is \emph{not} cohomologous to a scalar multiple of $F$.
Therefore, in either case, Theorem~\ref{theorem:closed} will follow
from Theorem~\ref{theorem:u-periodic}, provided that
Hypothesis~\ref{hypothesis:periodic} can be verified.  This we will
accomplish by reducing the problem to a problem about crossing rates.

The following lemma asserts that for compact surfaces of constant
negative curvature, the ergodic law \eqref{eq:intensities} for
intersections with a fixed geodesic segment extends from random
geodesics to closed geodesics.  Denote by $\lambda_{T,\varepsilon}$
the uniform distribution on the set of all (prime) closed geodesics
with length in $[T,T+\varepsilon]$, and let $\kappa_{\Upsilon}=1/4\pi
|\Upsilon |$. For any geodesic arc $\alpha$, let $|\alpha |$ be the
length of $\alpha$.

\begin{lemma}\label{lemma:closedGDcrossingInt}
Assume that $\Upsilon$ is a negatively curved compact surface.    For any geodesic segment
$\alpha$ and any closed geodesic $\beta$ let $N (\alpha ;\beta)$ be
the number of transversal intersections of $\beta$ with $\alpha$. Then
there is a constant $\kappa^{*}$ depending on the Riemannian metric
such that  for every geodesic segment $\alpha$, all sufficiently small
$\varepsilon >0$,  and all $\delta >0$,
\begin{equation}\label{eq:closedGDcrossingInt}
	\lim_{T \rightarrow \infty} \lambda_{T,\varepsilon}\{\beta
\,:\, |N (\alpha ;\beta )-\kappa^{*}|\beta ||\alpha ||>\delta
T\} =0.
\end{equation}
\end{lemma}

\begin{proof}
This follows from Theorem~7 of \cite{lalley:si1} by the same argument
used to prove the ergodic theorem for self-intersections (Theorem~1 of
\cite{lalley:si1}). In the case of constant curvature,
$\kappa^{*}=\kappa_{\Upsilon }$, since in constant curvature the
Liouville measure and the maximum entropy measure coincide.
\end{proof}

Hypothesis~\ref{hypothesis:periodic} concerns the quantity
$\Delta^{m}_{T}U (x)$ defined by equation \eqref{eq:Delta-T-U}.  For
any periodic sequence $x$ with minimum period $\tau (x)=\tau_{T} (x)$
and any integer $m$ this quantity is the maximum difference in
self-intersection count between (a) the closed geodesic $G_{x}$
corresponding to the periodic orbit of the suspension flow through
$(x,0)$ and (b) any geodesic segment $G_{y}=(\pi \circ \phi_{t}
(y,0))_{0\leq t\leq S_{\tau (x)}F (y)}$ where $y$ is some sequence
that agrees with $x$ in coordinates $-m\leq i\leq \tau (x)+m$.  If $m$
is large, any two such geodesic segments are close, because the
semi-conjugacy between the suspension flow and the geodesic flow is
H\"{o}lder continuous.   This can be quantified as follows.

\begin{lemma}\label{lemma:closeOrbits}
There exists $A>0$ such that for all $n\geq 1$, all sufficiently large
$m$, and all pairs $x,y\in \Sigma$ such that $x_{i}=y_{i}$ for
$-m\leq i\leq n+m$, 
\begin{equation}\label{eq:closeOrbits}
		d (\pi (\phi_{t} (x,0)),\pi (\phi_{t} (y,0))) \leq e^{-Am}
		\quad \text{for all}\;\; 0\leq t\leq \tau_{T} (x).
\end{equation}
\end{lemma}

\begin{proof}
By definition of the ``taxicab'' metric on $\Sigma_{F}$
(cf. \cite{bowen-walters}), the orbits $\phi_{t} (x,0)$ and $\phi_{t}
(y,0)$ must remain within distance $e^{-Bm}$, for a suitable constant $B>0$.
Because the semi-conjugacy $\pi$ is H\"{o}lder continuous, the
projections $\pi \circ \phi_{t} (x,0)$ and $\pi \circ \phi_{t}(y,0)$
must remain within distance $e^{-Am}$.
\end{proof}

\begin{remark}\label{remark:lengthDiff}
The lengths of the segments $G_{x}$ and $G_{y}$ will in general be
different, because $S_{\tau (x)}F (y)$ need not equal $S_{\tau (x)}F
(x)$. However, the difference in lengths can be at most $\delta_{m}$,
where $\delta_{m}$ is given by \eqref{eq:DeltaM}, which decays
exponentially in $m$.
\end{remark}

\begin{proof}
[Proof of Hypothesis \ref{hypothesis:periodic}] Fix geodesic segments
$G_{x},G_{y}$ as above, and consider the difference in their
self-intersection counts. To estimate this, consider how the
difference changes as $G_{x}$ is smoothly deformed to $G_{y}$ through
geodesic segments by smoothly moving the initial and final endpoints,
respectively, along smooth curves $C_{0}$ and $C_{1}$. In such a
homotopy, the self-intersection count will change only at intermediate
geodesic segments $G_{z}$ along the homotopy where one of the
endpoints passes through an interior point of the segment.  Now the
geodesic segment $G_{z}$ remains within distance $e^{-Am}$ of $G_{x}$,
by Lemma~\ref{lemma:closeOrbits} and Remark~\ref{remark:lengthDiff} so
for any interior point of $G_{z}$ that meets (say) the initial
endpoint of $G_{z}$, the corresponding point on $G_{x}$ must be within
distance $e^{-Am}$, and hence within distance $e^{-Am}$ of the curve
$C_{0}$. In particular, this corresponding point on $G_{x}$ must fall inside
a small rectangle $R_{0}$ surrounding $C_{0}$ whose sides are geodesic arcs.
Since the lengths of $C_{0}$ and $C_{1}$ are bounded above by
$e^{-Am}+\delta_{m}$ (by Lemma~\ref{lemma:closeOrbits} and
Remark~\ref{remark:lengthDiff}) the rectangle $R_{0}$ can be chosen so
that its sides all have lengths bounded by $e^{-A'm}$. 

This proves that the difference in the self-intersection counts of
$G_{x}$ and $G_{y}$ is bounded above by the number of crossings of
$\partial R_{0}$ and $\partial R_{1}$, where $R_{i}$ are rectangles
bounded by geodesic arcs of length $\leq e^{-A'm}$. Hence,
Lemma~\ref{lemma:closedGDcrossingInt}, for most periodic sequences $x$
of minimal period $\asymp T$ this difference is bounded above by
$(8+\varepsilon) e^{-A'm}T$ when $T$ is large. This implies
Hypothesis~\ref{hypothesis:periodic}. 
\end{proof}

\bibliographystyle{plain}
\bibliography{mainbib}

\end{document}